\documentclass[a4paper,11pt,reqno,noindent]{amsart}
\usepackage[centertags]{amsmath}
\usepackage{amsfonts,amssymb,amsthm,dsfont,cases,amscd,esint,enumerate}
\usepackage[T1]{fontenc}
\usepackage[english]{babel}
\usepackage[applemac]{inputenc}
\usepackage{newlfont}
\usepackage{color}
\usepackage[body={15cm,21.5cm},centering]{geometry} 
\usepackage{fancyhdr}
\usepackage{enumerate}

\theoremstyle{plain}
\newtheorem{theor10}{Theorem}
\newenvironment{theor1}
  {\pushQED{\qed}\begin{theor10}}
  {\popQED\end{theor10}}
\newtheorem{prop10}{Proposition}

\newtheorem{cor10}{Corollary}
\newenvironment{cor1}
  {\pushQED{\qed}\begin{cor10}}
  {\popQED\end{cor10}}
\newtheorem{theor0}{Theorem}[section]
\newenvironment{theor}
  {\pushQED{\qed}\begin{theor0}}
  {\popQED\end{theor0}}
\newtheorem{lem0}[theor0]{Lemma}
\newenvironment{lem}
  {\pushQED{\qed}\begin{lem0}}
  {\popQED\end{lem0}}
\newtheorem{prop0}[theor0]{Proposition}
\newenvironment{prop}
  {\pushQED{\qed}\begin{prop0}}
  {\popQED\end{prop0}}
\newtheorem{cor0}[theor0]{Corollary}
\newenvironment{cor}
  {\pushQED{\qed}\begin{cor0}}
  {\popQED\end{cor0}}
\newtheorem{propr0}[theor0]{Property}

\newtheorem{hyp0}[theor0]{Hypothesis}

\newtheorem{result0}[theor0]{Result}

\newtheorem{conj0}[theor0]{Conjecture}

\newtheorem{heur0}[theor0]{Heuristics}

\theoremstyle{definition}
\newtheorem{defin0}[theor0]{Definition}
\newenvironment{defin}
  {\pushQED{\qed}\begin{defin0}}
  {\popQED\end{defin0}}
\newtheorem{rems0}[theor0]{Remarks}

\newtheorem{ex0}[theor0]{Example}

\newtheorem{exs0}[theor0]{Examples}

\newtheorem{rem0}[theor0]{Remark}
\newenvironment{rem}
  {\pushQED{\qed}\begin{rem0}}
  {\popQED\end{rem0}}
\newtheorem{qu0}[theor0]{Question}

\newtheorem{qus0}[theor0]{Questions}

  \newtheorem{as0}[theor0]{Assumption}

\mathchardef\emptyset="001F
\numberwithin{equation}{section}

\newcommand{\N}{\mathbb N}

\newcommand{\e}{\varepsilon}
\newcommand{\Log}{|\!\log\e|}

\newcommand{\sym}{\operatorname{Sym}}

\newcommand{\calQ}{\mathcal{Q}}

\newcommand{\Lc}{\mathcal{L}}

\newcommand{\Dm}{\mathbb{D}}

\newcommand{\Sc}{\mathcal S}
\newcommand{\Rc}{\mathcal R}

\newcommand{\Hf}{\mathfrak H}

\newcommand{\R}{\mathbb R}
\newcommand{\Z}{\mathbb Z}
\newcommand{\C}{\mathbb C}

\newcommand{\F}{\mathcal F}
\newcommand{\Nc}{\mathcal N}

\newcommand{\op}{{\operatorname{op}}}
\newcommand{\E}{\mathbb{E}}

\newcommand{\ee}{e}
\newcommand{\Aa}{\boldsymbol a}

\newcommand{\dom}{{\operatorname{dom}}}
\newcommand{\Ld}{\operatorname{L}}

\newcommand{\step}[1]{\noindent \textit{Step} #1.}

\newcommand{\Pm}{\mathbb{P}}
\newcommand{\pr}[1]{\mathbb{P}\left[ #1 \right]}

\newcommand{\expec}[1]{\mathbb{E}\left[ #1 \right]}
\newcommand{\expecm}[1]{\mathbb{E}\big[ #1 \big]}
\newcommand{\var}[1]{\mathrm{Var}\left[#1\right]}
\newcommand{\varm}[1]{\mathrm{Var}\big[#1\big]}
\newcommand{\cov}[2]{\operatorname{Cov}\left[{#1};{#2}\right]}

\newcommand{\ent}[1]{\mathrm{Ent}\!\left[#1\right]}

\newcommand{\expeC}[2]{\mathbb{E}\left[\left. #1 \,\right|\,#2\right]}

\newcommand{\dTV}[2]{\operatorname{d}_{\operatorname{TV}}\left({#1};{#2}\right)}

\newcommand{\dWW}[2]{\operatorname{W}_2\left({#1};{#2}\right)}

\newcommand{\dG}[2]{\operatorname{d}\left({#1};{#2}\right)}

\usepackage[colorlinks,linkcolor=black,citecolor=black,urlcolor=black]{hyperref}

\title[Higher-order pathwise theory of fluctuations]{Higher-order pathwise theory of fluctuations\\in stochastic homogenization}

\author[M. Duerinckx]{Mitia Duerinckx}
\author[F. Otto]{Felix Otto}
\address[Mitia Duerinckx]{Laboratoire de Mathématique d'Orsay, UMR 8628, Université Paris-Sud, F-91405 Orsay, France \& Universit\'e Libre de Bruxelles, Département de Mathématique, Brussels, Belgium}
\email{mduerinc@ulb.ac.be}
\address[Felix Otto]{Max Planck Institute for Mathematics in the Sciences, Leipzig, Germany}
\email{otto@mis.mpg.de}

\begin{document}
\selectlanguage{english}

\maketitle

\begin{abstract}
We consider linear elliptic equations in divergence form with stationary random coefficients of integrable correlations. We characterize the fluctuations of a macroscopic observable of a solution to relative order~$\frac{d}{2}$, where~$d$ is the spatial dimension; the fluctuations turn out to be Gaussian. As for previous work on the leading order, this higher-order characterization relies on a pathwise proximity of the macroscopic fluctuations of a general solution to those of the (higher-order) correctors, via a (higher-order) two-scale expansion injected into the ``homogenization commutator'', thus confirming the scope of this notion.
This higher-order generalization sheds a clearer light on the algebraic structure of the higher-order versions of correctors, flux correctors, two-scale expansions, and homogenization commutators. It reveals that in the same way as this algebra provides a higher-order theory for microscopic spatial oscillations, it also provides a higher-order theory for macroscopic random fluctuations, although both phenomena are not directly related.
We focus on the model framework of an underlying Gaussian ensemble, which allows for an efficient use of (second-order) Malliavin calculus for stochastic estimates. On the technical side, we introduce annealed Calder\'on-Zygmund estimates for the elliptic operator with random coefficients, which conveniently upgrade the known quenched large-scale estimates.
\end{abstract}

\setcounter{tocdepth}{1}
\tableofcontents

\section{Introduction}

\subsection{General overview}
Let $\Aa$ be a random coefficient field on $\R^d$ that is stationary and ergodic and satisfies
the boundedness and ellipticity properties 
\begin{equation}\label{eq:elliptic}
|\Aa(x)\xi|\le|\xi|,\qquad\xi\cdot\Aa(x)\xi\ge\lambda|\xi|^2,\qquad\text{for all $x,\xi\in\mathbb{R}^d$},
\end{equation}
for some $\lambda>0$, and denote by $(\Omega,\F,\Pm)$ the underlying probability space.
Given a deterministic vector field $f\in C^\infty_c(\R^d)^d$, we consider the random family $(\nabla u_\e)_{\e>0}$ of unique Lax-Milgram solutions to the rescaled problems
\begin{align}\label{eq:first-def-ups}
-\nabla\cdot \Aa(\tfrac\cdot\e)\nabla u_\e\,=\, \nabla\cdot f\qquad\text{in $\R^d$},
\end{align}
where the rescaled coefficients $\Aa(\tfrac\cdot\e)$ vary on the ``microscopic'' scale $\e$,
and we study ``macroscopic'' observables of the form $\int_{\R^d}g\cdot\nabla u_\e$ with $g\in C_c^\infty(\R^d)^d$ deterministic. Qualitative homogenization theory~\cite{PapaVara,Kozlov-79} states that almost surely $\int_{\R^d}g\cdot\nabla u_\e\to\int_{\R^d}g\cdot\nabla\bar u$ as $\e\downarrow0$, where $\nabla\bar u$ solves the (deterministic) homogenized equation
\[-\nabla\cdot\bar\Aa\nabla\bar u=\nabla\cdot f\qquad\text{in $\R^d$},\]
and the homogenized coefficients $\bar\Aa\in\R^{d\times d}$ are given by $\bar\Aa\ee_i=\expec{\Aa(\nabla\varphi_i+\ee_i)}$ for $1\le i\le d$, in terms of the corrector $\varphi_i$, that is, the unique (up to a random additive constant) almost sure solution of the corrector equation,
\begin{equation}\label{eq:corr}
-\nabla\cdot\Aa(\nabla\varphi_i+\ee_i)=0\qquad\text{in $\R^d$,}
\end{equation}
in the class of functions the gradient of which is stationary, centered (that is, $\expec{\nabla\varphi_i}=0$), and has finite second moment.
In other words, the field-flux constitutive relation $\nabla w_\e\mapsto\Aa(\tfrac\cdot\e)\nabla w_\e$ is replaced on the macroscopic scale by the effective relation $\nabla\bar w\mapsto\bar\Aa\nabla\bar w$.

\smallskip
A well-travelled question in homogenization is to further characterize the oscillations of the solution $u_\e$.
As expected from formal asymptotics, the so-called two-scale expansion\footnote{We systematically use Einstein's summation rule on repeated indices.} $(1+\e\varphi_i(\tfrac\cdot\e)\nabla_i)\bar u$ precisely captures the oscillations of $u_\e$ to order $O(\e)$ (in dimension $d>2$), in the sense of
\[\big\|\nabla\big(u_\e-(1+\e\varphi_i(\tfrac\cdot\e)\nabla_i)\bar u\big)\big\|_{\Ld^2(\R^d\times\Omega)}=O(\e).\]
This expansion has a natural geometric interpretation: by definition~\eqref{eq:corr}, the corrector $\varphi_i$ defines the correction of Euclidean coordinates $x\mapsto x_i$ into $\Aa$-harmonic coordinates $x\mapsto x_i+\varphi_i(x)$,
\[-\nabla\cdot\Aa\nabla(x_i+\varphi_i)=0,\]
so that the expansion $(1+\e\varphi_i(\tfrac\cdot\e)\nabla_i)\bar u$ amounts to locally correcting the homogenized solution $\bar u$ in terms of oscillating $\Aa$-harmonic coordinates.
Such expansions are naturally pursued to higher order:
while the (centered) first-order corrector $\varphi=\varphi^1$ is characterized by the property that $(1+\varphi_i\nabla_i)\bar\ell$ is $\Aa$-harmonic for all affine functions~$\bar\ell$, the (centered) second-order corrector $\varphi^2$ is characterized by the property that for all quadratic polynomials $\bar q$ the expansion $(1+\varphi_i\nabla_i+\varphi_{ij}^2\nabla_{ij}^2)\bar q$ fully captures the oscillations imposed by the heterogeneous elliptic operator $-\nabla\cdot\Aa\nabla$ in the sense of $-\nabla\cdot\Aa\nabla(1+\varphi_i\nabla_i+\varphi_{ij}^2\nabla_{ij}^2)\bar q$ being deterministic (see Section~\ref{sec:homog-err} for a precise statement).
The second-order two-scale expansion $(1+\e\varphi_i(\tfrac\cdot\e)\nabla_i+\e^2\varphi^2_{ij}(\tfrac\cdot\e)\nabla^2_{ij})\bar u^2_\e$ then captures the oscillations of $u_\e$ to order $O(\e^2)$ (in dimension $d>4$), where $\bar u^2_\e$
is a proxy for a solution to the generically ill-posed higher-order homogenized equation
\[-\nabla\cdot(\bar\Aa+\e\bar\Aa^2_i\nabla_i)\nabla\bar U^2_\e=\nabla\cdot f\qquad\text{in $\R^d$},\]
in terms of the second-order homogenized coefficients $\bar\Aa^2_i\ee_j:=\expecm{\Aa(\nabla\varphi^2_{ij}+\varphi_i\ee_j)}$.
In other words, the effective field-flux constitutive relation $\nabla\bar w\mapsto\bar\Aa\nabla\bar w$ is refined to its second-order version $\nabla\bar w\mapsto(\bar\Aa+\e\bar\Aa^2_i\nabla_i)\nabla\bar w$.
The proxy $\bar u_\e^2$ for a second-order homogenized solution can e.g.\@ be chosen in form of $\bar u^2_\e:=\bar u+\e\tilde u^2$ with $\tilde u^2$ solving
\[-\nabla\cdot(\bar\Aa\nabla\tilde u^2+\bar\Aa_i^2\nabla\nabla_i\bar u)=0\qquad\text{in $\R^d$}.\]
Such a description of oscillations of $u_\e$ via higher-order two-scale expansions is classical in the periodic setting~\cite{BLP-78}. It also holds in the random setting for large enough dimension and for decaying enough correlations~\cite{Gu-17,BFFO-17} (cf.~\cite{GNO2,GNO-quant,AKM-book} for the first-order level):
under appropriate mixing conditions, higher-order expansions allow to capture oscillations up to the critical order $O(\e^{d/2})$ (with a logarithmic correction in even dimensions).
The limitation is due to the fact that, in the random setting, not all higher-order correctors are well-behaved: $\varphi^n$ can be chosen stationary with finite moments only for $n<\frac d2$.
Precise statements and short proofs are included in Section~\ref{sec:homog-err} below in the Gaussian setting.

\smallskip
While periodic homogenization boils down to the description of oscillations of $u_\e$, stochastic homogenization in addition means studying the random fluctuations of the macroscopic observables $\int_{\R^d}g\cdot\nabla u_\e$.
It was recently shown~\cite{GuM} that the centered rescaled observables $\e^{-d/2}\int_{\R^d}g\cdot(\nabla u_\e-\expec{\nabla u_\e})$ converge in law to a Gaussian.
Note that this does not enter the above theory of oscillations since accuracy there is precisely limited to order $O(\e^{d/2})$.
We may naturally look for a finer description of this convergence by means of a two-scale expansion.
As observed in~\cite{GuM}, however, the limiting variance of $\e^{-d/2}\int_{\R^d}g\cdot\nabla u_\e$ generically differs from that of $\e^{-d/2}\int_{\R^d}g\cdot\nabla(1+\e\varphi_i(\tfrac\cdot\e)\nabla_i)\bar u$, which shows that two-scale expansion cannot be applied naïvely when it comes to fluctuations. In~\cite{DGO1}, we unravelled the full mechanism behind this observation by means of the so-called homogenization commutator, which led to a new pathwise theory of fluctuations in stochastic homogenization (in agreement with the pathwise heuristics in~\cite{GuM}). More precisely, fluctuations of~$\nabla u_\e$ were shown to coincide with fluctuations of a suitable deterministic (Helmholtz) projection of the corresponding homogenization commutator, which behaves like a Gaussian white noise on large scales and for which the two-scale expansion is accurate at leading order.
In the present article we explain how this result naturally extends to higher order in parallel to the known theory of oscillations: we obtain a full characterization of fluctuations of $\e^{-d/2}\int_{\R^d}g\cdot\nabla u_\e$ to order $O(\e^{d/2})$ (with again a logarithmic correction in even dimensions). The main emphasis is on unravelling the suitable higher-order commutator structure.

\smallskip
The homogenization commutator can be viewed as a commutator between large-scale averaging and the field-flux relation; on first-order level it takes the form
\[\Xi_\e^1[\nabla w_\e]:=(\Aa(\tfrac\cdot\e)-\bar\Aa)\nabla w_\e.\]
This expression first appeared in~\cite{AKM2} (see also~\cite{AS}) and is particularly natural since $H$-convergence for~\eqref{eq:first-def-ups} (cf.~\cite{MuratTartar}) is precisely equivalent to weak convergence of $\Xi_\e[\nabla u_\e]$ to $0$ (cf.~\cite{DGO1}), thus summarizing the whole qualitative homogenization theory.
This convergence property is the mathematical formulation of the Hill-Mandel relation in mechanics~\cite{H63,H72} and can further be made quantitative~\cite{DGO1}.
For a higher-order theory, we need a higher-order extension of the commutator.
In view of the second-order effective field-flux relation $\nabla\bar w\mapsto(\bar\Aa+\e\bar\Aa^2_i\nabla_i)\nabla\bar w$, the second-order commutator is naturally defined as
\[\Xi_\e^2[\nabla w_\e]:=\big(\Aa(\tfrac\cdot\e)-\bar\Aa-\e\bar\Aa^2_i\nabla_i\big)\nabla w_\e,\]
and a higher-order Hill-Mandel relation can indeed be formulated in these terms (cf.\@ Lemma~\ref{lem:higher-Hill} below).
Note that $w_\e\mapsto\Xi_\e^1[\nabla w_\e]$ and $\Xi^2_\e[\nabla w_\e]$ are viewed as first- and second-order differential operators, respectively.
Next, we define suitable two-scale expansions of these homogenization commutators.
For the first order, we inject the first-order two-scale expansion of $w_\e$ into $\Xi_\e^1[\nabla\cdot]$ and truncate the obtained differential operator at first order, thus defining the following first-order standard commutator,
\[\Xi_\e^{\circ,1}[\nabla\bar w]:=\nabla_i\bar w\,(\Aa(\tfrac\cdot\e)-\bar\Aa)(\nabla\varphi_i(\tfrac\cdot\e)+\ee_i).\]
Similarly, injecting the second-order two-scale expansion of $w_\e$ into $\Xi_\e^2[\nabla\cdot]$ and truncating the obtained differential operator at second order, we define the following second-order standard commutator,
\begin{multline*}
\Xi^{\circ,2}_\e[\nabla\bar w]:=\nabla_i\bar w\,\big(\Aa(\tfrac\cdot\e)-\bar\Aa-\e\bar\Aa^2_j\nabla_j\big)(\nabla\varphi_i(\tfrac\cdot\e)+\ee_i)\\
+\e\nabla^2_{ij}\bar w\,\Big(\big(\Aa(\tfrac\cdot\e)-\bar\Aa-\e\bar\Aa^2_l\nabla_l\big)\big(\nabla\varphi_{ij}^2(\tfrac\cdot\e)+\varphi_i(\tfrac\cdot\e)\ee_j\big)-\bar\Aa^2_j\big(\nabla\varphi_i(\tfrac\cdot\e)+\ee_i\big)\Big).
\end{multline*}
The standard commutators $\bar w\mapsto\Xi_\e^{\circ,1}[\nabla\bar w]$ and $\Xi_\e^{\circ,2}[\nabla\bar w]$ are first- and second-order differential operators, respectively, with $\e$-rescaled (distributional) stationary random coefficients. These stationary differential operators are viewed as intrinsic quantities, with coefficients given by suitable combinations of correctors and of their gradients.
The above definitions are efficiently summarized in the form
\begin{eqnarray*}
\Xi^{\circ,1}_\e[\nabla\bar w](x)&=&\Xi_\e^1\big[\nabla(1+\e\varphi_i(\tfrac\cdot\e)\nabla_i)T_x^1\bar w\big](x),\\
\Xi^{\circ,2}_\e[\nabla\bar w](x)&=&\Xi_\e^2\big[\nabla\big(1+\e\varphi_i(\tfrac\cdot\e)\nabla_i+\e^2\varphi_{ij}^2(\tfrac\cdot\e)\nabla^2_{ij}\big)T_x^2\bar w\big](x),
\end{eqnarray*}
where $T_x^1\bar w$ and $T_x^2\bar w$ denote the first- and second-order Taylor polynomials of $\bar w$ at the basepoint $x$, respectively. Note that these quantities are centered: $\expecm{\Xi_\e^{\circ,n}[\nabla\bar w]}=0$ for $n=1,2$ and any smooth deterministic field $\bar w$.

\smallskip
Pursuing these constructions to higher order, our main result is threefold and is summarized as follows, under appropriate mixing conditions:
\begin{enumerate}[(i)]
\item Up to order $O(\e^{d/2})$ in the fluctuation scaling, the fluctuations of $\nabla u_\e$ are determined by those of the corresponding higher-order commutators; cf.\@ Theorem~\ref{th:main}(i) below.
\smallskip\item Up to order $O(\e^{d/2})$ in the fluctuation scaling (with a logarithmic correction in even dimensions), the two-scale expansions of the commutators are accurate: the fluctuations of the commutators are equivalent to those of the corresponding standard commutators; cf.\@ Theorem~\ref{th:main}(ii). This fully extends the pathwise theory of~\cite{DGO1} to higher order.
\smallskip\item The standard commutators are approximately local functions of the coefficients~$\Aa$. This allows to infer a quantitative characterization of their limiting covariance structure, as well as a quantitative central limit theorem (CLT); cf.\@ Theorem~\ref{th:main}(iii)--(iv). Similarly as for the coefficient field $\Aa$ itself, while at first order the scaling limit is given by a Gaussian white noise, corrections converge to derivatives of white noise.
\end{enumerate}

\subsection{Main results}
For simplicity, we focus on the model setting of a Gaussian coefficient field, in which case a powerful Malliavin calculus is available on the probability space, substantially simplifying the analysis.
More precisely, we set
\begin{equation}\label{eq:def-A}
\Aa(x):=a_0(G(x)),
\end{equation}
where $G$ is some $\R^\kappa$-valued centered stationary Gaussian random field on $\R^d$ constructed on a probability space $(\Omega,\F,\Pm)$ and where $a_0\in C^2_b(\R^\kappa)^{d\times d}$ is such that the ellipticity and boundedness assumptions~\eqref{eq:elliptic} are satisfied.
We use scalar notation $a_0',a_0''$ for derivatives of~$a_0$.
In addition, we assume that the Gaussian field $G$ has integrable correlations: more precisely, it has the form $G=c_0\ast\xi$, where $\xi$ denotes $\kappa$-dimensional white noise on $\R^d$ and where the kernel $c_0:\R^d\to\R^{\kappa\times\kappa}$ is chosen even (that is, $c_0(x)=c_0(-x)$ for all $x$) and satisfies the integrability condition
\begin{equation}\label{eq:cov-L1}
\int_{\R^d}\Big(\sup_{B(x)}|c_0|\Big)\,dx\,\le\,1,
\end{equation}
where $B(x)$ is the unit ball centered at $x$.
The covariance function of $G$ is then given by $c:=c_0\ast c_0$ and satisfies $\int_{\R^d}(\sup_{B(x)}|c|)\,dx\,\le\,1$.
The case of non-integrable correlations could be treated similarly, but it would lead to different scalings, cf.~\cite{DFG1}. Our main result is as follows.

\begin{theor1}\label{th:main}
Consider the above Gaussian setting with integrable correlations~\eqref{eq:cov-L1} and set $\ell:=\lceil\frac d2\rceil$, that is, the smallest integer $\ge\frac d2$. For all $1\le n\le\ell$, let $\bar u_\e^{n}$ and $\bar v_\e^{n}$ be the $n$th-order homogenized solutions of the primal and dual equations, respectively, as in Definition~\ref{def:homog-eqn} below, and let the $n$th-order homogenization commutator $\Xi_\e^n$ and its standard version $\Xi_\e^{\circ,n}$ be as in Definitions~\ref{def:Xin} and~\ref{def:Xin0}.
Then the following hold for all $0<\e\le\frac12$.
\begin{enumerate}[(i)]
\item \emph{Reduction to commutators:} For all $f,g\in C^\infty_c(\R^d)^d$, $1\le n\le\ell$, and $p<\infty$,
\begin{equation*}
\qquad N_p\bigg(\e^{-\frac d2}\int_{\R^d}g\cdot\nabla u_\e-\e^{-\frac d2}\int_{\R^d}\nabla\bar v_\e^{n}\cdot\Xi_\e^n[\nabla u_\e]\bigg)
\,\lesssim_{p,f,g}\,\e^{n},
\end{equation*}
where $N_p(X):=\expec{|X-\expec{X}\!|^p}^\frac1p$.
\smallskip\item \emph{Two-scale expansion of commutators:} For all $f,g\in C^\infty_c(\R^d)^d$, $1\le n\le\ell$, and $p<\infty$,
\[\qquad N_p\bigg(\e^{-\frac d2}\int_{\R^d} g\cdot\big(\Xi_\e^{n}[\nabla u_\e]-\Xi_\e^{\circ,n}[\nabla\bar u_\e^{n}]\big)\bigg)
\,\lesssim_{p,f,g}\,\e^n\mu_{d,n}(\tfrac1\e),\]
where
\begin{equation}\label{eq:pre-def-mudn}
\qquad\e^n\mu_{d,n}(\tfrac1\e)\,:=\,\left\{\begin{array}{lll}
\e^n&:&\text{$n<\ell$},\\
\e^{\frac d2}\Log^\frac12&:&\text{$n=\ell$ and $d$ even},\\
\e^{\frac d2}&:&\text{$n=\ell$ and $d$ odd}.
\end{array}\right.
\end{equation}
\item \emph{Local covariance structure:}
Further assume that $\int_{\R^d}|y|^\ell|c(y)|dy\le1$.
There exists an $(m+4)$th-order tensor $\calQ_{l'k'lk}^m$ for all $l',k',l,k\ge0$ with $l'+k'+l+k=m\le\ell-1$, such that for all $g,\nabla\bar w\in C^\infty_c(\R^d)^d$ and $1\le n\le\ell$, setting
\[\qquad\calQ_\e^n[g,\nabla\bar w]:=\sum_{m=0}^{n-1}\e^{m}\sum_{l',k',l,k\ge0\atop l'+k'+l+k=m}\calQ_{l'k'lk}^m\odot\int_{\R^d}\nabla^{l'}g\otimes\nabla^{k'}\nabla\bar w\otimes\nabla^{l}g\otimes\nabla^{k}\nabla\bar w,\]
where $\odot$ denotes the complete contraction of tensors of the same order,
there holds
\[\qquad\bigg|\var{\e^{-\frac d2}\int_{\R^d}g\cdot\Xi_{\e}^{\circ,n}[\nabla\bar w]}-\calQ_\e^n[g,\nabla\bar w]\bigg|
\lesssim_{g,\bar w} \e^n\mu_{d,n}(\tfrac1\e).\]
\item \emph{Normal approximation:}
For all $g,\nabla\bar w\in C^\infty_c(\R^d)^d$ and $1\le n\le\ell$,
\[\qquad\operatorname{d}_{\Nc}\Big({\e^{-\frac d2}\int_{\R^d}g\cdot\Xi_\e^{\circ,n}[\nabla\bar w]}\Big)
\,\lesssim_{g,\bar w}\,\frac{\e^{\frac d2}\Log\big(\Log\log^2\!\Log\big)^{\frac1{2d}n(n-1)}}{\varm{\e^{-\frac d2}\int_{\R^d}g\cdot\Xi_\e^{\circ,n}[\nabla\bar w]}},\]
where we have set
\[\operatorname{d}_{\Nc}(X)\,:=\,\dWW{\frac{X-\expec{X}}{\var{X}^{\frac12}}}{\Nc}+\dTV{\frac{X-\expec{X}}{\var{X}^{\frac12}}}{\Nc},\]
where $\dWW\cdot\Nc$ and $\dTV\cdot\Nc$ denote the $2$-Wasserstein and the total variation distances to a standard Gaussian law, respectively.
\end{enumerate}
In the estimates above, the notation $\lesssim_\gamma$ (with parameters $\gamma$) stands for $\le$ up to a multiplicative constant $C_\gamma\ge1$ that only depends on $d$, $\lambda$, $\|a_0\|_{W^{2,\infty}}$, and on the parameters $\gamma$, through an upper bound on suitable (weighted) Sobolev norms of $\gamma$ if $\gamma$ is a function.
\end{theor1}

Choosing $n=\ell$, we are thus led to an intrinsic description of the fluctuations of macroscopic observables $\int_{\R^d} g\cdot(\nabla u_\e-\expec{\nabla u_\e})$ with accuracy $O(\e^d)$ (with a logarithmic correction in even dimensions), that is, up to the {square of the CLT order}. In particular, non-Gaussian corrections are only expected beyond that order.

\begin{cor1}
With the assumptions and notation of Theorem~\ref{th:main}, for all $f,g\in C^\infty_c(\R^d)^d$, in terms of the intrinsic quantity
\[X_\e^\circ(g,f):=\e^{-\frac d2}\int_{\R^d}\nabla\bar v_\e^\ell\cdot\big(\Xi_\e^{\circ,\ell}[\nabla\bar u_\e^\ell]-\expecm{\Xi_\e^{\circ,\ell}[\nabla\bar u_\e^\ell]}\big),\]
there holds
for all $p<\infty$,
\begin{equation*}
\expec{\Big(\e^{-\frac d2}\int_{\R^d}g\cdot(\nabla u_\e-\expec{\nabla u_\e})-X_\e^\circ(g,f)\Big)^p}^\frac1p
\,\lesssim_{p,f,g}\,\e^\frac d2\times\left\{\begin{array}{lll}
\Log^{\frac12}&:&\text{$d$ even},\\
1&:&\text{$d$ odd},
\end{array}\right.
\end{equation*}
and, under the non-degeneracy condition $\calQ_{0000}\odot\int_{\R^d}(\nabla\bar v\otimes\nabla\bar u)^{\otimes2}\ne0$,
\begin{align}\label{eq:CLT-sq}
\dG{X_\e^\circ(g,f)~}{~\calQ_\e^\ell[\nabla\bar v_\e^\ell,\nabla\bar u_\e^\ell]^\frac12\,\Nc}
\,\lesssim_{f,g}\,\e^\frac d2\Log^{1+\rho}(\log\Log)^{2\rho},
\end{align}
where $\dG\cdot\cdot$ stands for the sum of the $2$-Wasserstein and the total variation distances, where $\Nc$ denotes a standard Gaussian random variable, and where
\[\rho\,:=\,\frac1{8d}\times\left\{\begin{array}{lll}
d(d-2)&:&\text{$d$ even},\\
(d+1)(d-1)&:&\text{$d$ odd}.
\end{array}\right.\qedhere\]
\end{cor1}

Another question that one may formulate concerns the description of fluctuations of higher-order correctors --- although this is somewhat artificial since only combinations of correctors in form of two-scale expansions are intrinsically relevant. As briefly described in Remark~\ref{rem:fluc-cor+}, the proof of the above main result can be adapted in a straightforward way to also describe the fluctuations of $\e^{n-\frac d2}\varphi^{n}(\tfrac\cdot\e)$ and $\e^{n-\frac d2}\nabla\varphi^{n+1}(\tfrac\cdot\e)$ with accuracy $O(\e^{\frac d2})$ for all $n\le\ell-1$ (with a logarithmic correction in even dimensions).

\subsection*{Outline of the article}
In Section~\ref{sec:homog-err}, we start by recalling the theory of oscillations in stochastic homogenization in terms of higher-order correctors and two-scale expansions.
Section~\ref{sec:commut} is devoted to the definition and elementary properties of higher-order homogenization commutators, including the (easy) proof of item~(i) of Theorem~\ref{th:main} as well as the statement of quantitative higher-order Hill-Mandel relations. Next, we turn to the rest of the proof of Theorem~\ref{th:main}, which is based on the following three main ingredients.
\begin{enumerate}[$\bullet$]
\item\emph{Malliavin calculus:}
In the present Gaussian setting, given a random variable $X$, a Poincaré inequality holds in the form $\var{X}\le C\,\expec{\|DX\|_\Hf^2}$ (cf.~\cite{Chernoff-81,Houdre-PerezAbreu-95}), where the Malliavin derivative $DX$ encodes the infinitesimal variation of $X$ with respect to changes in the underlying Gaussian field $G$ and where $\|\cdot\|_\Hf$ is the Hilbert norm associated with the covariance structure of $G$. Similarly, the approximate normality of~$X$ can be estimated by the size of $D^2X$ with help of a so-called second-order Poincaré inequality~\cite{C2,NPR-09}.
With these functional analytic tools at hand, we are thus reduced to estimating infinitesimal variations of the quantities of interest with respect to $G$, thus somehow linearizing the dependence on $G$, which is then particularly amenable to PDE methods. On a more technical level, rather than using ad hoc derivatives with respect to the coefficient field as in previous works in stochastic homogenization, we use here the full power of the Gaussian Malliavin calculus. Relevant definitions and results are recalled in Section~\ref{chap:Mall}.
\smallskip\item\emph{Representation formulas:}
In view of applying the above tools from Malliavin calculus, we consider the infinitesimal variations of the quantities of interest, for which we establish suitable representation formulas. This requires a good understanding of the algebra behind higher-order correctors and commutators, and is performed in Section~\ref{sec:rep-form}.
\smallskip\item\emph{PDE ingredients:}
It remains to estimate the infinitesimal variations of interest from their representation formulas.
Based on the large-scale Lipschitz regularity theory for $\Aa$-harmonic functions as developed in~\cite{AS,Armstrong-Mourrat-16,AKM2,GNO-reg}, we introduce in Section~\ref{sec:Lpreg} a new annealed Calder\'on-Zygmund theory for linear elliptic equations in divergence form with random coefficients.
This constitutes an upgrade of the quenched large-scale Calder\'on-Zygmund theory of~\cite{Armstrong-Daniel-16,AKM-book,GNO-reg}
(see also~\cite{JO-19} for a direct approach);
it happens to be particularly well suited for our purposes in this article and is of independent interest.
\end{enumerate}
Finally, with these three ingredients at hand, we turn to the proof of items~(ii), (iii), and~(iv) of Theorem~\ref{th:main} in Sections~\ref{sec:pathw}, \ref{sec:cov-conv}, and~\ref{sec:normal}, respectively.

\subsection*{Notation}
\begin{enumerate}[$\bullet$]
\item We denote by $C\ge1$ any constant that only depends on $d$, $\lambda$, and $\|a_0\|_{W^{2,\infty}}$. We use the notation $\lesssim$ (resp.~$\gtrsim$) for $\le C\times$ (resp.\@ $\ge\frac1C\times$) up to such a multiplicative constant $C$. We write $\simeq$ when both $\lesssim$ and $\gtrsim$ hold. We add subscripts to $C$, $\lesssim$, $\gtrsim$, $\simeq$ in order to indicate dependence on other parameters. If the subscript is a function, then it is understood as dependence on an upper bound on suitable (weighted) Sobolev norms of the function.
\smallskip\item The ball centered at $x$ of radius $r$ in $\R^d$ is denoted by $B_r(x)$, and we simply write $B(x)$ if $r=1$.
\smallskip\item For a function $f$ and $1\le p<\infty$, we write $[f]_p(x):=(\fint_{B(x)}|f|^p)^{1/p}$ for the local moving $\Ld^p$ average, and similarly $[f]_\infty(x):=\sup_{B(x)}|f|$.
\smallskip\item For a $k$th-order tensor $T$ we define for any fixed indices $i_1,\ldots,i_k$,
\[\sym_{i_1\ldots i_k}T_{i_1\ldots i_k}\,:=\,\frac1{k!}\sum_{\sigma\in\mathcal S_k}T_{i_{\sigma(1)}\ldots i_{\sigma(k)}},\]
with $\mathcal S_k$ denoting the set of permutations of the set $\{1,\ldots,k\}$. Although this notation may seem redundant at this point, it will prove very useful in the sequel.
\smallskip\item For $r\in\R$ we denote by $\lceil r\rceil$ the smallest integer~$\ge r$ and by $\lfloor r\rfloor$ the largest integer $\le r$. For $r,s\in\R$ we write $r\vee s:=\max\{r,s\}$ and $r\wedge s:=\min\{r,s\}$.
We set $\langle x\rangle:=(1+|x|^2)^{1/2}$ for $x\in\R^d$, and $\langle\nabla\rangle$ denotes the corresponding pseudo-differential operator.
\end{enumerate}

\section{Higher-order theory of oscillations}\label{sec:homog-err}
In this section, as a prelude to the study of higher-order fluctuations, we describe the classical higher-order theory of oscillations in stochastic homogenization in terms of two-scale expansions, including some new results and shorter arguments.
We start by recalling the definition of higher-order correctors, homogenized coefficients, fluxes, and flux correctors as motivated by formal two-scale expansions (e.g.~\cite{BLP-78,JKO94}).
A simple iterative way to view this definition is as follows: for $n\ge1$, assuming that the previous correctors $\varphi^1,\ldots,\varphi^{n-1}$ are well-defined as centered stationary objects, the next corrector gradient $\nabla\varphi^n$ is the unique centered stationary object that is characterized by the property that for all $n$th-order polynomials $\bar q$ the corrected polynomial $F^n[\bar q]:=\bar q+\sum_{k=1}^n\varphi^k_{i_1\ldots i_k}\nabla_{i_1\ldots i_k}^k\bar q$ captures oscillations of the heterogeneous operator in the sense that $\nabla\cdot\Aa\nabla F^n[\bar q]$ is deterministic. In that case, there must hold (cf.~\eqref{eq:pre-dec-homog})
\[\nabla\cdot\Aa\nabla F^n[\bar q]=\nabla\cdot\Big(\sum_{k=1}^{n-1}\bar\Aa^k_{i_1\ldots i_{k-1}}\nabla^{k-1}_{i_1\ldots i_{k-1}}\Big)\nabla\bar q,\]
which in addition characterizes the higher-order homogenized coefficients $\bar\Aa^2,\ldots,\bar\Aa^{n-1}$ next to the first-order one $\bar\Aa^1:=\bar\Aa$.
In other words, the higher-order corrector expansion $F^n[\cdot]$ precisely intertwines the heterogeneous elliptic operator with its higher-order homogenized version.
Note that the above only characterizes the symmetric part $\sym_{i_1\ldots i_n}\varphi_{i_1\ldots i_n}$ of correctors, which is indeed the only relevant quantity in view of two-scale expansions. It is however convenient to define correctors with a suitable non-symmetric part in such a way that the corresponding fluxes can be written in terms of skew-symmetric flux correctors, which is a well-known crucial tool in view of error estimates (e.g.~\cite[p.27]{JKO94}).
As stated below, in contrast with the periodic case, only $\ell:=\lceil\frac d2\rceil$ correctors can be defined, with the $\ell-1$ first of them being stationary; this number would be further reduced in the case of coefficient fields with stronger correlations (e.g.~\cite{GNO-quant,BFFO-17}).

\begin{defin}[Higher-order correctors]\label{def:cor}
Given $\ell:=\lceil\frac d2\rceil$, we inductively define the correctors $(\varphi^n)_{0\le n\le\ell}$, the homogenized coefficients $(\bar\Aa^n)_{1\le n\le\ell}$, the fluxes $(q^n)_{1\le n\le\ell}$, and the flux correctors $(\sigma^n)_{0\le n\le\ell}$ as follows:
\begin{enumerate}[$\bullet$]
\item $\varphi^0:=1$ and for all $1\le n\le\ell$ we define $\varphi^n:=(\varphi^n_{i_1\ldots i_n})_{1\le i_1,\ldots,i_n\le d}$ with $\varphi^n_{i_1\ldots i_n}$ a scalar field that satisfies
\begin{gather*}
-\nabla\cdot\Aa\nabla\varphi^n_{i_1\ldots i_n}=\nabla\cdot\big((\Aa\varphi^{n-1}_{i_1\ldots i_{n-1}}-\sigma^{n-1}_{i_1\ldots i_{n-1}})\,\ee_{i_n}\big),\\
\nabla\varphi^n_{i_1\ldots i_n}~\text{stationary},\quad\expec{|\nabla\varphi^n_{i_1\ldots i_n}|^2}<\infty,\quad \expec{\nabla\varphi^n_{i_1\ldots i_n}}=0.
\end{gather*}
For $n<\ell$ we can choose $\varphi^n$ itself stationary with $\expec{\varphi^n}=0$, while for $n=\ell$ we choose the anchoring $\varphi^\ell(0)=0$.
\smallskip\item For all $1\le n\le \ell$ we define $\bar\Aa^n:=(\bar\Aa^n_{i_1\ldots i_{n-1}})_{1\le i_1,\ldots,i_{n-1}\le d}$ with $\bar\Aa^n_{i_1\ldots i_{n-1}}$ the matrix given by
\begin{equation}\label{eq:def-baran}
\bar\Aa^n_{i_1\ldots i_{n-1}}\ee_{i_n}:=\expecm{\Aa\big(\nabla\varphi^{n}_{i_1\ldots i_n}+\varphi^{n-1}_{i_1\ldots i_{n-1}}\ee_{i_n}\big)}.
\end{equation}
\item For all $1\le n\le \ell$ we define $q^n:=(q^n_{i_1\ldots i_n})_{1\le i_1,\ldots,i_n\le d}$ with $q^n_{i_1\ldots i_n}$ the vector field given by
\[q^n_{i_1\ldots i_n}:=\Aa\nabla\varphi^{n}_{i_1\ldots i_n}+(\Aa\varphi^{n-1}_{i_1\ldots i_{n-1}}-\sigma^{n-1}_{i_1\ldots i_{n-1}})\,\ee_{i_n}-\bar\Aa_{i_1\ldots i_{n-1}}^{n}\ee_{i_n},\]
where the definition of $\bar\Aa^n$ and the normalization of $\sigma^{n-1}$ below ensure $\expec{q^n}=0$.
\smallskip\item $\sigma^0:=0$ and for all $1\le n\le\ell$ we define $\sigma^n:=(\sigma^n_{i_1\ldots i_n})_{1\le i_1,\ldots,i_n\le d}$ with $\sigma^n_{i_1\ldots i_n}$ a skew-symmetric matrix field that satisfies
\begin{gather*}
-\triangle\sigma^n_{i_1\ldots i_n}=\nabla\times q^n_{i_1\ldots i_n},\qquad\nabla\cdot\sigma^n_{i_1\ldots i_n}=q^n_{i_1\ldots i_n},\\
\nabla\sigma^n_{i_1\ldots i_n}~\text{stationary},\quad\expec{|\nabla\sigma^n_{i_1\ldots i_n}|^2}<\infty,\quad \expec{\nabla\sigma^n_{i_1\ldots i_n}}=0,
\end{gather*}
with the notation $(\nabla\times X)_{ij}:=\nabla_iX_j-\nabla_jX_i$ for a vector field $X$ and with the notation $(\nabla\cdot Y)_i:=\nabla_jY_{ij}$ for a matrix field $Y$. For $n<\ell$ we can choose $\sigma^n$ itself stationary with $\expec{\sigma^n}=0$, while for $n=\ell$ we choose the anchoring $\sigma^\ell(0)=0$.\qedhere
\end{enumerate}
\end{defin}

We first state that the above definition of higher-order correctors indeed makes sense.
This easily follows from~\cite[proof of Proposition~9 and Lemma~12]{BFFO-17} together with Lemma~\ref{lem:key-estimates-decomp} below in the considered Gaussian setting (see also~\cite{GNO-reg,GNO-quant,AKM-book}),
and the proof is omitted. Note that the optimal stochastic integrability (that is, the optimal constants $c_n$'s below) is not required in our analysis, hence is not addressed here.

\begin{prop}[Higher-order correctors~\cite{GNO-reg,GNO-quant,BFFO-17,AKM-book}]\label{prop:cor}
All the quantities in Definition~\ref{def:cor} exist and are uniquely defined. In addition, there exists a sequence $(c_n)_{0\le n\le\ell}$ of positive numbers with $c_0=\frac12$ such that for all $0\le n\le\ell$ and $1\le p<\infty$,
\[|\bar\Aa^n|\lesssim1,\qquad\expec{[\nabla\varphi^{n}]_2^p}^\frac1p\lesssim p^{c_{n-1}},\qquad\expec{[\varphi^n]_2^p(x)}^\frac1p+\expec{[\sigma^n]_2^p(x)}^\frac1p\,\lesssim\,p^{c_n}\mu_{d,n}(x),\]
where in line with~\eqref{eq:pre-def-mudn} we have set
\[\mu_{d,n}(x)\,:=\,\left\{\begin{array}{lll}
1&:&\text{$n<\ell$},\\
\log^\frac12(2+|x|)&:&\text{$n=\ell$ and $d$ even},\\
1+|x|^\frac12&:&\text{$n=\ell$ and $d$ odd}.
\end{array}\right.\qedhere\]
\end{prop}

For later reference we also state the following result on the fluctuation scaling of correctors, showing that higher-order correctors have stronger correlations hence worse fluctuation scalings, as first emphasized in~\cite{Gu-17}.

\begin{lem}[Fluctuation scaling of higher-order correctors]\label{lem:average-cor}
For all $g\in C^\infty_c(\R^d)$, $0\le n\le\ell-1$, and $1\le p<\infty$,
\begin{equation*}
\expec{\Big|\int_{\R^d}g(x)\big(\nabla\varphi^{n+1},\varphi^{n},\sigma^{n}\big)(\tfrac x\e)\,dx\Big|^p}^\frac1p
\lesssim_p\e^{\frac{d}{2}-n}\|[g]_2\|_{\Ld^\frac{2d}{d+2n}(\R^d)}.\qedhere
\end{equation*}
\end{lem}

If $\Aa$ is replaced by the pointwise transpose field $\Aa^*$, we write $(\varphi^{*,n})_{0\le n\le\ell}$, $(\bar\Aa^{*,n})_{1\le n\le\ell}$, $(q^{*,n})_{1\le n\le\ell}$, and $(\sigma^{*,n})_{0\le n\le\ell}$ for the corresponding objects.
While it is well-known that $\bar\Aa^{*,1}=(\bar\Aa^1)^*$, the following lemma extends this relation to higher order.

\begin{lem}[Symmetries of homogenized coefficients]\label{lem:sym-baran}
For all $1\le n\le\ell$,\footnote{For $n=1$ the notation $\bar\Aa^{*,n}_{ji_1\ldots i_{n-2}}\ee_{i_{n-1}}$ stands for $\bar\Aa^{*,1}\ee_{j}$, and we use a similar unifying notation throughout in the sequel.}
\begin{align}\label{eq:total-migration}
\sym_{i_1\ldots i_n}\big(\ee_j\cdot\bar\Aa^n_{i_1\ldots i_{n-1}}\ee_{i_n}\big)\,=\,(-1)^{n+1}\sym_{i_1\ldots i_n}\big(\ee_{i_n}\cdot\bar\Aa^{*,n}_{ji_1\ldots i_{n-2}}\ee_{i_{n-1}}\big).
\end{align}
In particular, if $\Aa$ is symmetric and if $n$ is even, there holds
\[\sym_{i_1\ldots i_{n+1}}\big(\ee_{i_{n+1}}\cdot\bar\Aa^n_{i_1\ldots i_{n-1}}\ee_{i_n}\big)=0.\qedhere\]
\end{lem}

\begin{rem}\label{rem:migration-BS}
The proof of Lemma~\ref{lem:sym-baran} proceeds by letting the successive inverse operators $(\nabla\cdot\Aa\nabla)^{-1}\nabla\cdot$ defining $(\nabla\varphi^n,\varphi^{n-1})$ ``migrate'' to the test function $\Aa$ in the definition~\eqref{eq:def-baran} of $\bar\Aa^n$.
Using the algebra of correctors, a full migration precisely leads to the dual expression $\bar\Aa^{*,n}$ and proves the relation~\eqref{eq:total-migration}.
Rather stopping this migration argument at an intermediate step, we obtain for all $1\le m< n\le\ell$,
\begin{multline*}
\sym_{i_1\ldots i_n}\big(\ee_j\cdot\bar\Aa^n_{i_1\ldots i_{n-1}}\ee_{i_n}\big)\,=\,(-1)^{m+1}\,\sym_{i_1\ldots i_n}\E\Big[\nabla\varphi^{*,m+1}_{ji_n\ldots i_{n-m+1}}\cdot\Aa\nabla\varphi^{n-m}_{i_1\ldots i_{n-m}}\\
-\varphi^{*,m}_{ji_n\ldots i_{n-m+2}}\ee_{i_{n-m+1}}\cdot\Aa\varphi^{n-m-1}_{i_1\ldots i_{n-m-1}}\ee_{i_{n-m}}\Big].
\end{multline*}
Note that this turns into~\eqref{eq:total-migration} for $m=n-1$.
We deduce for $n=2m+1$ odd,
\begin{multline*}
\sym_{i_1\ldots i_n}\big(\ee_j\cdot\bar\Aa^n_{i_1\ldots i_{n-1}}\ee_{i_n}\big)
\,=\,(-1)^{m+1}\,\sym_{i_1\ldots i_n}\E\Big[\nabla\varphi^{*,m+1}_{ji_n\ldots i_{m+2}}\cdot\Aa\nabla\varphi^{m+1}_{i_1\ldots i_{m+1}}\\
-\varphi^{*,m}_{ji_n\ldots i_{m+3}}\ee_{i_{m+2}}\cdot\Aa\varphi^{m}_{i_1\ldots i_{m}}\ee_{i_{m+1}}\Big],
\end{multline*}
and for $n=2m$ even, further using the algebra of correctors to symmetrize the expression,
\begin{multline*}
\sym_{i_1\ldots i_n}\big(\ee_j\cdot\bar\Aa^n_{i_1\ldots i_{n-1}}\ee_{i_n}\big)\\
\,=\,(-1)^{m+1}\,\sym_{i_1\ldots i_n}\E\Big[\big(\nabla\varphi^{*,m}_{ji_n\ldots i_{m+2}}+\varphi^{*,m-1}_{ji_n\ldots i_{m+3}}\ee_{i_{m+2}}\big)\cdot \Aa \varphi^{m}_{i_1\ldots i_{m}}\ee_{i_{m+1}}\\
-\varphi^{*,m}_{ji_n\ldots i_{m+2}}\ee_{i_{m+1}}\cdot \Aa\big(\nabla\varphi^{m}_{i_1\ldots i_{m}}+\varphi^{m-1}_{i_1\ldots i_{m-1}}\ee_{i_{m}}\big)\Big].
\end{multline*}
This shows that $\bar\Aa^n$ can be defined in terms of correctors $(\nabla\varphi^{s+1},\varphi^s)$ with $s\le \lfloor\frac n2\rfloor$ only.
This simple observation was independently made in~\cite[Theorem~3.5]{Pouch-19} (see also~\cite{Pouch-17}) and has the following striking consequence:
Choosing a periodization in law $\Aa_L$ of the coefficient field $\Aa$ and recalling that the periodic Poincaré inequality ensures that all the correctors $\varphi_L^n$ are defined as stationary objects so that all the corresponding homogenized coefficients $\bar\Aa_L^n$ are well-defined, the above computation entails that $\sym_{i_1\ldots i_n}\big(\bar\Aa^n_{L;i_1\ldots i_{n-1}}\ee_{i_n}\big)$ has a limit as $L\uparrow\infty$ for all $n\le d$ if $d$ is odd and for all $n< d$ if $d$ is even.
This partially solves a weak version of the Bourgain-Spencer conjecture~\cite{Bourgain-18,Lemm-18} as stated in~\cite[Section~4.2]{DGL}. Refinements in this direction are postponed to a future work.
\end{rem}

We turn to the accuracy of two-scale expansions in the homogenization regime.
Given $f\in\Ld^2(\R^d)^d$, we consider the unique Lax-Milgram solution $\nabla u_\e$ of the rescaled elliptic PDE~\eqref{eq:first-def-ups},
\[-\nabla\cdot\Aa(\tfrac\cdot\e)\nabla u_\e=\nabla\cdot f\qquad\text{in $\R^d$}.\]
Standard two-scale expansion techniques~\cite{BLP-78} formally suggest
\[u_{\e}=\sum_{k=0}^n\e^k\varphi_{i_1\ldots i_k}^k(\tfrac\cdot\e)\nabla^k_{i_1\ldots i_k}\bar U_{\e}^n+O(\e^{n+1}),\]
where $\bar U_{\e}^n$ satisfies the $n$th-order homogenized equation
\begin{align}\label{eq:eff-eqn-formal}
-\nabla\cdot\Big(\sum_{k=1}^n\e^{k-1}\bar\Aa_{i_1\ldots i_{k-1}}^k\nabla_{i_1\ldots i_{k-1}}^{k-1}\Big)\nabla\bar U_{\e}^n=\nabla\cdot f\qquad\text{in $\R^d$}.
\end{align}
The effective field-flux constitutive relation $\nabla\bar w\mapsto\bar\Aa\nabla\bar w$ is thus refined into the higher-order relation
\begin{align}\label{eq:higher-constitutive-rel}
\nabla\bar w\mapsto\sum_{k=1}^n\e^{k-1}\bar\Aa^k_{i_1\ldots i_{k-1}}\nabla\nabla^{k-1}_{i_1\ldots i_{k-1}}\bar w,
\end{align}
which includes dispersive corrections.
However, equation~\eqref{eq:eff-eqn-formal} is ill-posed in general for $n$ even and a suitable proxy needs to be devised (cf.\@ also~\cite{KMS-06}).
We start by introducing a notation for two-scale expansions.

\begin{defin}[Higher-order two-scale expansions]\label{def:EnFn}
For $0\le n\le\ell$, given a deterministic smooth function $\bar w$, its {$n$th-order two-scale expansion} $F^n_\e[\bar w]$ is defined as
\[F^n_\e[\bar w]:=\sum_{k=0}^n\e^k\varphi_{i_1\ldots i_k}^k(\tfrac\cdot\e)\,\nabla_{i_1\ldots i_k}^k\bar w.\]
At the level of gradients, we similarly define for $0\le n\le\ell$,
\[E^n_\e[\nabla\bar w]\,:=\,\sum_{k=0}^{n-1}\e^k\,\big(\nabla\varphi_{i_1\ldots i_{k+1}}^{k+1}+\varphi_{i_1\ldots i_{k}}^{k}\ee_{i_{k+1}}\big)(\tfrac\cdot\e)\,\nabla_{i_1\ldots i_{k+1}}^{k+1}\bar w.\qedhere\]
\end{defin}

In particular, note that the quantities $\nabla F^n_\e[\bar w]$ and $E^n_\e[\nabla\bar w]$ are related as follows: for all $0\le n\le\ell$,
\begin{align}\label{eq:link-E0E}
\nabla F^n_\e[\bar w]
\,=\,E^n_\e[\nabla\bar w]+\e^n\varphi_{i_1\ldots i_n}^n\nabla\nabla_{i_1\ldots i_{n}}^{n}\bar w,
\end{align}
where all coefficients of $E^n_\e[\nabla\bar w]$ are stationary.
We turn to the definition of a suitable proxy for the ill-posed higher-order homogenized equation~\eqref{eq:eff-eqn-formal}. For later reference we simultaneously consider the dual equations.

\begin{defin}[Higher-order homogenized equations]\label{def:homog-eqn}
Given $f,g\in C^\infty_c(\R^d)^d$, let $\nabla u_\e$ and $\nabla v_\e$ be the unique Lax-Milgram solutions of the rescaled elliptic PDEs
\begin{align}\label{eq:ueps-veps}
-\nabla\cdot\Aa(\tfrac\cdot\e)\nabla u_\e=\nabla\cdot f,\qquad-\nabla\cdot\Aa^*(\tfrac\cdot\e)\nabla v_\e=\nabla\cdot g\qquad\text{in $\R^d$}.
\end{align}
For $1\le n\le\ell$, we define the corresponding {$n$th-order homogenized solutions}
\[\bar u^n_\e:=\sum_{k=1}^n\e^{k-1}\tilde u^k,\qquad\bar v^n_\e:=\sum_{k=1}^n\e^{k-1}\tilde v^k,\]
where $\nabla\tilde u^1=\nabla\bar u$ and $\nabla\tilde v^1=\nabla\bar v$ are the unique Lax-Milgram solutions in $\R^d$ of the first-order homogenized equations,
\[-\nabla\cdot\bar\Aa\nabla\bar u=\nabla\cdot f,\qquad -\nabla\cdot\bar\Aa^{*}\nabla\bar v=\nabla\cdot g,\]
and where for $2\le n\le\ell$ we inductively define the $n$th-order corrections $\nabla\tilde u^n$ and $\nabla\tilde v^n$ as the unique Lax-Milgram solutions in $\R^d$ of
\begin{gather*}
-\nabla\cdot\bar\Aa\nabla\tilde u^n=\nabla\cdot\sum_{k=2}^n\bar\Aa^k_{i_1\ldots i_{k-1}}\nabla\nabla^{k-1}_{i_1\ldots i_{k-1}}\tilde u^{n+1-k},\\
-\nabla\cdot\bar\Aa^{*}\nabla\tilde v^n=\nabla\cdot\sum_{k=2}^n\bar\Aa^{*,k}_{i_1\ldots i_{k-1}}\nabla\nabla^{k-1}_{i_1\ldots i_{k-1}}\tilde v^{n+1-k}.\qedhere
\end{gather*}
\end{defin}

With the above notation, the main result from the higher-order corrector theory in stochastic homogenization takes the following guise.
For $q=2$, it is a straightforward consequence of the definition of higher-order correctors and of the corrector estimates of Proposition~\ref{prop:cor}. It extends~\cite{GNO2,GNO-quant,BFFO-17,AKM-book} to higher order and completes~\cite[Theorem~1.6]{Gu-17} up to the optimal order with finer norms. For $n=\ell$, the obtained maximal accuracy is $O(\e^{d/2})$ (with a logarithmic correction in even dimensions), indicating that the study of fluctuations goes beyond the theory of oscillations.

\begin{prop}[Accuracy of higher-order two-scale expansions]\label{prop:homog-err}
For all $f\in C^\infty_c(\R^d)^d$, $1\le n\le\ell$, and $1< q\le p<\infty$,
\begin{equation}\label{eq:concl-homog-err}
\expec{\big\|\big[\nabla(u_\e-F^n_\e[\bar u_\e^{n}])\big]_2\big\|_{\Ld^q(\R^d)}^p}^\frac1p\,\lesssim_{q,p}\,\e^n\mu_{d,n}(\tfrac1\e)\,\big\|\mu_{d,n}[\langle\nabla\rangle^{2n-1}f]_\infty\big\|_{\Ld^q(\R^d)}.\qedhere
\end{equation}
\end{prop}

\subsection{Proof of Lemma~\ref{lem:average-cor}}
The result is easily obtained based on Malliavin calculus and on the annealed Calder\'on-Zygmund theory of Section~\ref{sec:Lpreg}.
More precisely, for $0\le n\le\ell-1$, we start from Proposition~\ref{prop:Mall}(iii) together with~\eqref{eq:cov-L1-H} in the form
\begin{equation*}
\expec{\Big|\int_{\R^d}g\,(\nabla\varphi^{n+1},\varphi^n,\sigma^n)\Big|^p}^\frac1p
\lesssim_p\,\expec{\bigg(\int_{\R^d}\Big[\int_{\R^d}g\,(\nabla D\varphi^{n+1},D\varphi^{n},D\sigma^{n})\Big]_1^2\bigg)^\frac p2}^\frac1p,
\end{equation*}
and the conclusion follows after $\e$-rescaling from the estimates of Lemma~\ref{lem:key-estimates-decomp} on the Malliavin derivatives of higher-order correctors.
\qed

\subsection{Proof of Lemma~\ref{lem:sym-baran}}
As explained in Remark~\ref{rem:migration-BS}, the argument consists in letting the successive inverse operators $(-\nabla\cdot\Aa\nabla)^{-1}\nabla\cdot$ defining the correctors $(\nabla\varphi^n,\varphi^{n-1})$ migrate to the test function $\Aa$ in the definition of the homogenized coefficient
\[\bar\Aa^n_{i_1\ldots i_{n-1}}\ee_{i_n}=\expec{\Aa\big(\nabla\varphi^n_{i_1\ldots i_n}+\varphi^{n-1}_{i_1\ldots i_{n-1}}\ee_{i_n}\big)}.\]
Taking into account the particular algebraic structure of the correctors, a full migration precisely leads to the dual expression $\bar\Aa^{*,n}$ up to additional terms involving higher-order flux correctors, which disappear when taking the symmetric part.
We split the proof into two steps.

\medskip
\step1 Migration process: proof that for all $1\le m<n\le\ell$,
\begin{multline}\label{eq:main-ident-coeff0}
\expec{\nabla\varphi^{*,m}_{j_1\ldots j_m}\cdot\Aa\nabla\varphi^n_{i_1\ldots i_n}-\varphi_{j_1\ldots j_{m-1}}^{*,m-1}\ee_{j_m}\cdot\Aa\varphi_{i_1\ldots i_{n-1}}^{n-1}\ee_{i_{n}}}\\
\,=\,-\expec{\nabla\varphi^{*,m+1}_{j_1\ldots j_{m}i_n}\cdot\Aa\nabla\varphi_{i_1\ldots i_{n-1}}^{n-1}-\varphi^{*,m}_{j_1\ldots j_m}\ee_{i_n}\cdot\Aa\varphi^{n-2}_{i_1\ldots i_{n-2}}\ee_{i_{n-1}}}\\
-\expec{\ee_{i_{n}}\cdot\sigma_{j_1\ldots j_{m-1}}^{*,m-1}\ee_{j_m}\varphi_{i_1\ldots i_{n-1}}^{n-1}+\varphi^{*,m}_{j_1\ldots j_m}\ee_{i_n}\cdot\sigma^{n-2}_{i_1\ldots i_{n-2}}\ee_{i_{n-1}}}.
\end{multline}
We abundantly use the properties of higher-order correctors (cf.~Definition~\ref{def:cor}).
The equation for $\varphi^n$, in conjunction with the stationarity of all involved objects, and the skew-symmetry of $\sigma^{n-1}_{i_1\ldots i_{n-1}}$ yield
\begin{eqnarray*}
{\expec{\nabla\varphi^{*,m}_{j_1\ldots j_m}\cdot\Aa\nabla\varphi^n_{i_1\ldots i_n}}}
&=&\expec{-\nabla\varphi^{*,m}_{j_1\ldots j_m}\cdot\big(\Aa\varphi^{n-1}_{i_1\ldots i_{n-1}}-\sigma^{n-1}_{i_1\ldots i_{n-1}}\big)\,\ee_{i_n}}\\
&=&\expec{-\nabla\varphi^{*,m}_{j_1\ldots j_m}\cdot\Aa\varphi^{n-1}_{i_1\ldots i_{n-1}}\ee_{i_n}+\varphi^{*,m}_{j_1\ldots j_m}(\nabla\cdot\sigma^{n-1}_{i_1\ldots i_{n-1}})\cdot\ee_{i_n}},
\end{eqnarray*}
hence, using the equation $\nabla\cdot\sigma^{n-1}=q^{n-1}$ and the choice $\expec{\varphi^{*,m}}=0$,
\begin{multline*}
\expec{\nabla\varphi^{*,m}_{j_1\ldots j_m}\cdot\Aa\nabla\varphi^n_{i_1\ldots i_n}}\,=\,\E\Big[-\nabla\varphi^{*,m}_{j_1\ldots j_m}\cdot\Aa\varphi^{n-1}_{i_1\ldots i_{n-1}}\ee_{i_n}+\varphi^{*,m}_{j_1\ldots j_m}\ee_{i_n}\cdot\Aa\nabla\varphi^{n-1}_{i_1\ldots i_{n-1}}\\
+\varphi^{*,m}_{j_1\ldots j_m}\ee_{i_n}\cdot\Aa\varphi^{n-2}_{i_1\ldots i_{n-2}}\ee_{i_{n-1}}-\varphi^{*,m}_{j_1\ldots j_m}\ee_{i_n}\cdot\sigma^{n-2}_{i_1\ldots i_{n-2}}\ee_{i_{n-1}}\Big].
\end{multline*}
We focus on rewriting the second right-hand side term. Using the equation for $\varphi^{*,m+1}$ and the skew-symmetry of $\sigma^{*,m}_{j_1\ldots j_m}$, we find
\begin{multline*}
{\expec{\varphi^{*,m}_{j_1\ldots j_{m}}\ee_{i_{n}}\cdot\Aa\nabla\varphi_{i_1\ldots i_{n-1}}^{n-1}}\,=\,\expec{-\nabla\varphi^{*,m+1}_{j_1\ldots j_{m}i_n}\cdot\Aa\nabla\varphi_{i_1\ldots i_{n-1}}^{n-1}-\ee_{i_{n}}\cdot\sigma^{*,m}_{j_1\ldots j_{m}}\nabla\varphi_{i_1\ldots i_{n-1}}^{n-1}}}\\
\,=\,\expec{-\nabla\varphi^{*,m+1}_{j_1\ldots j_{m}i_n}\cdot\Aa\nabla\varphi_{i_1\ldots i_{n-1}}^{n-1}+(\nabla\cdot\sigma^{*,m}_{j_1\ldots j_{m}})\cdot\varphi_{i_1\ldots i_{n-1}}^{n-1}\ee_{i_{n}}},
\end{multline*}
hence, using the equation for $\nabla\cdot\sigma^{*,m}$ and the choice $\expec{\varphi^{n-1}}=0$,
\begin{multline*}
\expec{\varphi^{*,m}_{j_1\ldots j_{m}}\ee_{i_{n}}\cdot\Aa\nabla\varphi_{i_1\ldots i_{n-1}}^{n-1}}
\,=\,\E\Big[-\nabla\varphi^{*,m+1}_{j_1\ldots j_{m}i_n}\cdot\Aa\nabla\varphi_{i_1\ldots i_{n-1}}^{n-1}+\nabla\varphi_{j_1\ldots j_m}^{*,m}\cdot\Aa\varphi_{i_1\ldots i_{n-1}}^{n-1}\ee_{i_{n}}\\
+\varphi_{j_1\ldots j_{m-1}}^{*,m-1}\ee_{j_m}\cdot\Aa\varphi_{i_1\ldots i_{n-1}}^{n-1}\ee_{i_{n}}-\sigma_{j_1\ldots j_{m-1}}^{*,m-1}\ee_{j_m}\cdot\varphi_{i_1\ldots i_{n-1}}^{n-1}\ee_{i_{n}}\Big],
\end{multline*}
and the claim~\eqref{eq:main-ident-coeff0} follows.

\medskip
\step2 Conclusion.\\
For $1\le n\le\ell$, the definition of $\bar\Aa^n$ and the equation for $\varphi^{*,1}$ yield
\[\ee_j\cdot\bar\Aa^n_{i_1\ldots i_{n-1}}\ee_{i_n}\,=\,-\expec{\nabla\varphi_j^{*,1}\cdot\Aa\nabla\varphi_{i_1\ldots i_n}^n-\ee_j\cdot\Aa\varphi^{n-1}_{i_1\ldots i_{n-1}}\ee_{i_n}}.\]
Iterating identity~\eqref{eq:main-ident-coeff0} then leads to
\begin{multline*}
\ee_j\cdot\bar\Aa^n_{i_1\ldots i_{n-1}}\ee_{i_n}\,=\,(-1)^n\,\expec{\nabla\varphi^{*,n}_{ji_n\ldots i_2}\cdot\Aa\nabla\varphi^1_{i_1}-\varphi^{*,n-1}_{ji_n\ldots i_3}\ee_{i_2}\cdot\Aa\ee_{i_1}}\\
+\sum_{l=0}^{n-2}(-1)^l\,\expec{\ee_{i_{n-l}}\cdot\sigma^{*,l}_{ji_n\ldots i_{n-l+2}}\ee_{i_{n-l+1}}\varphi^{n-l-1}_{i_1\ldots i_{n-l-1}}+\varphi^{*,l+1}_{ji_n\ldots i_{n-l+1}}\ee_{i_{n-l}}\cdot\sigma_{i_1\ldots i_{n-l-2}}^{n-l-2}\ee_{i_{n-l-1}}}.
\end{multline*}
By skew-symmetry of $\sigma^{*,l}_{ji_n\ldots i_{n-l+2}}$ and $\sigma^{n-l-2}_{i_1\ldots i_{n-l-2}}$ for all $l$, we deduce
\begin{align*}
\sym_{i_1\ldots i_n}\ee_j\cdot\bar\Aa^n_{i_1\ldots i_{n-1}}\ee_{i_n}\,=\,(-1)^n\,\sym_{i_1\ldots i_n}\expec{\nabla\varphi^{*,n}_{ji_n\ldots i_2}\cdot\Aa\nabla\varphi^1_{i_1}-\varphi^{*,n-1}_{ji_n\ldots i_3}\ee_{i_2}\cdot\Aa\ee_{i_1}},
\end{align*}
hence, using the equation for $\varphi^{1}$ and the definition of $\bar\Aa^{*,n}$,
\begin{eqnarray*}
\sym_{i_1\ldots i_n}\ee_j\cdot\bar\Aa^n_{i_1\ldots i_{n-1}}\ee_{i_n}&=&(-1)^{n+1}\sym_{i_1\ldots i_n}\expec{\ee_{i_1}\cdot\Aa^*\big(\nabla\varphi^{*,n}_{ji_n\ldots i_2}+\varphi^{*,n-1}_{ji_n\ldots i_3}\ee_{i_2}\big)}\\
&=&(-1)^{n+1}\sym_{i_1\ldots i_n}\ee_{i_1}\cdot\bar\Aa^{*,n}_{ji_n\ldots i_3}\ee_{i_2},
\end{eqnarray*}
and the conclusion follows.\qed

\subsection{Proof of Proposition~\ref{prop:homog-err}}
We focus on the case $\e=1$ and drop it from all subscripts in the notation, while the final result is obtained after $\e$-rescaling.
We split the proof into two steps.

\nopagebreak\medskip
\step1 Equation for $F_n[\bar w]$: proof that for all $\bar w\in C^\infty_c(\R^d)$ and $n\ge0$,
\begin{multline}\label{eq:pre-dec-homog}
\nabla\cdot\Aa\nabla F^n[\bar w]
\,=\,\nabla\cdot\bigg(\sum_{k=1}^{n}\bar\Aa^k_{i_1\ldots i_{k-1}}\nabla\nabla^{k-1}_{i_1\ldots i_{k-1}}\bar w\bigg)\\
+\nabla\cdot\big((\Aa\varphi_{i_1\ldots i_n}^n-\sigma_{i_1\ldots i_n}^n)\nabla\nabla^n_{i_1\ldots i_n}\bar w\big).
\end{multline}
We argue by induction. The claim is obvious for $n=0$. Now, if it holds for some~$n\ge0$, we deduce
\begin{multline}\label{eq:pre-dec-homog-induc}
\nabla\cdot\Aa\nabla F^{n+1}[\bar w]
\,=\,\nabla\cdot\bigg(\sum_{k=1}^{n}\bar\Aa^k_{i_1\ldots i_{k-1}}\nabla\nabla^{k-1}_{i_1\ldots i_{k-1}}\bar w\bigg)\\
+\nabla\cdot\big((\Aa\varphi_{i_1\ldots i_n}^n-\sigma_{i_1\ldots i_n}^n)\nabla\nabla^n_{i_1\ldots i_n}\bar w\big)+\nabla\cdot\Aa\nabla\big(\varphi_{i_1\ldots i_{n+1}}^{n+1}\nabla_{i_1\ldots i_{n+1}}^{n+1}\bar w\big).
\end{multline}
The definition of $\sigma_{i_1\ldots i_{n+1}}^{n+1}$ (cf.~Definition~\ref{def:cor}) yields
\begin{multline*}
\nabla\cdot\big((\Aa\varphi_{i_1\ldots i_n}^n-\sigma_{i_1\ldots i_n}^n)\nabla\nabla^n_{i_1\ldots i_n}\bar w\big)
\,=\,\nabla\cdot\big((\nabla\cdot\sigma_{i_1\ldots i_{n+1}}^{n+1}) \nabla^{n+1}_{i_1\ldots i_{n+1}}\bar w\big)\\
-\nabla\cdot\big(\Aa\nabla\varphi_{i_1\ldots i_{n+1}}^{n+1}\nabla^{n+1}_{i_1\ldots i_{n+1}}\bar w\big)+\nabla\cdot\big(\bar\Aa_{i_1\ldots i_n}^{n+1}\nabla\nabla^{n}_{i_1\ldots i_{n}}\bar w\big).
\end{multline*}
Hence, using the skew-symmetry of $\sigma_{i_1\ldots i_{n+1}}^{n+1}$ and decomposing
\[\nabla\varphi_{i_1\ldots i_{n+1}}^{n+1}\nabla^{n+1}_{i_1\ldots i_{n+1}}\bar w=\nabla(\varphi_{i_1\ldots i_{n+1}}^{n+1}\nabla^{n+1}_{i_1\ldots i_{n+1}}\bar w)-\varphi_{i_1\ldots i_{n+1}}^{n+1}\nabla\nabla^{n+1}_{i_1\ldots i_{n+1}}\bar w,\]
we obtain
\begin{multline*}
\nabla\cdot\big((\Aa\varphi_{i_1\ldots i_n}^n-\sigma_{i_1\ldots i_n}^n)\nabla\nabla^n_{i_1\ldots i_n}\bar w\big)
\,=\,\nabla\cdot\big((\Aa\varphi_{i_1\ldots i_{n+1}}^{n+1}-\sigma_{i_1\ldots i_{n+1}}^{n+1})\nabla\nabla^{n+1}_{i_1\ldots i_{n+1}}\bar w\big)\\
-\nabla\cdot\Aa\nabla\big(\varphi_{i_1\ldots i_{n+1}}^{n+1}\nabla^{n+1}_{i_1\ldots i_{n+1}}\bar w\big)+\nabla\cdot\big(\bar\Aa_{i_1\ldots i_n}^{n+1}\nabla\nabla^{n}_{i_1\ldots i_{n}}\bar w\big).
\end{multline*}
Injecting this into~\eqref{eq:pre-dec-homog-induc} leads to the claim~\eqref{eq:pre-dec-homog} at level $n+1$.

\medskip
\step2 Conclusion.\\
Let $n\ge1$ be fixed.
Applying~\eqref{eq:pre-dec-homog} with $\bar w:=\bar u^{n}$ and subtracting the equation~\eqref{eq:ueps-veps} for $u$, we obtain
\begin{multline}\label{eq:pre-concl-dec-homog}
-\nabla\cdot\Aa\nabla\big(u-F^n[\bar u^{n}]\big)\\
\,=\,\nabla\cdot\bigg(f+\sum_{k=1}^{n}\bar\Aa^k_{i_1\ldots i_{k-1}}\nabla\nabla^{k-1}_{i_1\ldots i_{k-1}}\bar u^{n}\bigg)+\nabla\cdot\big((\Aa\varphi_{i_1\ldots i_n}^n-\sigma_{i_1\ldots i_n}^n)\nabla\nabla^n_{i_1\ldots i_n}\bar u^{n}\big).
\end{multline}
We now examine the equation satisfied by $\bar u^n=\sum_{k=1}^n\tilde u^k$: summing the defining equations for $(\tilde u^k)_{1\le k\le n}$ (cf.~Defintion~\ref{def:homog-eqn}), we find
\[-\nabla\cdot\bar\Aa^1\nabla\bar u^n=\nabla\cdot f+\nabla\cdot\Big(\sum_{k=2}^n\bar\Aa^k_{i_1\ldots i_{k-1}}\nabla\nabla^{k-1}_{i_1\ldots i_{k-1}}\sum_{l=k}^n\tilde u^{l+1-k}\Big),\]
or equivalently,
\begin{align*}
-\nabla\cdot\Big(\sum_{k=1}^n\bar\Aa^k_{i_1\ldots i_{k-1}}\nabla^{k-1}_{i_1\ldots i_{k-1}}\Big)\nabla\bar u^n=\nabla\cdot f
-\nabla\cdot\Big(\sum_{k=2}^n\bar\Aa^k_{i_1\ldots i_{k-1}}\nabla\nabla^{k-1}_{i_1\ldots i_{k-1}}\sum_{l=n+2-k}^{n}\tilde u^{l}\Big).
\end{align*}
(This is to be compared with the formal ill-posed equation~\eqref{eq:eff-eqn-formal}.)
Inserting this into~\eqref{eq:pre-concl-dec-homog} yields
\begin{multline}\label{eq:u-dev-scale}
-\nabla\cdot\Aa\nabla\big(u-F^n[\bar u^{n}]\big)\,=\,\nabla\cdot\Big(\sum_{k=2}^n\bar\Aa^k_{i_1\ldots i_{k-1}}\nabla\nabla^{k-1}_{i_1\ldots i_{k-1}}\sum_{l=n+2-k}^{n}\tilde u^{l}\Big)\\
+\nabla\cdot\big((\Aa\varphi_{i_1\ldots i_n}^n-\sigma_{i_1\ldots i_n}^n)\nabla\nabla^n_{i_1\ldots i_n}\bar u^{n}\big).
\end{multline}
An energy estimate and the corrector estimates of Proposition~\ref{prop:cor} then imply for all $p<\infty$,
\begin{eqnarray*}
\expec{\big\|\nabla\big(u-F^n[\bar u^{n}]\big)\big\|_{\Ld^2(\R^d)}^p}^\frac1p&\lesssim_p&\|\mu_{d,n}\nabla^{n+1}\bar u^{n}\|_{\Ld^2(\R^d)}+\sum_{k=2}^{n}~\sum_{l=n+2-k}^n\|\nabla^{k}\tilde u^{l}\|_{\Ld^2(\R^d)}\\
&\lesssim_p&\sum_{k=n}^{2n-1}\|\mu_{d,n}\nabla^{k}f\|_{\Ld^2(\R^d)},
\end{eqnarray*}
where the last estimate follows from the weighted Calder\'on-Zygmund theory for the constant-coefficient equations defining $(\tilde u^k)_{1\le k\le n}$.
The conclusion~\eqref{eq:concl-homog-err} in $\Ld^2(\R^d)$ follows after $\e$-rescaling. (Note that for the sake of shortness in the statement the norms of $\nabla^kf$ for $n\le k\le2n-1$ are replaced by the norm of $\langle\nabla\rangle^{2n-1}f$, which is no longer scale invariant.)
The corresponding result in $\Ld^q(\R^d)$ for general $1<q<\infty$ requires to replace the energy estimate for~\eqref{eq:u-dev-scale} by some $\Ld^q$ theory, as provided e.g.\@ by Theorem~\ref{th:CZ-ann} below; details are omitted.
\qed

\section{Higher-order homogenization commutators}\label{sec:commut}

As discovered in~\cite{AKM2,DGO1}, the homogenization commutator $\Xi_\e^1[\nabla u_\e]:=(\Aa(\tfrac\cdot\e)-\bar\Aa^1)\nabla u_\e$, which takes the form of a commutator between large-scale averaging and the field-flux constitutive relation, plays a key role in the study of fluctuations in stochastic homogenization.
Indeed, while the two-scale expansion of the solution $\nabla u_\e$ is not accurate in the fluctuation scaling~\cite{GuM}, the expansion of its commutator $\Xi_\e^1[\nabla u_\e]$ is accurate~\cite{DGO1}.
The leading-order fluctuations are then governed by the standard homogenization commutator, that is, the commutator of the corrector field, $\Xi_\e^{\circ,1}[\nabla\bar u^1]:=(\Aa(\tfrac\cdot\e)-\bar\Aa^1)(\nabla\varphi_i^1(\tfrac\cdot\e)+\ee_i)\nabla_i\bar u^1$.
For higher-order fluctuations, a suitable higher-order correction of the homogenization commutator $\Xi_\e^1[\nabla u_\e]$ is defined as follows.

\begin{defin}[Higher-order homogenization commutators]\label{def:Xin}
For $0\le n\le\ell$, given a random function $w$, its {$n$th-order homogenization commutator} $\Xi^n_\e[\nabla w]$ is the (distributional) random vector field
\[\Xi_{\e}^n[\nabla w]\,:=\,\Big(\Aa(\tfrac\cdot\e)-\sum_{k=1}^{n}\e^{k-1}\bar\Aa^k_{i_1\ldots i_{k-1}}\nabla^{k-1}_{i_1\ldots i_{k-1}}\Big)\nabla w.\qedhere\]
\end{defin}

This definition is natural in view of the higher-order effective field-flux constitutive relation~\eqref{eq:higher-constitutive-rel}.
Note that the symmetries of the homogenized coefficients (cf.~Lemma~\ref{lem:sym-baran}) and an integration by parts lead to the following duality relation for all smooth and compactly supported random functions $w,w'$,
\begin{align}\label{eq:sym-Xin}
\int_{\R^d}\nabla w'\cdot\Xi^n_\e[\nabla w]\,=\,\int_{\R^d}\Xi^{*,n}_\e[\nabla w']\cdot\nabla w.
\end{align}
Similarly as for the first-order commutator~\cite[(1.8)]{DGO1}, we state that the higher-order fluctuations of the field $\nabla u_\e$ are determined by those of the higher-order commutators. While at first order we have the exact relation $\int_{\R^d}g\cdot(\nabla u_\e-\nabla\bar u)=\int_{\R^d}\nabla\bar v\cdot\Xi_\e^1[\nabla u_\e]$, this does no longer hold at higher orders due to the choice of a proxy for solutions of the generically ill-posed higher-order homogenized equations~\eqref{eq:eff-eqn-formal} (cf.~\eqref{eq:reduction-commut} below).
A corresponding formula can be deduced for the flux $\Aa(\tfrac\cdot\e)\nabla u_\e$.
This key principle corresponds to item~(i) in Theorem~\ref{th:main}.

\begin{prop}[Reduction to commutators]\label{prop:reduc-nablau}
For all $1\le n\le\ell$ and $p<\infty$,
\[N_p\bigg(\e^{-\frac d2}\int_{\R^d}g\cdot\nabla u_\e-\e^{-\frac d2}\int_{\R^d}\nabla\bar v_\e^{n}\cdot\Xi_\e^n[\nabla u_\e]\bigg)
\,\lesssim_p\,\e^{n}\|\langle\nabla\rangle^{2(n-1)}g\|_{\Ld^4(\R^d)}\|f\|_{\Ld^4(\R^d)}.\]
where we recall the notation $N_p(X)=\expec{|X-\expec{X}|^p}^\frac1p$.
\end{prop}

Next, we wish to emphasize that the above definition of the higher-order commutator~$\Xi_\e^n$ leads to a higher-order Hill-Mandel relation (e.g.~\cite{TMB-12} in the mechanics literature).
First recall that the first-order commutator $\Xi^1_{\e}[\nabla u_\e]:=(\Aa-\bar\Aa)(\tfrac\cdot\e)\nabla u_\e$ was a natural quantity to consider in~\cite{AKM2,DGO1} since the very definition of $H$-convergence~\cite{MuratTartar} is equivalent to the weak convergence of $\Xi^1_\e[\nabla u_\e]$ to $0$, which is the mathematical formulation of the classical Hill-Mandel relation~\cite{H63,H72}.
Interestingly, this can be made quantitative: a direct computation for the solution of~\eqref{eq:ueps-veps} yields
\[\ee_j\cdot\Xi^1_{\e}[\nabla u_\e]
\,=\,-\,\e\,\nabla_k\Big(\big(\Aa^*\varphi_j^{*,1}-\sigma_j^{*,1}\big)(\tfrac\cdot\e)\,\ee_k\cdot\nabla u_\e+\varphi^{*,1}_j(\tfrac\cdot\e) f_k\Big)+\nabla\varphi^{*,1}_j(\tfrac\cdot\e)\cdot f,\]
where the right-hand side is checked to be of order $O(\e)$ in a weak norm if $d>2$.
This property extends to higher order as follows.

\begin{lem}[Higher-order Hill-Mandel relation]\label{lem:higher-Hill}
For all $f\in C^\infty_c(\R^d)^d$ and $1\le n\le\ell$,
\begin{multline}\label{eq:higher-Hill}
\ee_j\cdot\Xi_\e^n[\nabla u_\e]
\,=\,(-\e)^n\nabla^n_{i_1\ldots i_n}\Big(\big(\Aa^*\varphi^{*,n}_{ji_1\ldots i_{n-1}}-\sigma^{*,n}_{ji_1\ldots i_{n-1}}\big)(\tfrac\cdot\e)\ee_{i_n}\cdot\nabla u_\e+\varphi_{ji_1\ldots i_{n-1}}^{*,n}(\tfrac\cdot\e)\,f_{i_n}\Big)\\
+\sum_{k=0}^{n-1}(-\e)^{k}\nabla^{k}_{i_1\ldots i_{k}}\Big(\big(\nabla\varphi_{ji_1\ldots i_{k}}^{*,k+1}+\mathds1_{k>0}\varphi_{ji_1\ldots i_{k-1}}^{*,k}\ee_{i_k}\big)(\tfrac\cdot\e)\,\cdot f\Big),
\end{multline}
hence, for all $g\in C^\infty_c(\R^d)^d$ and $p<\infty$,
\[N_p\Big(\int_{\R^d}g\cdot\Xi_{\e}^n[\nabla u_\e]\Big)\,\lesssim_p\,\e^n\mu_{d,n}(\tfrac1\e)\Big(\|\mu_{d,n}\nabla^ng\|_{\Ld^2(\R^d)}+\|g\|_{H^\frac d2\cap\Ld^\infty(\R^d)}\Big)\|[f]_{2d}\|_{\Ld^2(\R^d)}.\qedhere\]
\end{lem}

Finally, the relevant higher-order two-scale expansion of the higher-order commutator $\Xi^n_\e[\nabla u_\e]$ takes the form $\Xi^{\circ,n}_\e[\nabla\bar u^n_\e]$, where the so-called standard commutator $\Xi^{\circ,n}_\e[\nabla\cdot]$ is the $n$th-order differential operator with $\e$-rescaled (distributional) stationary random coefficients that is obtained by inserting the $n$th-order two-scale expansion $F_\e^n[\cdot]$ into the commutator $\Xi^n_\e[\nabla\cdot]$ and by truncating the obtained differential operator at order $n$. In other words, the standard commutator $\Xi^{\circ,n}_\e[\nabla\cdot]$ is the $n$th-order linear differential operator characterized by $\Xi^{\circ,n}_\e[\nabla\bar q]=\Xi^{n}_\e[\nabla F_\e^n[\bar q]]$ for all $n$th-order polynomials $\bar q$. This is efficiently expressed as follows.

\begin{defin}[Standard higher-order homogenization commutators]\label{def:Xin0}
For $0\le n\le\ell$, given a smooth deterministic function $\bar w$, we consider its $n$th-order Taylor polynomial with basepoint $x$,
\[T^n_{x}\bar w(y)\,:=\,\bar w(x)+\sum_{1\le|\alpha|\le n}\frac{(y-x)^\alpha}{\alpha!}\nabla^\alpha\bar w(x),\]
where we use standard multi-index notation,
and the {$n$th-order standard homogenization commutator} $\Xi^{\circ,n}_\e[\nabla\bar w]$ is then defined as the (distributional) random vector field
\[\Xi_{\e}^{\circ,n}[\nabla\bar w](x)\,:=\,\Xi_{\e}^{n}\big[\nabla F^n_\e[T^n_x\bar w]\big](x)\,\stackrel{\eqref{eq:link-E0E}}=\,\Xi_{\e}^{n}\big[E^n_\e[\nabla T^n_x\bar w]\big](x).\]
(Henceforth by a slight abuse of notation we similarly define $\Xi^n[H]$ as in Definition~\ref{def:Xin} even when $H$ is not a gradient field.)
\end{defin}

This definition ensures $\expec{\Xi_\e^{\circ,n}[\nabla\bar w]}=0$ for all $0\le n\le\ell$ and $\bar w\in C^\infty_c(\R^d)$.
Note that $\Xi_{\e}^{\circ,n}[\nabla\bar w]$ in general differs from $\Xi_{\e}^{n}[E_\e^{n}[\nabla\bar w]]$ whenever $n>1$, even though $E^n_\e[\nabla T^n_x\bar w](x)=E^n_\e[\nabla\bar w](x)$. An explicit formula is as follows.

\begin{lem}[Explicit formula for $\Xi^{\circ,n}_\e$]\label{lem:expl-form-Xi0}
For all $\bar w\in C^\infty_c(\R^d)$ and $1\le n\le\ell$,
\begin{align*}
&\Xi_{\e}^{\circ,n}[\nabla\bar w]\,=\,\Xi_{\e}^n[E^n_\e[\nabla\bar w]]\\
&\hspace{2cm}+\sum_{k=1}^{n-1}\e^{k}\big(\sym_{i_1\ldots i_{k}}\bar\Aa^{k+1}_{i_1\ldots i_{k}}\big)\sum_{s=0}^{n-1}\e^s\,\sum_{l=0}^{k+s-n}\binom{k}l\nabla^{k-l}_{i_{l+1}\ldots i_k}\nabla_{j_1\ldots j_{s+1}}^{s+1}\bar w\\
&\hspace{6cm}\times\nabla^l_{i_{1}\ldots i_l}\Big(\big(\nabla\varphi_{j_1\ldots j_{s+1}}^{s+1}+\varphi_{j_1\ldots j_s}^s\ee_{j_{s+1}}\big)(\tfrac\cdot\e)\Big).\qedhere
\end{align*}
\end{lem}

\begin{rem}\label{rem:fluc-cor+}
While Proposition~\ref{prop:reduc-nablau} reduces higher-order fluctuations of the field $\nabla u_\e$ to corresponding fluctuations of its higher-order homogenization commutator, we briefly argue that fluctuations of higher-order correctors can similarly be reduced to fluctuations of higher-order standard commutators.
A simple adaptation of the proof of Proposition~\ref{prop:reduc-nablau} yields for $0\le n\le\ell-1$,
\[N_p\bigg(\e^{-\frac d2}\int_{\R^d}g\cdot E_\e^n[\nabla\bar w]-\e^{-\frac d2}\int_{\R^d}\nabla\bar v_\e^{n}\cdot\Xi_\e^n[E_\e^n[\nabla\bar w]]\bigg)
\,\lesssim_{p,f,g}\,\e^{n}.\]
Inserting the formula of Lemma~\eqref{lem:expl-form-Xi0} and the definition of $E_\e^n$ (cf.~Definition~\ref{def:EnFn}), the expression on the left-hand side can, after straightforward computations, be rewritten as follows,
\begin{align}\label{eq:split-E-Xi}
&\quad\e^{-\frac d2}\int_{\R^d}g\cdot E_\e^n[\nabla\bar w]-\e^{-\frac d2}\int_{\R^d}\nabla\bar v_\e^{n}\cdot\Xi_\e^n[E_\e^n[\nabla\bar w]]\\
&=\sum_{s=0}^{n-1}\e^{s-\frac d2}\int_{\R^d}T_{\e;j_1\ldots j_{s+1}}^{n,s}(f,g)\cdot \big(\nabla\varphi^{s+1}_{j_1\ldots j_{s+1}}+\varphi^s_{j_1\ldots j_s}\ee_{j_{s+1}}\big)(\tfrac\cdot\e)
-\e^{-\frac d2}\int_{\R^d}\nabla\bar v_\e^{n}\cdot\Xi_{\e}^{\circ,n}[\nabla\bar w],\nonumber
\end{align}
in terms of
\begin{multline*}
T_{\e;j_1\ldots j_{s+1}}^{n,s}(f,g)\,:=\,g\,\nabla^{s+1}_{j_1\ldots j_{s+1}}\bar w+\sum_{k=1}^{n-1}\e^{k}(-1)^{k+s-n}\,\sum_{r=0}^{k+s-n}\binom{k}{r}\binom{k-1-r}{k+s-n-r}\\
\times\big(\sym_{i_1\ldots i_{k}}\bar\Aa^{k+1}_{i_1\ldots i_{k}}\big)^*\big(\nabla\nabla^r_{i_{1}\ldots i_r}\bar v_\e^{n}\big)\big(\nabla^{k-r}_{i_{r+1}\ldots i_k}\nabla_{j_1\ldots j_{s+1}}^{s+1}\bar w\big).
\end{multline*}
If we wish to describe only first-order fluctuations, we may replace $T_\e^{n,s}(f,g)$ by $g\nabla^{s+1}\bar w$ and $\nabla\bar v_\e^n$ by $\nabla\bar v$, and use Lemma~\ref{lem:average-cor} to estimate the remainder, to the effect of
\[N_p\bigg(\e^{-\frac d2}\int_{\R^d}g\cdot E_\e^n[\nabla\bar w]-\e^{-\frac d2}\int_{\R^d}\nabla\bar v\cdot\Xi_{\e}^{\circ,n}[\nabla\bar w]\bigg)
\,\lesssim_{p,g,\bar w}\,\e,\]
hence, recalling Definition~\ref{def:EnFn} and taking differences, for all $0\le n\le\ell-1$,
\begin{multline*}
N_p\bigg(\e^{n-1-\frac d2}\int_{\R^d}g\,\nabla^n_{j_1\ldots j_n}\bar w\cdot\big(\nabla\varphi^{n}_{j_1\ldots j_n}+\varphi^{n-1}_{j_1\ldots j_{n-1}}\ee_{j_n}\big)(\tfrac\cdot\e)\\
-\e^{-\frac d2}\int_{\R^d}\nabla\bar v\cdot\big(\Xi_{\e}^{\circ,n}[\nabla\bar w]-\Xi_{\e}^{\circ,n-1}[\nabla\bar w]\big)\bigg)
\,\lesssim_{p,g,\bar w}\,\e.
\end{multline*}
This characterizes the first-order fluctuations of $\sym_{j_1\ldots j_n}(\nabla\varphi^{n}_{j_1\ldots j_n}+\varphi^{n-1}_{j_1\ldots j_{n-1}}\ee_{j_n})$ in terms of (differences of) standard commutators. Note that fluctuations of the latter are easily jointly characterized by repeating the analysis of Sections~\ref{sec:cov-conv}--\ref{sec:normal}, and are seen to converge to suitable linear combinations of derivatives of Gaussian white noise.
From here we may inductively infer first-order fluctuations of $\sym_{j_1\ldots j_n}\varphi^{n}_{j_1\ldots j_n}$ for all $n\le\ell-1$, while corresponding higher-order fluctuations are similarly extracted from~\eqref{eq:split-E-Xi}.
\end{rem}

\subsection{Proof of Proposition~\ref{prop:reduc-nablau}}
We focus on the case $\e=1$ and drop it from all subscripts in the notation, while the final result is obtained after $\e$-rescaling.
We split the proof into three steps.

\medskip
\step1 Proof of
\begin{multline}\label{eq:reduction-commut}
\int_{\R^d}g\cdot\nabla u-\int_{\R^d}\nabla\bar v^{n}\cdot f\,=\,\int_{\R^d}\nabla\bar v^{n}\cdot \Xi^n[\nabla u]\\
+\int_{\R^d}\bigg(\sum_{k=2}^{n}\sum_{j=n+2-k}^n\bar\Aa_{i_1\ldots i_{k-1}}^{*,k}\nabla\nabla^{k-1}_{i_1\ldots i_{k-1}}\tilde v^{j}\bigg)\cdot\nabla u.
\end{multline}
Testing equation $-\nabla\cdot\bar\Aa^{*,1}\nabla\bar v^1=\nabla\cdot g$ with $u$ and testing equation $-\nabla\cdot\Aa\nabla u=\nabla\cdot f$ with~$\bar v^{n}$, we find
\begin{eqnarray*}
\int_{\R^d}g\cdot\nabla u-\int_{\R^d}\nabla\bar v^{n}\cdot f\,=\,\int_{\R^d}\big(\Aa^*\nabla\bar v^{n}-\bar\Aa^{*,1}\nabla\bar v^1\big)\cdot \nabla u.
\end{eqnarray*}
Noting that the definition of $\Xi^n[\nabla u]$ (cf.~Definition~\ref{def:Xin}) together with~\eqref{eq:sym-Xin} yields after integration by parts
\[\int_{\R^d}\nabla\bar v^{n}\cdot \Xi^n[\nabla u]\,=\,\int_{\R^d}\bigg(\Aa^*\nabla\bar v^{n}-\sum_{k=1}^{n}\bar\Aa_{i_1\ldots i_{k-1}}^{*,k}\nabla\nabla^{k-1}_{i_1\ldots i_{k-1}}\bar v^{n}\bigg)\cdot\nabla u,\]
we deduce
\begin{multline*}
\int_{\R^d}g\cdot\nabla u-\int_{\R^d}\nabla\bar v^{n}\cdot f\,=\,\int_{\R^d}\nabla\bar v^{n}\cdot \Xi^n[\nabla u]\\
+\int_{\R^d}\bigg(\sum_{k=1}^{n}\bar\Aa_{i_1\ldots i_{k-1}}^{*,k}\nabla\nabla^{k-1}_{i_1\ldots i_{k-1}}\bar v^{n}-\bar\Aa^{*,1}\nabla\bar v^1\bigg)\cdot\nabla u.
\end{multline*}
Decomposing $\bar v^{n}:=\sum_{j=1}^{n}\tilde v^j$ and using the defining equations for $(\tilde v^j)_{2\le j\le n}$ (cf.~Definition~\ref{def:homog-eqn}) in the resummed form
\[\nabla\cdot\sum_{k=1}^n\sum_{j=1}^{n+1-k}\bar\Aa^{*,k}_{i_1\ldots i_{k-1}}\nabla\nabla^{k-1}_{i_1\ldots i_{k-1}}\tilde v^{j}=\nabla\cdot\bar\Aa^{*,1}\nabla\bar v^1,\]
the claim follows.

\medskip
\step2 CLT scaling: for all $g\in C^\infty_c(\R^d)^d$,
\[\expec{\Big(\int_{\R^d}g\cdot(\nabla u-\expec{\nabla u})\Big)^p}^\frac1p\,\lesssim_p\,\|[g]_2\|_{\Ld^4(\R^d)}\|[f]_2\|_{\Ld^4(\R^d)}\,\le\,\|g\|_{\Ld^4(\R^d)}\|f\|_{\Ld^4(\R^d)}.\]
This result is classical~\cite{MaO,GNO-quant,DGO2}. A short proof based on Malliavin calculus and on the annealed Calder\'on-Zygmund theory of Section~\ref{sec:Lpreg} is easily obtained as in~\eqref{eq:CLT-scaling-pr}; details are omitted.

\medskip
\step3 Conclusion.\\
Since $\int_{\R^d}\nabla\bar v^n\cdot f$ is deterministic, the identity of Step~1 yields
\begin{multline*}
\expec{\Big(\int_{\R^d}g\cdot\big(\nabla u-\expec{\nabla u}\big)-\int_{\R^d}\nabla\bar v^{n}\cdot\big(\Xi^n[\nabla u]-\expec{\Xi^n[\nabla u]}\big)\Big)^p}^\frac1p\\
\,\lesssim\,\sum_{k=2}^{n}\sum_{j=n+2-k}^n\expec{\Big(\int_{\R^d}\bar\Aa_{i_1\ldots i_{k-1}}^{*,k}\nabla\nabla^{k-1}_{i_1\ldots i_{k-1}}\tilde v^{j}\cdot(\nabla u-\expec{\nabla u})\Big)^p}^\frac1p.
\end{multline*}
Applying the CLT scaling result of Step~2 and the Calder\'on-Zygmund theory for the constant-coefficient equations defining $(\tilde v^k)_{1\le k\le n}$, this expression is estimated by
\[\lesssim_p\,\sum_{k=n}^{2(n-1)}\|\nabla^kg\|_{\Ld^4(\R^d)}\|f\|_{\Ld^4(\R^d)},\]
and the conclusion follows after $\e$-rescaling.
\qed

\subsection{Proof of Lemma~\ref{lem:higher-Hill}}
We focus on the case $\e=1$ and drop it from all subscripts in the notation, while the final result is obtained after $\e$-rescaling.
We split the proof into two steps.

\medskip
\step1 Proof that for all $1\le n\le\ell$,
\begin{multline}\label{eq:rewrite-Xieps-en}
\ee_j\cdot\Xi^n[\nabla u]
\,=\,(-1)^n\nabla^n_{i_1\ldots i_n}\Big(\big(\Aa^*\varphi^{*,n}_{ji_1\ldots i_{n-1}}-\sigma^{*,n}_{ji_1\ldots i_{n-1}}\big)\ee_{i_n}\cdot\nabla u+\varphi_{ji_1\ldots i_{n-1}}^{*,n}\,f_{i_n}\Big)\\
+\sum_{k=0}^{n-1}(-1)^{k}\nabla^{k}_{i_1\ldots i_{k}}\Big(\big(\nabla\varphi_{ji_1\ldots i_{k}}^{*,k+1}+\mathds1_{k>0}\varphi_{ji_1\ldots i_{k-1}}^{*,k}\ee_{i_k}\big)\,\cdot f\Big),
\end{multline}
which coincides with~\eqref{eq:higher-Hill} after $\e$-rescaling.

\noindent
First, by Lemma~\ref{lem:sym-baran}, the $n$th-order commutator can alternatively be written as
\begin{align}\label{eq:redef-Xi}
\ee_j\cdot\Xi^n[\nabla u]\,=\,\Aa^*\ee_j\cdot\nabla u-\sum_{k=1}^{n}(-1)^{k-1}\nabla^{k-1}_{i_1\ldots i_{k-1}}\big(\bar\Aa^{*,k}_{ji_1\ldots i_{k-2}}\ee_{i_{k-1}}\cdot\nabla u\big).
\end{align}
It remains to show that for all $0\le n\le\ell$,
\begin{align}\label{eq:pre-dec-commut-00}
&\Aa^*\ee_j\cdot\nabla u=(-1)^n\nabla^n_{i_1\ldots i_n}\Big(\big(\Aa^*\varphi_{ji_1\ldots i_{n-1}}^{*,n}-\sigma^{*,n}_{ji_1\ldots i_{n-1}}\big)\ee_{i_n}\cdot\nabla u\Big)\\
&+\sum_{k=1}^n(-1)^{k-1}\nabla_{i_1\ldots i_{k-1}}^{k-1}\big(\varphi_{ji_1\ldots i_{k-1}}^{*,k}\nabla\cdot\Aa\nabla u\big)
+\sum_{k=1}^{n}(-1)^{k-1}\nabla^{k-1}_{i_1\ldots i_{k-1}}\big(\bar\Aa^{*,k}_{ji_1\ldots i_{k-2}}\ee_{i_{k-1}}\cdot\nabla u\big),\nonumber
\end{align}
while the result~\eqref{eq:rewrite-Xieps-en} indeed follows after injecting the equation $-\nabla\cdot\Aa\nabla u=\nabla\cdot f$ and writing $\varphi^{*,k}_{ji_1\ldots i_{k-1}}\nabla\cdot f=\nabla\cdot(\varphi^{*,k}_{ji_1\ldots i_{k-1}}f)-\nabla\varphi^{*,k}_{ji_1\ldots i_{k-1}}\cdot f$.
We prove~\eqref{eq:pre-dec-commut-00} by induction. It is obvious for $n=0$ (recall that for $n=0$ the notation $\Aa^*\varphi_{ji_1\ldots i_{n-1}}^{*,n}\ee_{i_n}$ stands for $\Aa^*\ee_j$).
Assume that it holds for some $n\ge0$. In order to deduce it at level $n+1$, it suffices to prove that
\begin{multline}\label{eq:pre-dec-commut-000}
\nabla^{n}_{i_1\ldots i_n}\Big(\big(\Aa^*\varphi_{ji_1\ldots i_{n-1}}^{*,n}-\sigma_{ji_1\ldots i_{n-1}}^{*,n}\big)\ee_{i_{n}}\cdot\nabla u\Big)
=-\nabla^{n+1}_{i_1\ldots i_nl}\Big(\big(\Aa^*\varphi_{ji_1\ldots i_n}^{*,n+1}-\sigma^{*,n+1}_{ji_1\ldots i_n}\big)\ee_l\cdot\nabla u\Big)\\
+\nabla^{n}_{i_1\ldots i_n}\big(\varphi_{ji_1\ldots i_n}^{*,n+1}\nabla\cdot\Aa\nabla u\big)
+\nabla^{n}_{i_1\ldots i_n}\big(\bar\Aa_{ji_1\ldots i_{n-1}}^{*,n+1}\ee_{i_n}\cdot\nabla u\big).
\end{multline}
The equation for $\nabla\cdot\sigma^{*,n+1}$ (cf.~Definition~\ref{def:cor}) yields
\begin{multline*}
\nabla^{n}_{i_1\ldots i_n}\Big(\big(\Aa^*\varphi_{ji_1\ldots i_{n-1}}^{*,n}-\sigma_{ji_1\ldots i_{n-1}}^{*,n}\big)\ee_{i_{n}}\cdot\nabla w\Big)=\nabla^{n}_{i_1\ldots i_n}\Big(\big(\nabla\cdot\sigma^{*,n+1}_{ji_1\ldots i_n}\big)\cdot\nabla w\Big)\\
-\nabla^{n}_{i_1\ldots i_n}\big(\nabla\varphi_{ji_1\ldots i_n}^{*,n+1}\cdot\Aa\nabla w\big)+\nabla^{n}_{i_1\ldots i_n}\big(\bar\Aa_{ji_1\ldots i_{n-1}}^{*,n+1}\ee_{i_n}\cdot\nabla w\big),
\end{multline*}
and the claim~\eqref{eq:pre-dec-commut-000} follows from the skew-symmetry of $\sigma^{*,n+1}_{ji_1\ldots i_n}$.

\medskip
\step2 Conclusion.

\noindent
For $g\in C^\infty_c(\R^d)^d$, combining~\eqref{eq:rewrite-Xieps-en} with the energy estimate $\|\nabla u\|_{\Ld^2(\R^d)}\lesssim\|f\|_{\Ld^2(\R^d)}$ and with the corrector estimates of Proposition~\ref{prop:cor}, we find for all $p<\infty$,
\begin{multline*}
\expec{\Big|\int_{\R^d}g\cdot\Xi^n[\nabla u]\Big|^p}^\frac1p
\,\lesssim_p\,\|\mu_{d,n}\nabla^ng\|_{\Ld^2(\R^d)}\|f\|_{\Ld^2(\R^d)}\\
+\sum_{k=0}^{n-1}\expec{\Big|\int_{\R^d}\big(\nabla\varphi_{ji_1\ldots i_{k}}^{*,k+1}+\mathds1_{k>0}\varphi_{ji_1\ldots i_{k-1}}^{*,k}\ee_{i_k}\big)\,\cdot f\nabla^{k}_{i_1\ldots i_{k}}g_j\Big|^p}^\frac1p.
\end{multline*}
In order to optimally estimate the last contribution, we must exploit stochastic cancellations in form of Lemma~\ref{lem:average-cor}, which yields after $\e$-rescaling, for all $1\le n\le\ell$,
\begin{multline*}
\expec{\Big|\int_{\R^d}g\cdot\Xi^n_\e[\nabla u_\e]\Big|^p}^\frac1p
\,\lesssim_p\,\e^n\mu_{d,n}(\tfrac1\e)\\
\times\bigg(\|\mu_{d,n}\nabla^ng\|_{\Ld^2(\R^d)}\|f\|_{\Ld^2(\R^d)}+\sum_{k=0}^{n-1}\|[f\nabla^{k}g]_2\|_{\Ld^{\frac{2d}{d+2k}}(\R^d)}\bigg).
\end{multline*}
Since $\ell-1\le\frac12(d-1)$, using Hölder's and Sobolev's inequalities in the form
\begin{eqnarray*}
\max_{0\le k\le\ell-1}\big\|[f\nabla^{k}g]_2\big\|_{\Ld^{\frac{2d}{d+2k}}(\R^d)}&\le&\max_{0\le k\le\ell-1}\big\|[f]_{\frac{2d}{d-2k}}\big\|_{\Ld^2(\R^d)}\|\nabla^{k}g\|_{\Ld^{\frac dk}(\R^d)}\\
&\lesssim&\|[f]_{2d}\|_{\Ld^2(\R^d)}\|g\|_{H^\frac d2\cap\Ld^\infty(\R^d)},
\end{eqnarray*}
the conclusion follows.
\qed

\subsection{Proof of Lemma~\ref{lem:expl-form-Xi0}}
First note that, given a linear differential operator $O$ of order~$m$, say $O[\bar w]=\sum_{k=0}^mP^k_{j_1\ldots j_k}(x)\nabla^k_{j_1\ldots j_k}\bar w$, the characterizing property of the Taylor polynomial yields
\begin{align}\label{eq:op-diff-of-Taylor}
O[T^n_x\bar w](x)=\sum_{k=0}^{m\wedge n}P^k_{j_1\ldots j_k}(x)\nabla_{j_1\ldots j_k}^k\bar w(x),
\end{align}
which amounts to truncating the differential operator $O$ at order $n$.
Now observe that $\Xi_{\e}^n[E_\e^n[\nabla\cdot]]$ is a linear differential operator of order $2n-1$ with (distributional) stationary random coefficients.
Composing the explicit definitions of $E_\e^n$ and $\Xi_{\e}^n$ (cf.~Definitions~\ref{def:EnFn} and~\ref{def:Xin}), we obtain
\begin{multline*}
\Xi_{\e}^n[E_\e^n[\nabla\bar w]]\,=\,(\Aa(\tfrac\cdot\e)-\bar\Aa^1)E_\e^n[\nabla\bar w]\\
-\sum_{k=2}^{n}\e^{k-1}\bar\Aa_{i_1\ldots i_{k-1}}^{k}\nabla^{k-1}_{i_1\ldots i_{k-1}}\sum_{s=0}^{n-1}\e^{s}\big(\nabla\varphi_{j_1\ldots j_{s+1}}^{s+1}+\varphi_{j_1\ldots j_s}^s\ee_{j_{s+1}}\big)(\tfrac\cdot\e)\,\nabla_{j_1\ldots j_{s+1}}^{s+1}\bar w.
\end{multline*}
Hence, relabelling $k-1$ as $k$ in the second right-hand side contribution, using the symmetry in $i_1,\ldots,i_{k}$, and expanding the multiple derivative $\nabla^{k}_{i_1\ldots i_{k}}$ by using the general Leibniz rule,
\begin{multline}\label{eq:form-Xi-E-comp}
\Xi_{\e}^n[E_\e^n[\nabla\bar w]]\,=\,(\Aa(\tfrac\cdot\e)-\bar\Aa^1) E_\e^n[\nabla\bar w]\\
-\sum_{k=1}^{n-1}\e^{k}\big(\sym_{i_1\ldots i_k}\bar\Aa_{i_1\ldots i_{k}}^{k+1}\big)\sum_{s=0}^{n}\e^{s}\sum_{l=0}^k\binom{k}{l}\nabla^{k-l}_{i_{l+1}\ldots i_{k}}\nabla_{j_1\ldots j_{s+1}}^{s+1}\bar w\\
\times\nabla^{l}_{i_1\ldots i_{l}}\Big(\big(\mathds1_{s\ne n}\nabla\varphi_{j_1\ldots j_{s+1}}^{s+1}+\varphi_{j_1\ldots j_s}^s\ee_{j_{s+1}}\big)(\tfrac\cdot\e)\Big).
\end{multline}
Now applying~\eqref{eq:op-diff-of-Taylor} to the definition of $\Xi_{\e}^{\circ,n}[\nabla\bar w]$ (cf.~Definition~\ref{def:Xin0}) together with the above formula for $\Xi_{\e}^n[E_\e^n[\nabla\bar w]]$, we are led to
\begin{multline*}
\Xi_{\e}^{\circ,n}[\nabla\bar w]\,=\,(\Aa(\tfrac\cdot\e)-\bar\Aa^1)E_\e^n[\nabla\bar w]\\
-\sum_{k=1}^{n-1}\e^{k}\big(\sym_{i_1\ldots i_k}\bar\Aa_{i_1\ldots i_{k}}^{k+1}\big)\sum_{s=0}^{n}\e^{s}\sum_{l=k+s+1-n}^k\binom{k}{l}\nabla^{k-l}_{i_{l+1}\ldots i_{k}}\nabla_{j_1\ldots j_{s+1}}^{s+1}\bar w\\
\times\nabla^{l}_{i_1\ldots i_{l}}\Big(\big(\mathds1_{s\ne n}\nabla\varphi_{j_1\ldots j_{s+1}}^{s+1}+\varphi_{j_1\ldots j_s}^s\ee_{j_{s+1}}\big)(\tfrac\cdot\e)\Big).
\end{multline*}
Comparing this with~\eqref{eq:form-Xi-E-comp} yields the conclusion.\qed

\section{Reminder on Malliavin calculus}\label{chap:Mall}

In this section, we recall some basic definitions of the Malliavin calculus with respect to the Gaussian field $G$ (e.g.~\cite{Malliavin-97,Nualart,NP-book} for details). For simplicity, we focus on the explicit description of the Malliavin derivative on a simple dense subspace of random variables (cf.~$\Rc$ hereafter), where it is seen to coincide with the formal $\Ld^2$-gradient with respect to $G$. Next, we give a precise statement of the different tools from Malliavin calculus that will be used in the sequel. We emphasize that the reader should not be intimidated by the somehow abstract appearance of the theory: for our application to stochastic homogenization all (averaged) quantities of interest happen to be Malliavin smooth, and their Malliavin derivatives are systematically computed by duality arguments so that in the end everything is expressed in terms of the (explicit) Malliavin derivative of the coefficient field~$\Aa$ itself.

\medskip
The Gaussian random field $G$ can be seen as a random Schwartz distribution, that is, as a random element in $\Sc'(\R^d)^\kappa$. Indeed, for all $\zeta_1,\zeta_2\in C^\infty_c(\R^d)^\kappa$, we define $G(\zeta_1)$, $G(\zeta_2)$ (or formally $\int_{\R^d} G\zeta_1$, $\int_{\R^d} G\zeta_2$) as centered Gaussian random variables with covariance
\[\cov{G(\zeta_1)}{G(\zeta_2)}\,:=\,\iint_{\R^d\times\R^d} \zeta_1(x)\cdot c(x-y)\,\zeta_2(y)\,dxdy.\]
We define $\Hf$ as the closure of $C^\infty_c(\R^d)^\kappa$ for the (semi)norm
\begin{equation}\label{eq:def-Hf-ps}
\|\zeta_1\|_{\Hf}^2:=\langle\zeta_1,\zeta_1\rangle_\Hf,\qquad\langle\zeta_1,\zeta_2\rangle_\Hf:=\iint_{\R^d\times\R^d} \zeta_1(x)\cdot c(x-y)\,\zeta_2(y)\,dxdy.
\end{equation}
The space $\Hf$ (up to taking the quotient with respect to the kernel of $\|\cdot\|_\Hf$) is a separable Hilbert space. Under the integrability condition~\eqref{eq:cov-L1}, it obviously contains $\Ld^2(\R^d)^\kappa$ and there holds
\begin{align}\label{eq:cov-L1-H}
\|\zeta\|_{\Hf}^2\,\lesssim\,\int_{\R^d}[\zeta]_1^2.
\end{align}
In view of the isometry relation $\cov{G(\zeta_1)}{G(\zeta_2)}=\langle\zeta_1,\zeta_2\rangle_\Hf$,
the random field $G$ is said to be an {isonormal Gaussian process} over $\Hf$.

Without loss of generality, we work under the minimality assumption $\F=\sigma(G)$,
which implies that the linear subspace
\[\Rc:=\Big\{g\big(G(\zeta_1),\ldots,G(\zeta_n)\big)\,:\,n\in\N,\,g\in C_c^\infty(\R^n),\,\zeta_1,\ldots,\zeta_n\in C_c^\infty(\R^d)^\kappa\Big\}\subset\Ld^2(\Omega)\]
is dense in $\Ld^2(\Omega)$. This allows to define operators and prove properties on the simpler subspace $\Rc$ before extending them to $\Ld^2(\Omega)$ by density. For $r\ge1$ we similarly define
\[\Rc(\Hf^{\otimes r}):=\Big\{\sum_{i=1}^n\psi_iX_i\,:\,n\in\N,\,X_1,\ldots, X_n\in\Rc,\,\psi_1,\ldots,\psi_n\in \Hf^{\otimes r}\Big\}\subset\Ld^2(\Omega;\Hf^{\otimes r}),\]
which is dense in $\Ld^2(\Omega;\Hf^{\otimes r})$.
For a random variable $X\in \Rc$, say $X=g(G(\zeta_1),\ldots,G(\zeta_n))$, we define its {Malliavin derivative} $DX\in \Ld^2(\Omega;\Hf)$ as
\begin{align}\label{eq:D-expl}
DX=\sum_{i=1}^n\zeta_i \,\partial_i g (G(\zeta_1),\ldots,G(\zeta_n)),
\end{align}
which coincides with the formal $\Ld^2$ gradient of $X$ with respect to $G$.
For an element $X\in \Rc(\Hf^{\otimes r})$ with $r\ge1$, say $X=\sum_{i=1}^n\psi_iX_i$, the Malliavin derivative $DX\in\Ld^2(\Omega;\Hf^{\otimes(r+1)})$ is defined as
$DX=\sum_{i=1}^n\psi_i\otimes DX_i$.
For $j\ge1$, we iteratively define the $j$th-order Malliavin derivative $D^j:\Rc(\Hf^{\otimes r})\to\Ld^2(\Omega;\Hf^{\otimes (r+j)})$ for all $r\ge0$.
For all $r,m\ge0$, we then set
\begin{gather*}
\|X\|_{\Dm^{m,2}(\Hf^{\otimes r})}^2:=\langle X,X\rangle_{\Dm^{m,2}(\Hf^{\otimes r})},\\
\langle X,Y\rangle_{\Dm^{m,2}(\Hf^{\otimes r})}:=\expec{\langle X,Y\rangle_{\Hf^{\otimes r}}}+\sum_{j=1}^m\expec{\langle D^jX,D^jY\rangle_{\Hf^{\otimes(r+j)}}},
\end{gather*}
we define the {Malliavin-Sobolev space} $\Dm^{m,2}(\Hf^{\otimes r})$ as the closure of $\Rc(\Hf^{\otimes r})$ for this norm, and we extend the Malliavin derivatives $D^j$ by density to these spaces.

Next, we define a {divergence operator} $D^*$ as the adjoint of the Malliavin derivative $D$ and we start with its domain: for $r\ge1$, we set
\begin{multline*}
\dom\,(D^*|_{\Ld^2(\Omega,\Hf^{\otimes r})})\,:=\,\Big\{X\in\Ld^2(\Omega,\Hf^{\otimes r})\,:\,\exists\, C<\infty\text{ such that }\\
\expec{\langle DY,X\rangle_{\Hf^{\otimes r}}}^2\le C\,\expecm{\|Y\|_{\Hf^{\otimes(r-1)}}^2}\text{ for all $Y\in\Dm^{1,2}(\Omega;\Hf^{\otimes (r-1)})$}\Big\},
\end{multline*}
and for all $X\in\dom\,(D^*|_{\Ld^2(\Omega;\Hf^{\otimes r})})$ we define the divergence $D^*X\in\Ld^2(\Omega;\Hf^{\otimes (r-1)})$ by duality via the relation
\begin{equation}\label{eq:def-Dst}
\expec{\langle Y,D^*X\rangle_{\Hf^{\otimes(r-1)}}}=\expec{\langle DY,X\rangle_{\Hf^{\otimes r}}} \qquad  \mbox{for all }Y\in\Dm^{1,2}(\Hf^{\otimes (r-1)}).
\end{equation}
One can show that $\dom\,(D^*|_{\Ld^2(\Omega;\Hf^{\otimes r})})\supset\Dm^{1,2}(\Hf^{\otimes r})$ (e.g.~\cite[Proposition~1.3.1]{Nualart}).
Combining the Malliavin derivative and the associated divergence, we construct the so-called {Ornstein-Uhlenbeck operator} (or infinite-dimensional Laplacian)
\[\Lc:=D^* D,\]
which is an essentially self-adjoint positive operator.
The explicit actions of $D^*$ and $\Lc$ on~$\Rc$ are easily computed (e.g.~\cite[p.34]{NP-book}): for $X\in\Rc(\Hf^{\otimes r})$, say $X=\sum_{i=1}^n\psi_iX_i$, we have
$\Lc X=\sum_{i=1}^n\psi_i\,\Lc X_i$, while for $X_i=g(G(\zeta_1),\ldots,G(\zeta_n))$,
\[\Lc X_i=\sum_{j=1}^nG(\zeta_j)\,\partial_jg (G(\zeta_1),\ldots,G(\zeta_n))-\sum_{j,k=1}^n\langle \zeta_j,\zeta_k\rangle_{\Hf}\,\partial^2_{jk} g (G(\zeta_1),\ldots,G(\zeta_n)).\]
In particular, a direct computation (e.g.~\cite[p.35]{NP-book}) leads to the crucial commutator relation
\begin{equation}\label{eq:commut-Mall}
D\Lc\,=\,(1+\Lc)D.
\end{equation}
In addition, one checks that $\Lc$ is shift-invariant.
While the positivity of $\Lc$ ensures that the inverse operator $(1+\Lc)^{-1}$ is contractive on $\Ld^2(\Omega)$, it is easily checked to be also contractive on $\Ld^p(\Omega)$ for all $p>1$ (cf.\@ e.g.~\cite[Proposition~3.2]{MO}).
Although we focus here on the explicit description of operators on their core $\Rc$, the action of $D,D^*,\Lc$ on general $\Ld^2$ random variables are naturally described in terms of their Wiener chaos expansion (e.g.~\cite{Nualart,NP-book}).

\medskip
Based on the above definitions, we state the following proposition collecting various useful results for the fine analysis of functionals of the Gaussian field $G$.
Item~(i) is classical~\cite{Nash-58,Chernoff-81,Houdre-PerezAbreu-95}, as well as item~(iii) (e.g.~\cite{Bakry-94} and references therein).
Item~(ii) is well-known in the discrete Gaussian setting~\cite{HS-94} (cf.\@ also~\cite{Sjostrand-96,MO}).
Item~(iv) in total variation distance (and similarly in $1$-Wasserstein distance) is a consequence of Stein's method: it was first obtained in the discrete setting by Chatterjee~\cite{C2}, while the present Malliavin analogue is due to~\cite{NP-08,NPR-09} (of which we give a slightly different formulation here).
The corresponding result in $2$-Wasserstein distance is different and is due to~\cite{LNP-15}.
Note the particular role of the so-called Stein kernel $\langle DX,(1+\Lc)^{-1}DX\rangle_{\Hf}$ in items~(ii) and~(iv).
A short proof is included in Appendix~\ref{app:Mall} for the reader's convenience.

\begin{prop}[\cite{Houdre-PerezAbreu-95,HS-94,C2,NP-08,NPR-09,LNP-15}]\label{prop:Mall}$ $
\begin{enumerate}[(i)]
\item \emph{First-order Poincaré inequality:} For all $X\in \Ld^2(\Omega)$,
\begin{align*}
\var{X}\,\le\,\expec{\|DX\|_{\Hf}^2}.
\end{align*}
\item \emph{Helffer-Sjöstrand identity:} For all $X,Y\in \Dm^{1,2}(\Omega)$,
\begin{align}\label{eq:HS}
\cov{X}{Y}\,=\,\expec{\langle DX,(1+\Lc)^{-1}DY\rangle_{\Hf}}.
\end{align}
\item \emph{Logarithmic Sobolev inequality:} For all $X\in\Ld^2(\Omega)$,
\[\ent{X^2}\,:=\,\expec{X^2\log\frac{X^2}{\expec{X^2}}}\,\le\,2\,\expec{\|DX\|_\Hf^2},\]
hence in particular, for all $p\ge1$,
\[\expec{|X-\expec{X}|^{2p}}\,\le\,(2p+1)^p\,\expecm{\|DX\|_\Hf^{2p}}.\]
\item \emph{Second-order Poincaré inequality:} For all $X\in \Ld^2(\Omega)$ with $\expec{X}=0$ and $\var{X}=1$,
\begin{eqnarray*}
\quad\dWW{X}{\Nc}\vee\dTV{X}{\Nc}&\le& 2\,\var{\langle DX,(1+\Lc)^{-1}DX\rangle_{\Hf}}^\frac12\\
&\le&3\,\expec{\|D^2X\|_{\op}^4}^\frac14\expec{\|DX\|_\Hf^4}^\frac14,
\end{eqnarray*}
where $\dWW\cdot\Nc$ and $\dTV\cdot\Nc$ denote the $2$-Wasserstein and the total variation distance to a standard Gaussian law, respectively, and where the operator norm of $D^2X$ is defined by
\begin{equation}\label{eq:def-op}
\|D^2X\|_{\op}\,:=\,\sup_{\zeta,\zeta'\in\Hf\atop\|\zeta\|_\Hf=\|\zeta'\|_\Hf=1}\langle D^2X,\zeta\otimes\zeta'\rangle_{\Hf^{\otimes2}}.
\qedhere
\end{equation}
\end{enumerate}
\end{prop}

\begin{rem}
Malliavin-type calculus has also been successfully developed in some non-Gaussian settings, to which most of our conclusions could thus be extended.
For instance, in the case of an i.i.d.\@ discrete random field $G:=(G_x)_{x\in \Z^d}$, a corresponding calculus is based on the so-called Glauber derivative $D_xX=\expeC{X}{G_x}-X$ (e.g.~\cite{DGO1}).
Another important example is that of a Poisson point process $\xi$ on $\R^d$, in which case the Malliavin derivative is chosen as the difference operator $D_xX(\xi)=X(\xi\cup\{x\})-X(\xi\setminus\{x\})$ (e.g.~\cite{Peccati-Reitzner}).
\end{rem}

\section{Representation formulas}\label{sec:rep-form}

As recalled in Proposition~\ref{prop:Mall} above, tools from Malliavin calculus allow to linearize the dependence on the underlying Gaussian field $G$ and reduce various subtle questions about quantities of interest to the estimation of their Malliavin derivative, that is, of their local dependence with respect to $G$.
In this spirit, the key ingredient for items~(ii) and~(iii) of Theorem~\ref{th:main} consists in suitable representation formulas for the Malliavin derivative of higher-order commutators.
\begin{enumerate}[$\bullet$]
\item \emph{Locality of higher-order commutators:}\\
The key property of homogenization commutators for fluctuations is their improved locality with respect to the coefficient field. While this was first discovered at first order in~\cite{DGO1} (see also~\cite{GuM,AKM2,GO4}), formula~\eqref{eq:commut-Xin-1} below indicates that the higher-order commutator $\Xi^n_\e[\nabla u_\e]$ behaves like a local quantity at the level of its Malliavin derivative up to a higher-order error $O(\e^n)$. More precisely, the formula takes the form
\begin{multline*}
\hspace{0.2cm}\ee_j\cdot D\Xi^n_\e[\nabla u_\e]
\,=\,\sum_{k=0}^{n-1}(-1)^k\e^k\nabla^k_{i_1\ldots i_k}\Big(\big(\nabla\varphi_{ji_1\ldots i_k}^{*,k+1}+\varphi^{*,k}_{ji_1\ldots i_{k-1}}\ee_{i_k}\big)(\tfrac\cdot\e)\cdot D\Aa(\tfrac\cdot\e)\nabla u_\e\Big)\,+\,O(\e^n),
\end{multline*}
where the main right-hand side terms are indeed local in view of $D_z\Aa=a_0'(G(z))\,\delta(\cdot-z)$,
and where the error can be checked to be of order $O(\e^n)$ in $\Ld^2(\Omega;\Hf)$ when integrated with a smooth test function.
While the full formula~\eqref{eq:commut-Xin-1} below is given for simplicity for $\e=1$, the $\e$-scaling is transparent by counting the number of derivatives that fall on test functions.
The proof consists in exploiting the algebraic structure of higher-order correctors and commutators in order to make this clean local structure appear.

\smallskip\noindent
Since the standard commutator $\Xi^{\circ,n}_\e[\nabla\bar w]$ is obtained by applying $\Xi_\e^n[\nabla\cdot]$ to the two-scale expansion $F_\e^n[\nabla\bar w]$ whenever $\bar w$ is an $n$th-order polynomial, we may deduce a corresponding formula~\eqref{eq:commut-Xin-2} for its Malliavin derivative,
\begin{multline*}
\quad\ee_j\cdot D\Xi^{\circ,n}_\e[\nabla\bar w]\\
\,=\,\sum_{k=0}^{n-1}(-1)^k\e^k\nabla^k_{i_1\ldots i_k}\Big(\big(\nabla\varphi_{ji_1\ldots i_k}^{*,k+1}+\varphi_{ji_1\ldots i_{k-1}}^{*,k}\ee_{i_k}\big)(\tfrac\cdot\e)\cdot D\Aa(\tfrac\cdot\e) E^n_\e[\nabla\bar w]\Big)+O(\e^n).
\end{multline*}
In view of Proposition~\ref{prop:Mall}(ii), this improved locality property provides a simple access to the higher-order characterization of the asymptotic covariance structure of large-scale averages of commutators, cf.~Section~\ref{sec:cov-conv}.

\smallskip\item \emph{Accuracy of the two-scale expansion of higher-order commutators:}\\
Comparing the formulas for Malliavin derivatives of the commutator $\Xi^n_\e[\nabla u_\e]$ and of its standard version $\Xi^{\circ,n}_\e[\nabla\bar w]$, and choosing $\bar w=\bar u_\e^n$ the homogenized solution, we find for the difference,
\begin{multline*}
\quad\ee_j\cdot D\big(\Xi^n_\e[\nabla u_\e]-\Xi^{\circ,n}_\e[\nabla\bar u_\e^n]\big)
=O(\e^n)\\
+\sum_{k=0}^{n-1}(-1)^k\e^k\nabla^k_{i_1\ldots i_k}\Big(\big(\nabla\varphi_{ji_1\ldots i_k}^{*,k+1}+\varphi^{*,k}_{ji_1\ldots i_{k-1}}\ee_{i_k}\big)(\tfrac\cdot\e)
\cdot D\Aa(\tfrac\cdot\e)\big(\nabla u_\e-E_\e^n[\nabla\bar u_\e^n]\big)\Big),
\end{multline*}
where we exactly recognize the higher-order two-scale expansion error $\nabla u_\e-E_\e^n[\nabla\bar u_\e^n]$, which is known to be of small order $O(\e^n)$ in view of Proposition~\ref{prop:homog-err}.
This constitutes the core of the proof of the higher-order pathwise result, cf.~Section~\ref{sec:pathw}.
\end{enumerate}
We turn to the precise statement of the representation formulas.
These are understood in a weak sense: when integrated in space with a smooth test function, the quantities in consideration are indeed Malliavin smooth and the formulas make sense after integrations by parts (see also Lemmas~\ref{lem:key-estimates-decomp} and~\ref{lem:key-estimates-2nd}).

\begingroup\allowdisplaybreaks
\begin{prop}[Formulas for first Malliavin derivatives]\label{prop:1st-variation}
Let $\e=1$ and drop it from all subscripts in the notation.
For all $1\le n\le\ell$ there holds for any random function $w$ with $\nabla\cdot\Aa\nabla w$ deterministic,
\begin{multline}\label{eq:commut-Xin-1}
\ee_j\cdot D\Xi^{n}[\nabla w]
\,=\,\sum_{k=0}^{n}(-1)^{k}\nabla^k_{i_1\ldots i_k}\Big(\big(\mathds1_{k\ne n}\nabla\varphi_{ji_1\ldots i_k}^{*,k+1}+\varphi_{ji_1\ldots i_{k-1}}^{*,k}\ee_{i_k}\big)\cdot D\Aa\nabla w\Big)\\
+(-1)^{n}\nabla^{n}_{i_1\ldots i_{n}}\Big(\big(\Aa^*\varphi_{ji_1\ldots i_{n-1}}^{*,n}-\sigma_{ji_1\ldots i_{n-1}}^{*,n}\big)\ee_{i_{n}}\cdot\nabla D w\Big),
\end{multline}
and also, for any smooth deterministic function $\bar w$,
\begin{multline}\label{eq:commut-Xin-2}
\ee_j\cdot D \Xi^{\circ,n}[\nabla\bar w]
\,=\,\sum_{k=0}^{n}(-1)^{k}\nabla^k_{i_1\ldots i_k}\Big(\big(\mathds1_{k\ne n}\nabla\varphi_{ji_1\ldots i_k}^{*,k+1}+\varphi_{ji_1\ldots i_{k-1}}^{*,k}\ee_{i_k}\big)\cdot D\Aa\, E^{n}[\nabla\bar w]\Big)\\
\hspace{-4.5cm}-\sum_{k=0}^{n}(-1)^k\sum_{s=n-k}^{n-1}~\sum_{l=0}^{k+s-n}\binom{k}{l}\big(\nabla^{k-l}_{i_{l+1}\ldots i_k}\nabla^{s+1}_{j_1\ldots j_{s+1}}\bar w\big)~\nabla_{i_1\ldots i_l}^l\\
\quad\times\Big(\sym_{i_1\ldots i_k}\big(\mathds1_{k\ne n}\nabla\varphi_{ji_1\ldots i_k}^{*,k+1}+\varphi_{ji_1\ldots i_{k-1}}^{*,k}\ee_{i_k}\big)\cdot D\Aa\big(\nabla\varphi^{s+1}_{j_1\ldots j_{s+1}}+\varphi^{s}_{j_1\ldots j_{s}}\ee_{j_{s+1}}\big)\Big)\\
\hspace{-5.4cm}+(-1)^{n}\sum_{s=0}^{n-1}~\sum_{l=s+1}^{n}\binom{n}{l}\big(\nabla^{n-l}_{i_{l+1}\ldots i_{n}}\nabla^{s+1}_{j_1\ldots j_{s+1}}\bar w\big)~\nabla_{i_1\ldots i_l}^l\\
\times\Big(\sym_{i_1\ldots i_{n}}\big((\Aa^*\varphi_{ji_1\ldots i_{n-1}}^{*,n}-\sigma_{ji_1\ldots i_{n-1}}^{*,n})\ee_{i_{n}}\big)\cdot D\big(\nabla\varphi^{s+1}_{j_1\ldots j_{s+1}}+\varphi^{s}_{j_1\ldots j_{s}}\ee_{j_{s+1}}\big)\Big).
\end{multline}
In particular,
\begin{eqnarray}
\lefteqn{\hspace{-0.5cm}\ee_j\cdot D\big(\Xi^{n}[\nabla u]-\Xi^{\circ,n}[\nabla\bar u^{n}]\big)}\nonumber\\
&\hspace{-0.2cm}=&\sum_{k=0}^{n}(-1)^k\nabla^k_{i_1\ldots i_k}\Big(\big(\mathds1_{k\ne n}\nabla\varphi_{ji_1\ldots i_k}^{*,k+1}+\varphi_{ji_1\ldots i_{k-1}}^{*,k}\ee_{i_k}\big)\cdot D\Aa\big(\nabla u-E^{n}[\nabla\bar u^{n}]\big)\Big)\nonumber\\
&&\hspace{-0.5cm}+\sum_{k=0}^{n}(-1)^k\sum_{s=n-k}^{n-1}~\sum_{l=0}^{k+s-n}\binom{k}{l}\big(\nabla^{k-l}_{i_{l+1}\ldots i_k}\nabla^{s+1}_{j_1\ldots j_{s+1}}\bar u^{n}\big)~\nabla_{i_1\ldots i_l}^l\nonumber\\
&&\hspace{-0.5cm}\quad\times\Big(\sym_{i_1\ldots i_k}\big(\mathds1_{k\ne n}\nabla\varphi_{ji_1\ldots i_k}^{*,k+1}+\varphi_{ji_1\ldots i_{k-1}}^{*,k}\ee_{i_k}\big)\cdot D\Aa\big(\nabla\varphi^{s+1}_{j_1\ldots j_{s+1}}+\varphi^{s}_{j_1\ldots j_{s}}\ee_{j_{s+1}}\big)\Big)\nonumber\\
&&\hspace{-0.5cm}+(-1)^{n}\nabla^{n}_{i_1\ldots i_{n}}\Big(\big(\Aa^*\varphi_{ji_1\ldots i_{n-1}}^{*,n}-\sigma_{ji_1\ldots i_{n-1}}^{*,n}\big)\ee_{i_{n}}\cdot D\big(\nabla u-E^{n}[\nabla\bar u^{n}]\big)\Big)\nonumber\\
&&\hspace{-0.5cm}+(-1)^{n}\sum_{s=0}^{n-1}~\sum_{l=0}^{s}\binom{n}{l}\big(\nabla^{n-l}_{i_{l+1}\ldots i_{n}}\nabla^{s+1}_{j_1\ldots j_{s+1}}\bar u^{n}\big)~\nabla_{i_1\ldots i_l}^l\nonumber\\
&&\hspace{-0.5cm}\quad\times\Big(\sym_{i_1\ldots i_{n}}\big((\Aa^*\varphi_{ji_1\ldots i_{n-1}}^{*,n}-\sigma_{ji_1\ldots i_{n-1}}^{*,n})\ee_{i_{n}}\big)\cdot D\big(\nabla\varphi^{s+1}_{j_1\ldots j_{s+1}}+\varphi^{s}_{j_1\ldots j_{s}}\ee_{j_{s+1}}\big)\Big).\label{eq:commut-Xin-3}
\end{eqnarray}
\qedhere
\end{prop}

In view of the approximate normality result of Theorem~\ref{th:main}(iv),
the second-order Poincaré inequality of Proposition~\ref{prop:Mall}(iv) requires to further estimate the second Malliavin derivative of higher-order standard commutators, and we establish the following formula.
Although some of the terms are proportional to $D^2\Aa$, hence display an exact locality, some others unavoidably take the form $D\Aa D\nabla\varphi^1$, hence are only approximately local. This is due to the nonlinear dependence of the solution operator on the coefficient field, and is the reason why some logarithms are lost in the convergence rate to normality in Theorem~\ref{th:main}(iv), cf.~Section~\ref{sec:normal}.

\begin{prop}[Formulas for second Malliavin derivatives]\label{prop:2nd-variation}
Let $\e=1$ and drop it from all subscripts in the notation.
For all $1\le n\le\ell$, there holds for any smooth deterministic function $\bar w$,
\begin{align*}
&\ee_j\cdot D^2 \Xi^{\circ,n}[\nabla\bar w]
\,=\,
\sum_{k=0}^{n}(-1)^{k}\nabla^k_{i_1\ldots i_k}\Big(\big(\mathds1_{k\ne n}\nabla\varphi_{ji_1\ldots i_k}^{*,k+1}+\varphi_{ji_1\ldots i_{k-1}}^{*,k}\ee_{i_k}\big)\cdot D^2\Aa\, E^{n}[\nabla\bar w]\Big)\nonumber\\
&\hspace{0.5cm}-\sum_{k=0}^{n}(-1)^k\sum_{s=n-k}^{n-1}~\sum_{l=0}^{k+s-n}\binom{k}{l}\big(\nabla^{k-l}_{i_{l+1}\ldots i_k}\nabla^{s+1}_{j_1\ldots j_{s+1}}\bar w\big)~\nabla_{i_1\ldots i_l}^l\nonumber\\
&\hspace{0.5cm}\quad\times\Big(\sym_{i_1\ldots i_k}\big(\mathds1_{k\ne n}\nabla\varphi_{ji_1\ldots i_k}^{*,k+1}+\varphi_{ji_1\ldots i_{k-1}}^{*,k}\ee_{i_k}\big)\cdot D^2\Aa\big(\nabla\varphi^{s+1}_{j_1\ldots j_{s+1}}+\varphi^{s}_{j_1\ldots j_{s}}\ee_{j_{s+1}}\big)\Big)\nonumber\\
&\hspace{0.5cm}+2\sum_{k=0}^{n}(-1)^{k}\nabla^k_{i_1\ldots i_k}\Big(\big(\mathds1_{k\ne n}\nabla\varphi_{ji_1\ldots i_k}^{*,k+1}+\varphi_{ji_1\ldots i_{k-1}}^{*,k}\ee_{i_k}\big)\cdot D\Aa\, D E^{n}[\nabla\bar w]\Big)\nonumber\\
&\hspace{0.5cm}-2\sum_{k=0}^{n}(-1)^k\sum_{s=n-k}^{n-1}~\sum_{l=0}^{k+s-n}\binom{k}{l}\big(\nabla^{k-l}_{i_{l+1}\ldots i_k}\nabla^{s+1}_{j_1\ldots j_{s+1}}\bar w\big)~\nabla_{i_1\ldots i_l}^l\nonumber\\
&\hspace{0.5cm}\quad\times\Big(\sym_{i_1\ldots i_k}\big(\mathds1_{k\ne n}\nabla\varphi_{ji_1\ldots i_k}^{*,k+1}+\varphi_{ji_1\ldots i_{k-1}}^{*,k}\ee_{i_k}\big)\cdot D\Aa\,D\big(\nabla\varphi^{s+1}_{j_1\ldots j_{s+1}}+\varphi^{s}_{j_1\ldots j_{s}}\ee_{j_{s+1}}\big)\Big)\nonumber\\
&\hspace{0.5cm}+(-1)^{n}\sum_{s=0}^{n-1}~\sum_{l=s+1}^{n}\binom{n}{l}\big(\nabla^{n-l}_{i_{l+1}\ldots i_{n}}\nabla^{s+1}_{j_1\ldots j_{s+1}}\bar w\big)~\nabla_{i_1\ldots i_l}^l\nonumber\\
&\hspace{0.5cm}\quad\times\Big(\sym_{i_1\ldots i_{n}}\big((\Aa^*\varphi_{ji_1\ldots i_{n-1}}^{*,n}-\sigma_{ji_1\ldots i_{n-1}}^{*,n})\ee_{i_{n}}\big)\cdot D^2\big(\nabla\varphi^{s+1}_{j_1\ldots j_{s+1}}+\varphi^{s}_{j_1\ldots j_{s}}\ee_{j_{s+1}}\big)\Big).
\qedhere
\end{align*}
\end{prop}

\subsection{Proof of Proposition~\ref{prop:1st-variation}}
We split the proof into three steps.

\medskip
\step1 Proof that for any random function $w$ with $\nabla\cdot\Aa\nabla w$ deterministic there holds for all $0\le n\le\ell$,
\begin{multline}\label{eq:pre-dec-commut}
\ee_j\cdot D\Xi^{n}[\nabla w]\,=\,
\sum_{k=0}^{n}(-1)^{k}\nabla^k_{i_1\ldots i_k}\Big(\big(\mathds1_{k\ne n}\nabla\varphi_{ji_1\ldots i_k}^{*,k+1}+\varphi_{ji_1\ldots i_{k-1}}^{*,k}\ee_{i_k}\big)\cdot D\Aa\nabla w\Big)\\
+(-1)^{n}\nabla^{n}_{i_1\ldots i_n}\Big(\big(\Aa^*\varphi_{ji_1\ldots i_{n-1}}^{*,n}-\sigma_{ji_1\ldots i_{n-1}}^{*,n}\big)\ee_{i_{n}}\cdot\nabla D w\Big).
\end{multline}
We argue by induction.
The result is obvious for $n=0$ (recall~$\Xi^0[\nabla w]=\Aa\nabla w$, cf.~Definition~\ref{def:Xin}). Assume that it holds for some $n\ge0$.
In order to deduce it at level $n+1$, we appeal to the alternative definition~\eqref{eq:redef-Xi} of the commutator, so that it suffices to prove
\begin{multline}\label{eq:pre-dec-commut-0}
\nabla^{n}_{i_1\ldots i_n}\big((\Aa^*\varphi_{ji_1\ldots i_{n-1}}^{*,n}-\sigma_{ji_1\ldots i_{n-1}}^{*,n})\ee_{i_{n}}\cdot\nabla D w\big)\\
=-\nabla^{n+1}_{i_1\ldots i_nl}\big((\Aa^*\varphi_{ji_1\ldots i_n}^{*,n+1}-\sigma^{*,n+1}_{ji_1\ldots i_n})\ee_l\cdot\nabla D w\big)-\nabla^{n+1}_{i_1\ldots i_nl}\big(\varphi_{ji_1\ldots i_n}^{*,n+1}\ee_l\cdot D\Aa\nabla w\big)\\
+\nabla^{n}_{i_1\ldots i_n}\big(\nabla\varphi_{ji_1\ldots i_n}^{*,n+1}\cdot D\Aa\nabla w\big)+\nabla^{n}_{i_1\ldots i_n}\big(\bar\Aa_{ji_1\ldots i_{n-1}}^{*,n+1}\ee_{i_n}\cdot\nabla D w\big).
\end{multline}
The definition of $\sigma^{*,n+1}$ (cf.~Definition~\ref{def:cor}) yields
\begin{multline*}
\nabla^{n}_{i_1\ldots i_n}\big((\Aa^*\varphi_{ji_1\ldots i_{n-1}}^{*,n}-\sigma_{ji_1\ldots i_{n-1}}^{*,n})\ee_{i_{n}}\cdot\nabla D w\big)=\nabla^{n}_{i_1\ldots i_n}\big((\nabla\cdot\sigma^{*,n+1}_{ji_1\ldots i_n})\cdot\nabla D w\big)\\
-\nabla^{n}_{i_1\ldots i_n}\big(\nabla\varphi_{ji_1\ldots i_n}^{*,n+1}\cdot\Aa\nabla D w\big)+\nabla^{n}_{i_1\ldots i_n}\big(\bar\Aa_{ji_1\ldots i_{n-1}}^{*,n+1}\ee_{i_n}\cdot\nabla D w\big).
\end{multline*}
Then using the skew-symmetry of $\sigma^{*,n+1}_{ji_1\ldots i_n}$ and the equation $D(\nabla\cdot\Aa\nabla w)=0$ in the form
\begin{eqnarray*}
\nabla\varphi_{ji_1\ldots i_n}^{*,n+1}\cdot\Aa\nabla D w&=&\nabla\cdot\big(\varphi_{ji_1\ldots i_n}^{*,n+1}\Aa\nabla D w\big)+\varphi_{ji_1\ldots i_n}^{*,n+1}\nabla\cdot D\Aa\nabla w\\
&=&\nabla\cdot\big(\varphi_{ji_1\ldots i_n}^{*,n+1}\Aa\nabla D w\big)+\nabla\cdot\big(\varphi_{ji_1\ldots i_n}^{*,n+1}D\Aa\nabla w\big)-\nabla\varphi_{ji_1\ldots i_n}^{*,n+1}\cdot D\Aa\nabla w,
\end{eqnarray*}
the claim~\eqref{eq:pre-dec-commut-0} follows.

\medskip

\step2 Proof that for any random function $w$ with $\nabla\cdot\Aa\nabla w$ deterministic and for any smooth deterministic function $\bar w$ there holds for all $0\le n,m\le\ell$,
\begin{eqnarray}\label{eq:dec-commut+}
\lefteqn{\ee_j\cdot D\Xi^{n}\big[\nabla w-E^m[\nabla\bar w]\big]}\nonumber\\
&=&\sum_{k=0}^{n}(-1)^k\nabla^k_{i_1\ldots i_k}\Big(\big(\mathds1_{k\ne n}\nabla\varphi_{ji_1\ldots i_k}^{*,k+1}+\varphi_{ji_1\ldots i_{k-1}}^{*,k}\ee_{i_k}\big)\cdot D\Aa\big(\nabla w-E^m[\nabla\bar w]\big)\Big)\nonumber\\
&&+\sum_{k=0}^{n}(-1)^k\sum_{s=1}^m~\sum_{l=0}^{k+s-m-1}\binom{k}{l}\big(\nabla^{k-l}_{i_{l+1}\ldots i_k}\nabla^s_{j_1\ldots j_s}\bar w\big)~\nabla_{i_1\ldots i_l}^l\nonumber\\
&&\quad\times\Big(\sym_{i_1\ldots i_k}\big(\mathds1_{k\ne n}\nabla\varphi_{ji_1\ldots i_k}^{*,k+1}+\varphi_{ji_1\ldots i_{k-1}}^{*,k}\ee_{i_k}\big)\cdot D\Aa\big(\nabla\varphi^s_{j_1\ldots j_s}+\varphi^{s-1}_{j_1\ldots j_{s-1}}\ee_{j_s}\big)\Big)\nonumber\\
&&+(-1)^{n}\nabla^{n}_{i_1\ldots i_n}\Big(\big(\Aa^*\varphi_{ji_1\ldots i_{n-1}}^{*,n}-\sigma_{ji_1\ldots i_{n-1}}^{*,n}\big)\ee_{i_{n}}\cdot D\big(\nabla w-E^m[\nabla\bar w]\big)\Big)\nonumber\\
&&+(-1)^n\sum_{s=1}^m~\sum_{l=0}^{n+s-m-1}\binom{n}{l}\big(\nabla^{n-l}_{i_{l+1}\ldots i_n}\nabla^s_{j_1\ldots j_s}\bar w\big)~\nabla_{i_1\ldots i_l}^l\nonumber\\
&&\quad\times\Big(\sym_{i_1\ldots i_n}\big((\Aa^*\varphi_{ji_1\ldots i_{n-1}}^{*,n}-\sigma_{ji_1\ldots i_{n-1}}^{*,n})\ee_{i_{n}}\big)\cdot D\big(\nabla\varphi^s_{j_1\ldots j_s}+\varphi^{s-1}_{j_1\ldots j_{s-1}}\ee_{j_s}\big)\Big)\nonumber\\
&&+\sum_{k=1}^{n-1}(-1)^k\sum_{s=1}^m~\sum_{l=0}^{k+s-m-1}\binom{k}{l}\big(\nabla_{i_{l+1}\ldots i_k}^{k-l}\nabla^s_{j_1\ldots j_s}\bar w\big)~\nabla^l_{i_1\ldots i_l}\nonumber\\
&&\quad\times\Big(\sym_{i_1\ldots i_k}\big(\bar\Aa^{*,k+1}_{ji_1\ldots i_{k-1}}\ee_{i_k}\big)\cdot D\big(\nabla\varphi_{j_1\ldots j_s}^s+\varphi_{j_1\ldots j_{s-1}}^{s-1}\ee_{j_s}\big)\Big),
\end{eqnarray}
where we recall that by a slight abuse of notation we similarly define $\Xi^n[H]$ as in Definition~\ref{def:Xin} even if $H$ is not a gradient field.

\medskip\noindent
Let $n\ge0$ be fixed. We argue by induction on $m$. For $m=0$, the result~\eqref{eq:dec-commut+} coincides with~\eqref{eq:pre-dec-commut} (recall~$E^0[\nabla\bar w]=0$, cf.~Definition~\ref{def:EnFn}). Now assume that~\eqref{eq:dec-commut+} holds for some $m\ge0$ and let us deduce it at the level~$m+1$.
Given a polynomial $\bar q$ of order $m+1$, noting that $\nabla^{m+2}\bar q$ vanishes, identity~\eqref{eq:pre-dec-homog} ensures that $\nabla\cdot\Aa\nabla F^{m+1}[\bar q]$ is deterministic. We then apply~\eqref{eq:dec-commut+} with $w=F^{m+1}[\bar q]$ and $\bar w=\bar q$.
Abundantly using that $\nabla^{m+1}\bar q$ is constant and noting that identity~\eqref{eq:link-E0E} together with the definition of $E^m$ yields
\begin{eqnarray*}
\nabla F^{m+1}[\bar q]-E^{m}[\nabla\bar q]&=&E^{m+1}[\nabla\bar q]-E^m[\nabla\bar q]\\
&=&\big(\nabla\varphi_{j_1\ldots j_{m+1}}^{m+1}+\varphi_{j_1\ldots j_m}^m\ee_{j_{m+1}}\big)\nabla^{m+1}_{j_1\ldots j_{m+1}}\bar q,
\end{eqnarray*}
this leads to
\begin{eqnarray*}
\lefteqn{\ee_j\cdot D\Xi^{n}\big[\nabla\varphi_{j_1\ldots j_{m+1}}^{m+1}+\varphi_{j_1\ldots j_m}^m\ee_{j_{m+1}}\big]\nabla_{j_1\ldots j_{m+1}}^{m+1}\bar q}\\
&=&\sum_{k=0}^{n}(-1)^k\big(\nabla_{j_1\ldots j_{m+1}}^{m+1}\bar q\big)~\nabla^k_{i_1\ldots i_k}\\
&&\quad\times\Big(\big(\mathds1_{k\ne n}\nabla\varphi_{ji_1\ldots i_k}^{*,k+1}+\varphi_{ji_1\ldots i_{k-1}}^{*,k}\ee_{i_k}\big)\cdot D\Aa\big(\nabla\varphi_{j_1\ldots j_{m+1}}^{m+1}+\varphi_{j_1\ldots j_m}^m\ee_{j_{m+1}}\big)\Big)\\
&&\hspace{-0.6cm}+\sum_{k=0}^{n}(-1)^k\sum_{s=1}^m\binom{k}{k+s-m-1}\big(\nabla^{m+1-s}_{i_{k+s-m}\ldots i_k}\nabla^s_{j_1\ldots j_s}\bar q\big)~\nabla_{i_1\ldots i_{k+s-m-1}}^{k+s-m-1}\\
&&\quad\times\Big(\sym_{i_1\ldots i_k}\big(\mathds1_{k\ne n}\nabla\varphi_{ji_1\ldots i_k}^{*,k+1}+\varphi_{ji_1\ldots i_{k-1}}^{*,k}\ee_{i_k}\big)\cdot D\Aa\big(\nabla\varphi^s_{j_1\ldots j_s}+\varphi^{s-1}_{j_1\ldots j_{s-1}}\ee_{j_s}\big)\Big)\\
&&\hspace{-0.6cm}+(-1)^{n}\big(\nabla^{m+1}_{j_1\ldots j_{m+1}}\bar q\big)~\nabla^{n}_{i_1\ldots i_n}\\
&&\quad\times\Big(\big(\Aa^*\varphi_{ji_1\ldots i_{n-1}}^{*,n}-\sigma_{ji_1\ldots i_{n-1}}^{*,n}\big)\ee_{i_{n}}\cdot D\big(\nabla\varphi_{j_1\ldots j_{m+1}}^{m+1}+\varphi_{j_1\ldots j_m}^m\ee_{j_{m+1}}\big)\Big)\\
&&\hspace{-0.6cm}+(-1)^n\sum_{s=1}^m\binom{n}{n+s-m-1}\big(\nabla^{m+1-s}_{i_{n+s-m}\ldots i_n}\nabla^s_{j_1\ldots j_s}\bar q\big)~\nabla_{i_1\ldots i_{n+s-m-1}}^{n+s-m-1}\\
&&\quad\times\Big(\sym_{i_1\ldots i_n}\big(\Aa^*\varphi_{ji_1\ldots i_{n-1}}^{*,n}-\sigma_{ji_1\ldots i_{n-1}}^{*,n}\big)\ee_{i_{n}}\cdot D\big(\nabla\varphi^s_{j_1\ldots j_s}+\varphi^{s-1}_{j_1\ldots j_{s-1}}\ee_{j_s}\big)\Big)\\
&&\hspace{-0.6cm}+ \sum_{k=1}^{n-1}(-1)^k\sum_{s=1}^m\binom{k}{k+s-m-1}\big(\nabla_{i_{k+s-m}\ldots i_k}^{m+1-s}\nabla^s_{j_1\ldots j_s}\bar q\big)~\nabla^{k+s-m-1}_{i_1\ldots i_{k+s-m-1}}\\
&&\quad\times\Big(\sym_{i_1\ldots i_k}\big(\bar\Aa^{*,k+1}_{ji_1\ldots i_{k-1}}\ee_{i_k}\big)\cdot D\big(\nabla\varphi_{j_1\ldots j_s}^s+\varphi_{j_1\ldots j_{s-1}}^{s-1}\ee_{j_s}\big)\Big).
\end{eqnarray*}
Note that both sides of this identity depend linearly on the (constant and deterministic) symmetric tensor $\nabla^{m+1}\bar q$. Hence, we may replace it by the (deterministic) symmetric tensor $\nabla^{m+1}\bar w$, to the effect of
\begin{eqnarray}\label{eq:dec-commut-ord-m-seul}
\lefteqn{\ee_j\cdot D\Xi^{n}\big[\nabla\varphi_{j_1\ldots j_{m+1}}^{m+1}+\varphi_{j_1\ldots j_m}^m\ee_{j_{m+1}}\big]\nabla_{j_1\ldots j_{m+1}}^{m+1}\bar w}\nonumber\\
&=&\sum_{k=0}^{n}(-1)^k\big(\nabla_{j_1\ldots j_{m+1}}^{m+1}\bar w\big)~\nabla^k_{i_1\ldots i_k}\nonumber\\
&&\quad\times\Big(\big(\mathds1_{k\ne n}\nabla\varphi_{ji_1\ldots i_k}^{*,k+1}+\varphi_{ji_1\ldots i_{k-1}}^{*,k}\ee_{i_k}\big)\cdot D\Aa\big(\nabla\varphi_{j_1\ldots j_{m+1}}^{m+1}+\varphi_{j_1\ldots j_m}^m\ee_{j_{m+1}}\big)\Big)\nonumber\\
&&\hspace{-0.6cm}+\sum_{k=0}^{n}(-1)^k\sum_{s=1}^m\binom{k}{k+s-m-1}\big(\nabla^{m+1-s}_{i_{k+s-m}\ldots i_k}\nabla^s_{j_1\ldots j_s}\bar w\big)~\nabla_{i_1\ldots i_{k+s-m-1}}^{k+s-m-1}\nonumber\\
&&\quad\times\Big(\sym_{i_1\ldots i_k}\big(\mathds1_{k\ne n}\nabla\varphi_{ji_1\ldots i_k}^{*,k+1}+\varphi_{ji_1\ldots i_{k-1}}^{*,k}\ee_{i_k}\big)\cdot D\Aa\big(\nabla\varphi^s_{j_1\ldots j_s}+\varphi^{s-1}_{j_1\ldots j_{s-1}}\ee_{j_s}\big)\Big)\nonumber\\
&&\hspace{-0.6cm}+(-1)^{n}\big(\nabla^{m+1}_{j_1\ldots j_{m+1}}\bar w\big)~\nabla^{n}_{i_1\ldots i_n}\nonumber\\
&&\quad\times\Big(\big(\Aa^*\varphi_{ji_1\ldots i_{n-1}}^{*,n}-\sigma_{ji_1\ldots i_{n-1}}^{*,n}\big)\ee_{i_{n}}\cdot D\big(\nabla\varphi_{j_1\ldots j_{m+1}}^{m+1}+\varphi_{j_1\ldots j_m}^m\ee_{j_{m+1}}\big)\Big)\nonumber\\
&&\hspace{-0.6cm}+(-1)^n\sum_{s=1}^m\binom{n}{n+s-m-1}\big(\nabla^{m+1-s}_{i_{n+s-m}\ldots i_n}\nabla^s_{j_1\ldots j_s}\bar w\big)~\nabla_{i_1\ldots i_{n+s-m-1}}^{n+s-m-1}\nonumber\\
&&\quad\times\Big(\sym_{i_1\ldots i_n}\big((\Aa^*\varphi_{ji_1\ldots i_{n-1}}^{*,n}-\sigma_{ji_1\ldots i_{n-1}}^{*,n})\ee_{i_{n}}\big)\cdot D\big(\nabla\varphi^s_{j_1\ldots j_s}+\varphi^{s-1}_{j_1\ldots j_{s-1}}\ee_{j_s}\big)\Big)\nonumber\\
&&\hspace{-0.6cm}+\sum_{k=1}^{n-1}(-1)^k\sum_{s=1}^m\binom{k}{k+s-m-1}\big(\nabla_{i_{k+s-m}\ldots i_k}^{m+1-s}\nabla^s_{j_1\ldots j_s}\bar w\big)~\nabla^{k+s-m-1}_{i_1\ldots i_{k+s-m-1}}\nonumber\\
&&\quad\times\Big(\sym_{i_1\ldots i_k}\big(\bar\Aa^{*,k+1}_{ji_1\ldots i_{k-1}}\ee_{i_k}\big)\cdot D\big(\nabla\varphi_{j_1\ldots j_s}^s+\varphi_{j_1\ldots j_{s-1}}^{s-1}\ee_{j_s}\big)\Big).
\end{eqnarray}
By definition of $E^m$ and $\Xi^{n}$ (cf.~Definitions~\ref{def:EnFn} and~\ref{def:Xin}) together with Lemma~\ref{lem:sym-baran}, using the general Leibniz rule, we find
\begin{eqnarray*}
\lefteqn{\ee_j\cdot D\Xi^{n}\big[\nabla w-E^{m+1}[\nabla\bar w]\big]}\\
&=&\ee_j\cdot D\Xi^{n}\big[\nabla w-E^m[\nabla\bar w]\big]-\ee_j\cdot D\Xi^{n}\Big[\big(\nabla\varphi_{j_1\ldots j_{m+1}}^{m+1}+\varphi_{j_1\ldots j_m}^m\ee_{j_{m+1}}\big)\nabla_{j_1\ldots j_{m+1}}^{m+1}\bar w\Big]\\
&=&\ee_j\cdot D\Xi^{n}\big[\nabla w-E^m[\nabla\bar w]\big]-\ee_j\cdot D\Xi^{n}\big[\nabla\varphi_{j_1\ldots j_{m+1}}^{m+1}+\varphi_{j_1\ldots j_m}^m\ee_{j_{m+1}}\big]\nabla_{j_1\ldots j_{m+1}}^{m+1}\bar w\\
&&+\sum_{k=1}^{n-1}(-1)^{k}\sum_{l=0}^{k-1}\binom{k}{l}\big(\nabla_{i_{l+1}\ldots i_k}^{k-l}\nabla_{j_1\ldots j_{m+1}}^{m+1}\bar w\big)~\nabla^l_{i_1\ldots i_l}\\
&&\hspace{2cm}\times\Big(\sym_{i_1\ldots i_k}\big(\bar\Aa_{ji_1\ldots i_{k-1}}^{*,k+1}\ee_{i_k}\big)\cdot D\big(\nabla\varphi_{j_1\ldots j_{m+1}}^{m+1}+\varphi_{j_1\ldots j_m}^m\ee_{j_{m+1}}\big)\Big).
\end{eqnarray*}
Injecting identities~\eqref{eq:dec-commut+} and~\eqref{eq:dec-commut-ord-m-seul} into this equality, using again the definition of $E^m$ (cf.~Definition~\ref{def:EnFn}) and the general Leibniz rule in the forms
\begin{eqnarray*}
\lefteqn{\sum_{k=0}^{n}(-1)^k\nabla^k_{i_1\ldots i_k}\Big(\big(\mathds1_{k\ne n}\nabla\varphi_{ji_1\ldots i_k}^{*,k+1}+\varphi_{ji_1\ldots i_{k-1}}^{*,k}\ee_{i_k}\big)\cdot D\Aa\big(\nabla w-E^m[\nabla\bar w]\big)\Big)}\\
&=&\sum_{k=0}^{n}(-1)^k\nabla^k_{i_1\ldots i_k}\Big(\big(\mathds1_{k\ne n}\nabla\varphi_{ji_1\ldots i_k}^{*,k+1}+\varphi_{ji_1\ldots i_{k-1}}^{*,k}\ee_{i_k}\big)\cdot D\Aa\big(\nabla w-E^{m+1}[\nabla\bar w]\big)\Big)\\
&&+\sum_{k=0}^{n}(-1)^k\sum_{l=0}^k\binom{k}{l}\big(\nabla_{i_{l+1}\ldots i_k}^{k-l}\nabla^{m+1}_{j_1\ldots j_{m+1}}\bar w\big)~\nabla^l_{i_1\ldots i_l}\\
&&\qquad\times\Big(\sym_{i_1\ldots i_k}\big(\mathds1_{k\ne n}\nabla\varphi_{ji_1\ldots i_k}^{*,k+1}+\varphi_{ji_1\ldots i_{k-1}}^{*,k}\ee_{i_k}\big)\cdot D\Aa\big(\nabla\varphi_{j_1\ldots j_{m+1}}^{m+1}+\varphi_{j_1\ldots j_m}^{m}\ee_{j_{m+1}}\big)\Big),
\end{eqnarray*}
and
\begin{eqnarray}\label{eq:folike}
\lefteqn{(-1)^{n}\nabla^{n}_{i_1\ldots i_n}\Big(\big(\Aa^*\varphi_{ji_1\ldots i_{n-1}}^{*,n}-\sigma_{ji_1\ldots i_{n-1}}^{*,n}\big)\ee_{i_{n}}\cdot D\big(\nabla w-E^{m}[\nabla\bar w]\big)\Big)}\nonumber\\
&=&(-1)^{n}\nabla^{n}_{i_1\ldots i_n}\Big(\big(\Aa^*\varphi_{ji_1\ldots i_{n-1}}^{*,n}-\sigma_{ji_1\ldots i_{n-1}}^{*,n}\big)\ee_{i_{n}}\cdot D\big(\nabla w-E^{m+1}[\nabla\bar w]\big)\Big)\nonumber\\
&&+(-1)^{n}\sum_{l=0}^n\binom{n}l\big(\nabla_{i_{l+1}\ldots i_k}^{n-l}\nabla_{j_1\ldots j_{m+1}}^{m+1}\bar w\big)~\nabla^{l}_{i_1\ldots i_l}\\
&&\qquad\times\Big(\sym_{i_1\ldots i_n}\big((\Aa^*\varphi_{ji_1\ldots i_{n-1}}^{*,n}-\sigma_{ji_1\ldots i_{n-1}}^{*,n})\ee_{i_{n}}\big)\cdot D\big(\nabla\varphi_{j_1\ldots j_{m+1}}^{m+1}+\varphi_{j_1\ldots j_{m}}^{m}\ee_{j_{m+1}}\big)\Big),\nonumber
\end{eqnarray}
and recombining the terms, the result~\eqref{eq:dec-commut+} follows at level $m+1$.
For instance, the increment of the first triple sum over $(k,s,l)$ splits into three contributions,
{\small\begin{multline*}
\sum_{k=0}^n\sum_{s=1}^{m+1}\sum_{l=0}^{k+s-(m+1)-1}-\sum_{k=0}^n\sum_{s=1}^{m}\sum_{l=0}^{k+s-m-1}
\,=\,\sum_{k=0}^n\sum_{l=0}^{k}\mathds1_{s=m+1}-\sum_{k=0}^n\mathds1_{s=m+1\atop l=k}-\sum_{k=0}^n\sum_{s=1}^m\mathds1_{l=k+s-m-1}.
\end{multline*}}

\medskip
\step3 Conclusion.\\
Applying~\eqref{eq:pre-dec-commut} yields~\eqref{eq:commut-Xin-1}.
We now turn to~\eqref{eq:commut-Xin-3} and first note that by definition of $\Xi^{\circ,n}$ (cf.~Definition~\ref{def:Xin0}) and by~\eqref{eq:link-E0E} we have
\[\big(\Xi^n\big[E^n[\nabla\bar u^n]\big]-\Xi^{\circ,n}[\nabla\bar u^n]\big)(x)\,=\,\Xi^n\big[E^n[\nabla(\bar u^n-T^n_x\bar u^n)]\big](x),\]
so that by definition of $E^n$ (cf.~Definition~\ref{def:EnFn}) and by~\eqref{eq:redef-Xi} we obtain
\begin{multline*}
\ee_j\cdot\big(\Xi^n\big[E^n[\nabla\bar u^n]\big]-\Xi^{\circ,n}[\nabla\bar u^n]\big)\\
=\,-\sum_{k=1}^{n-1}(-1)^k\sum_{s=0}^{n-1}~\sum_{l=0}^{k+s-n}\binom{k}{l}\big(\nabla_{i_{l+1}\ldots i_k}^{k-l}\nabla^{s+1}_{j_1\ldots j_{s+1}}\bar u^{n}\big)~\nabla^l_{i_1\ldots i_l}\\
\times\Big(\sym_{i_1\ldots i_k}\big(\bar\Aa^{*,k+1}_{ji_1\ldots i_{k-1}}\ee_{i_k}\big)\cdot\big(\nabla\varphi_{j_1\ldots j_{s+1}}^{s+1}+\varphi_{j_1\ldots j_{s}}^{s}\ee_{j_{s+1}}\big)\Big).
\end{multline*}
We now apply $D$ to this identity, add~\eqref{eq:dec-commut+} with $m=n$, with $(w,\bar w)$ replaced by $(u,\bar u^n)$, and with $s$ replaced by $s+1$, and so obtain~\eqref{eq:commut-Xin-3}.
Finally, \eqref{eq:commut-Xin-2} follows by subtracting~\eqref{eq:commut-Xin-3} (with $(u,\bar u^n)=(w,\bar w)$) from~\eqref{eq:commut-Xin-1} and inserting the formula for $DE^n[\nabla\bar u^n]$ already used in~\eqref{eq:folike} for $m=n$.
\qed

\subsection{Proof of Proposition~\ref{prop:2nd-variation}}
The conclusion follows by repeating the proof of Proposition~\ref{prop:1st-variation}, taking advantage of the same algebraic identities, but now further keeping track of the mixed term in $D^2(\Aa\nabla w)=\Aa\nabla D^2w+2D\Aa\nabla D w+D^2\Aa\,\nabla w$.
\qed
\endgroup

\section{Annealed Calder\'on-Zygmund theory}\label{sec:Lpreg}

In this section, we establish the following new annealed Calder\'on-Zygmund estimate for linear elliptic equations in divergence form with random coefficients.
This constitutes a useful upgrade of the quenched large-scale Calder\'on-Zygmund estimates of~\cite{Armstrong-Daniel-16,AKM-book,GNO-reg}.
The merit of~\eqref{eq:CZ}--\eqref{eq:UTb} below is that the stochastic $\Ld^q$ norm is \emph{inside} the spatial $\Ld^p$ norm; while there is a tiny loss in stochastic integrability, the spatial integrability is as in the constant-coefficient case.

\begin{theor}\label{th:CZ-ann}
Consider the Helmholtz projection $T:=\nabla(\nabla\cdot\Aa\nabla)^{-1}\nabla\cdot\Aa$.
There exists a $\frac18$-Lipschitz stationary field $r_*\ge1$ on $\R^d$ such that for all $h\in C^\infty_c(\R^d;\Ld^\infty(\Omega))^d$ and $1<q\le p<\infty$ there holds
\begin{align}\label{eq:pre-res-CZann-st}
\Bigg(\int_{\R^d}\E\bigg[\Big(\fint_{B_*(x)}|Th|^2\Big)^\frac{q}2\bigg]^\frac{p}{q}dx\Bigg)^\frac1p\,\lesssim_{p,q}\,\Bigg(\int_{\R^d}\E\bigg[\Big(\fint_{B_*(x)}|h|^2\Big)^\frac{q}2\bigg]^\frac{p}{q}dx\Bigg)^\frac1p,
\end{align}
where $B_*(x):=B_{r_*(x)}(x)$. In addition, in the Gaussian setting with integrable correlations~\eqref{eq:cov-L1}, the stationary random field $r_*$ satisfies $\expec{\exp(\frac1Cr_*^d)}\le2$ for some $C\simeq1$.
In particular, for all $1<p,q<\infty$ and $0<\delta\le\frac12$,
\begin{equation}\label{eq:CZ}
\|[Th]_2\|_{\Ld^p(\R^d;\Ld^{q}(\Omega))}\,\lesssim_{p,q}\,
\delta^{-(\frac1{p\wedge q\wedge2}-\frac1{p\vee q\vee2})}\,|\!\log\delta|^{2|\frac1{q}-\frac1p|}\,
\|[h]_2\|_{\Ld^p(\R^d;\Ld^{q+\delta}(\Omega))}.\qedhere
\end{equation}
\end{theor}

We shall also make use of the following consequence of this annealed estimate for the corresponding Riesz potentials. As opposed to the above, this can alternatively be deduced from the annealed Green's function estimates in~\cite{MaO}.

\begin{cor}\label{cor:UT}
Consider the Riesz potentials $U_{\ee,\circ}:=\nabla\triangle^{-1}\ee\cdot$ and $U_\ee:=\nabla(\nabla\cdot\Aa\nabla)^{-1}\ee\cdot\Aa$ with $|\ee|=1$. For all $h\in C^\infty_c(\R^d;\Ld^\infty(\Omega))^d$ and $1<p,q<\infty$ with $\frac{dp}{d+p}>1$,
\begin{equation}\label{eq:UT}
\|[U_{\ee,\circ} h]_2\|_{\Ld^p(\R^d;\Ld^q(\Omega))}\,\lesssim_{p}\,\|[h]_2\|_{\Ld^\frac{dp}{d+p}(\R^d;\Ld^{q}(\Omega))},
\end{equation}
and for all $0<\delta\le\frac12$,
\begin{equation}\label{eq:UTb}
\|[U_\ee h]_2\|_{\Ld^p(\R^d;\Ld^q(\Omega))}\,\lesssim_{p,q}\,\delta^{-(\frac1{p\wedge q\wedge2}-\frac1{p\vee q\vee2})}\,|\!\log\delta|^{2|\frac1{q}-\frac1p|}\,\|[h]_2\|_{\Ld^\frac{dp}{d+p}(\R^d;\Ld^{q+\delta}(\Omega))}.\qedhere
\end{equation}
\end{cor}

\subsection{Proof of Theorem~\ref{th:CZ-ann}}\label{sec:pr-th61}

Our argument is based on a multiple use of the following version of the Calder\'on-Zygmund lemma due to Shen~\cite[Theorem~2.1]{Shen-07} (cf.\@ also~\cite[Theorem~2.4]{Shen-12}), based on ideas by Caffarelli and Peral~\cite{CP-98}.
For a ball $B$, we set $B=B(x_B,r_B)$ and abusively write $\alpha B:=B(x_B,\alpha r_B)$ for $\alpha>0$.

\begin{lem}[\cite{Shen-07,CP-98}]\label{lem:shen}
Let $1\le p_0<p_1\le\infty$, $C_0>0$, and $f,g\in\Ld^{p_0}\cap\Ld^{p_1}(\R^d)$. Assume that for all balls $B\subset \R^d$ there exist measurable functions $f_{B,0}$ and $f_{B,1}$ such that $f= f_{B,0}+f_{B,1}$ on $B$ and
\begin{eqnarray*}
\Big(\fint_{B}|f_{B,0}|^{p_0}\Big)^\frac1{p_0}&\le&C_0\Big(\fint_{C_0B}|g|^{p_0}\Big)^\frac1{p_0},\\
\Big(\fint_{\frac1{C_0}B}|f_{B,1}|^{p_1}\Big)^\frac1{p_1}&\le&C_0\Big(\fint_{B}|f_{B,1}|^{p_0}\Big)^\frac1{p_0}.
\end{eqnarray*}
Then,
for all $p_0<q<p_1$,
\[\Big(\int_{\R^d}|f|^q\Big)^\frac1q\,\lesssim_{C_0,p_0,q,p_1}\,\Big(\int_{\R^d}|g|^q\Big)^\frac1q.\qedhere\]
\end{lem}

Combining this lemma with the quenched large-scale Lipschitz regularity theory due to~\cite{AS,Armstrong-Mourrat-16,GNO-reg} leads to the following quenched large-scale Calder\'on-Zygmund estimate. This slightly shortens the proof of~\cite{Armstrong-Daniel-16} and~\cite[Corollary~4]{GNO-reg} (cf.\@ also~\cite[Section~7]{AKM-book}).

\begin{prop}[Quenched large-scale Calder\'on-Zygmund estimate]\label{prop:CZ-loc}
There exists a stationary random field $r_*$ as in the statement of Theorem~\ref{th:CZ-ann} such that for all $h\in C^\infty_c(\R^d;\Ld^\infty(\Omega))^d$ and $1<p<\infty$,
\begin{eqnarray*}
\bigg(\int_{\R^d}\Big(\fint_{B_*(x)}|Th|^2\Big)^\frac p2dx\bigg)^\frac1p\,\lesssim_p\,\bigg(\int_{\R^d}\Big(\fint_{B_*(x)}|h|^2\Big)^\frac p2dx\bigg)^\frac1p,
\end{eqnarray*}
where we recall the notation $B_*(x):=B_{r_*(x)}(x)$.
\end{prop}

Before proceeding to the proof, we state the following useful properties of averages on the family of balls $\{B_*(x)\}_{x\in\R^d}$.
The point is that these balls have a radius that varies with their center $x$.

\begin{samepage}
\begin{lem}\label{lem:aver-Bstar}
For any measurable function $f$ on $\R^d$, the following estimates hold.
\begin{enumerate}[(i)]
\item For all balls $B\subset\R^d$ with $r_B\ge\frac14r_*(x_B)$, we have
\begin{equation}
\fint_{B}|f|\,\lesssim\,\fint_{2B}\Big(\fint_{B_*(x)}|f|\Big)\,dx,\label{eq:equiv-2}
\end{equation}
hence, in particular,
\begin{equation}
\int_{\R^d}|f|\,\simeq\,\int_{\R^d}\Big(\fint_{B_*(x)}|f|\Big)\,dx.\label{eq:equiv-1}
\end{equation}
\item For all balls $B\subset\R^d$ and $x_0\in B$ with $r_B\le\frac14r_*(x_0)$, we have
\begin{equation}\label{eq:equiv-3b}
\Big(\fint_{B_*(x_0)}|f|\Big)^\frac12\,\lesssim\,\fint_{5B}\Big(\fint_{B_*(x)}|f|\Big)^\frac12\,dx,
\end{equation}
hence, in particular,
\begin{equation}\label{eq:equiv-3}
\fint_{B_*(x_0)}|f|\,\lesssim\,\fint_{5B}\Big(\fint_{B_*(x)}|f|\Big)\,dx.\qedhere
\end{equation}
\end{enumerate}
\end{lem}
\end{samepage}

We may now proceed to the proof of Proposition~\ref{prop:CZ-loc}, which amounts to combining the quenched large-scale Lipschitz regularity theory of~\cite{AS,Armstrong-Mourrat-16,GNO-reg} together with Lemma~\ref{lem:shen}.

\begin{proof}[Proof of Proposition~\ref{prop:CZ-loc}]
As in~\cite[Step~1 of the proof of Corollary~4]{GNO-reg}, the random field~$r_*$ is chosen as the largest $\frac18$-Lipschitz lower bound on the so-called minimal radius defined in~\cite{GNO-reg} (what is denoted here by $r_*$ thus corresponds to the notation $\underline{r_*}$ in~\cite{GNO-reg}).
We split the proof into two steps.

\medskip
\step1 For all balls $B\subset\R^d$, decomposing $Th=\nabla w_{B,0}+\nabla w_{B,1}$ with
\begin{equation}\label{eq:decomp-T}
-\nabla\cdot\Aa\nabla w_{B,0}=\nabla\cdot(\Aa h\mathds1_{B}),\qquad-\nabla\cdot\Aa\nabla w_{B,1}=\nabla\cdot(\Aa h\mathds1_{\R^d\setminus B}),
\end{equation}
we prove for all $2\le p\le\infty$,
\begin{eqnarray}
\bigg(\fint_{B}\Big(\fint_{B_*(x)}|\nabla w_{B,0}|^2\Big)dx\bigg)^\frac12&\lesssim&\bigg(\fint_{2B}\Big(\fint_{B_*(x)}|h|^2\Big)dx\bigg)^\frac12,\label{eq:wB1}\\
\bigg(\fint_{\frac1{12}B}\Big(\fint_{B_*(x)}|\nabla w_{B,1}|^2\Big)^\frac p2dx\bigg)^\frac1p&\lesssim&\bigg(\fint_{B}\Big(\fint_{B_*(x)}|\nabla w_{B,1}|^2\Big)dx\bigg)^\frac12,\label{eq:wB2}
\end{eqnarray}
We start with the proof of~\eqref{eq:wB1}, and we first consider the case $r_B\ge\frac14r_*(x_B)$. An energy estimate on the equation for $w_{B,0}$ yields
\begin{align}\label{eq:en-est-wB1}
\int_{\R^d}|\nabla w_{B,0}|^2\,\lesssim\,\int_{B}|h|^2,
\end{align}
hence, combined with~\eqref{eq:equiv-2} and~\eqref{eq:equiv-1},
\begin{multline*}
\fint_{B}\Big(\fint_{B_*(x)}|\nabla w_{B,0}|^2\Big)dx\,\le\,|B|^{-1}\int_{\R^d}\Big(\fint_{B_*(x)}|\nabla w_{B,0}|^2\Big)dx\\
\,\lesssim\,\fint_{B}|h|^2\,\lesssim\,\fint_{2B}\Big(\fint_{B_*(x)}|h|^2\Big)dx,
\end{multline*}
which proves~\eqref{eq:wB1}.
Next, we consider the case $r_B\le\frac14r_*(x_B)$. For $x_0\in B$,
using the $\frac18$-Lipschitz property of $r_*$ in form of $|B_*(x_0)|\simeq|B_*(x_B)|$,
using the energy estimate~\eqref{eq:en-est-wB1}, and applying~\eqref{eq:equiv-3} (with $x_0$ and $B$ replaced by $x_B$ and $\frac1{10}B$), we find
\begin{align}\label{eq:wB1-point}
\fint_{B_*(x_0)}|\nabla w_{B,0}|^2\,\lesssim\,|B_*(x_B)|^{-1}\int_B|h|^2\,\lesssim\,\fint_{B_*(x_B)}|h|^2\,\lesssim\,\fint_{\frac12B}\Big(\fint_{B_*(x)}|h|^2\Big)dx.
\end{align}
(The choice of $\frac12B$ in the right-hand side does not matter here and is only made for later reference.)
Integrating this estimate over $x_0\in B$, the conclusion~\eqref{eq:wB1} follows.

\medskip\noindent
We turn to the proof of~\eqref{eq:wB2}, which is a consequence of quenched large-scale Lipschitz regularity theory, noting that $w_{B,1}$ is indeed $\Aa$-harmonic in~$B$. It is obviously sufficient to prove for all $x_0\in\frac1{12}B$,
\begin{align}
\fint_{B_*(x_0)}|\nabla w_{B,1}|^2\,\lesssim\,\fint_{B}\Big(\fint_{B_*(x)}|\nabla w_{B,1}|^2\Big)dx.\label{eq:wB2-bis}
\end{align}
If $r_*(x_0)\ge \frac13r_B$, this follows from~\eqref{eq:equiv-3} (with $B$ replaced by $\frac1{12}B$).
Let now $x_0\in\frac1{12}B$ be fixed with $r_*(x_0)\le\frac13r_B$.
Since the ball $B(x_0,r_*(x_0)+\frac1{12}r_B)$ is contained in $\frac12B$, where $w_{B,1}$ is $\Aa$-harmonic, the quenched large-scale Lipschitz regularity theory (e.g.~\cite[Theorem~1]{GNO-reg}) implies
\begin{align*}
\fint_{B_*(x_0)}|\nabla w_{B,1}|^2\,\lesssim\,\fint_{B(x_0,r_*(x_0)+\frac1{12}r_B)}|\nabla w_{B,1}|^2\,\lesssim\,\fint_{\frac12B}|\nabla w_{B,1}|^2.
\end{align*}
Since for $r_*(x_0)\le\frac13r_B$ the $\frac18$-Lipschitz property of $r_*$ yields $r_*(x_B)\le r_*(x_0)+\frac1{96}r_B\le2r_B$, the conclusion~\eqref{eq:wB2-bis} follows from~\eqref{eq:equiv-2} (with $B$ replaced by $\frac12B$).

\medskip
\step2 Conclusion.\\
Combining~\eqref{eq:wB1} and~\eqref{eq:wB2}, and applying Lemma~\ref{lem:shen} with
\begin{gather*}
p_0=2\le p_1\le\infty,\qquad f(x)=\Big(\fint_{B_*(x)}|Th|^2\Big)^\frac12,\qquad g(x)=\Big(\fint_{B_*(x)}|h|^2\Big)^\frac12,\\
f_{B,0}(x)=\Big(\fint_{B_*(x)}|\nabla w_{B,0}|^2\Big)^\frac12,\qquad f_{B,1}(x)=\Big(\fint_{B_*(x)}|\nabla w_{B,1}|^2\Big)^\frac12,
\end{gather*}
we conclude for all $2\le p<\infty$,
\begin{eqnarray*}
\bigg(\int_{\R^d}\Big(\fint_{B_*(x)}|Th|^2\Big)^\frac p2dx\bigg)^\frac1p\,\lesssim_p\,\bigg(\int_{\R^d}\Big(\fint_{B_*(x)}|h|^2\Big)^\frac p2dx\bigg)^\frac1p.
\end{eqnarray*}
A standard duality argument allows to deduce the corresponding result for $1<p\le2$ (e.g.~\cite[Step~7 of the proof of Corollary~4]{GNO-reg}).
\end{proof}

We need two additional ingredients for the proof of Theorem~\ref{th:CZ-ann}.
The first is a deterministic regularity result for $\Aa$-harmonic functions. It follows e.g.\@ from~\cite[proof of Lemma~4]{BGO-17},
but a short proof is included below for the reader's convenience.

\begin{lem}\label{lem:Bella}
If $w$ is $\Aa$-harmonic on a ball $B$, there holds
\[\Big(\fint_{\frac12B}|\nabla w|^2\Big)^\frac12\,\lesssim\,\fint_{B}|\nabla w|.\qedhere\]
\end{lem}

The second ingredient is needed to pass from statements on averages at the scale~$r_*$ to corresponding statements on unit scale.
This can be done at no loss in the spatial, but a small controlled loss in the stochastic integrability.

\begin{lem}\label{lem:postproc-r*1}
Let the stationary random field $r_*$ satisfy $\expec{\exp(\frac1Cr_*^d)}\le2$ for some $C\simeq1$.
For all $f\in C^\infty_c(\R^d;\Ld^\infty(\Omega))$ and $1\le q\le p<\infty$, there hold
\begin{enumerate}[(i)]
\item for all $r>q$,
\begin{multline*}
\qquad\bigg(\int_{\R^d}\E\bigg[\Big(\fint_{B_*(x)}|f|^2\Big)^\frac{q}2\bigg]^\frac{p}{q}dx\bigg)^\frac1p\\
\,\lesssim_{p,q}\,\Big(\frac1{q}-\frac1r\Big)^{-(\frac{1}{q}-\frac{1}2)_+}L\Big(\frac1{q}-\frac1r\Big)^{\frac{1}{q}-\frac1p}\Big(\int_{\R^d}\expec{[f]_2^{r}}^\frac{p}{r}\Big)^\frac1p;
\end{multline*}
\item for all $r<q$,
\begin{multline*}
\qquad\bigg(\int_{\R^d}\E\bigg[\Big(\fint_{B_*(x)}|f|^2\Big)^\frac{q}2\bigg]^\frac{p}{q}dx\bigg)^\frac1p\\
\,\gtrsim_{p,q}\,\Big(\frac1r-\frac1{q}\Big)^{(\frac12-\frac1{p})_+}L\Big(\frac1r-\frac1{q}\Big)^{-(\frac1{q}-\frac1p)}\Big(\int_{\R^d}\expec{[f]_2^r}^\frac pr\Big)^\frac1p;
\end{multline*}
\end{enumerate}
where we have set $L(t):=\log(2+\frac1t)$.
\end{lem}

With these ingredients at hand, we turn the proof of Theorem~\ref{th:CZ-ann}. The argument is based on a second application of Lemma~\ref{lem:shen} starting with the quenched large-scale Calder\'on-Zygmund estimate of Proposition~\ref{prop:CZ-loc}.

\begin{proof}[Proof of Theorem~\ref{th:CZ-ann}]
Let $r_*$ be chosen as in the proof of Proposition~\ref{prop:CZ-loc} above.
We split the proof into two steps.

\medskip
\step1 For all balls $B\subset\R^d$, decomposing $Th=\nabla w_{B,0}+\nabla w_{B,1}$ as in~\eqref{eq:decomp-T},
we prove for all $1< q<\infty$ and $1\le p_1\le \infty$,
\begin{gather}
\Bigg(\fint_{B}\E\bigg[\Big(\fint_{B_*(x)}|\nabla w_{B,0}|^2\Big)^\frac{q}2\bigg]dx\Bigg)^\frac1{q}\,\lesssim_{q}\,\Bigg(\fint_{6B}\E\bigg[\Big(\fint_{B_*(x)}|h|^2\Big)^\frac{q}2\bigg]dx\Bigg)^\frac1{q},\label{eq:wB1+}\\
\Bigg(\fint_{\frac1{24}B}\E\bigg[\Big(\fint_{B_*(x)}|\nabla w_{B,1}|^2\Big)^\frac{q}2\bigg]^\frac{p_1}{q}dx\Bigg)^\frac1{p_1}\,\lesssim_{q}\,\bigg(\fint_{B}\E\bigg[\Big(\fint_{B_*(x)}|\nabla w_{B,1}|^2\Big)^\frac{q}2\bigg]dx\bigg)^\frac1{q}.\label{eq:wB2+}
\end{gather}
We start with the proof of~\eqref{eq:wB1+}.
It suffices to show
\begin{align}\label{eq:rewr-wB1+}
\fint_{B}\Big(\fint_{B_*(x)}|\nabla w_{B,0}|^2\Big)^\frac{q}2dx\,\lesssim_{q}\,\fint_{6B}\Big(\fint_{B_*(x)}|h|^2\Big)^\frac{q}2dx,
\end{align}
since the conclusion~\eqref{eq:wB1+} then follows by taking the expectation.
First consider the case $r_B\le\frac14r_*(x_B)$.
Integrating~\eqref{eq:wB1-point} over $x_0\in B$ yields
\begin{align}\label{eq:bound01N}
\fint_{B}\Big(\fint_{B_*(x)}|\nabla w_{B,0}|^2\Big)^\frac q2dx\,\lesssim_q\,\bigg(\fint_{\frac12B}\Big(\fint_{B_*(x)}|h|^2\Big)dx\bigg)^\frac q2.
\end{align}
The claim~\eqref{eq:rewr-wB1+} follows from this estimate together with Jensen's inequality if $q\ge2$, but a different argument is needed if $q<2$.
For $x_0\in\frac12B$, since the condition $r_B\le\frac14r_*(x_B)$ and the $\frac18$-Lipschitz property of $r_*$ imply $r_B\le\frac14r_*(x_0)+\frac1{64}r_B$, hence $r_B\le\frac12 r_*(x_0)$, it follows from~\eqref{eq:equiv-3b} (with $B$ replaced by $\frac12B$) that
\begin{align*}
\Big(\fint_{B_*(x_0)}|h|^2\Big)^\frac12
\,\lesssim\,\fint_{\frac52B}\Big(\fint_{B_*(x)}|h|^2\Big)^{\frac12}dx.
\end{align*}
Combined with~\eqref{eq:bound01N}, this leads to~\eqref{eq:rewr-wB1+} by Jensen's inequality.
It remains to consider the case $r_B\ge\frac14r_*(x_B)$.
Applying Proposition~\ref{prop:CZ-loc} (with $q$ playing the role of $p$) to the equation for $w_{B,0}$ yields
\begin{align}\label{eq:rewr-wB1-appl-CZ}
\int_{\R^d}\Big(\fint_{B_*(x)}|\nabla w_{B,0}|^2\Big)^\frac{q}2dx\,\lesssim_{q}\,\int_{\R^d}\Big(\fint_{B_*(x)}\mathds1_{B}|h|^2\Big)^\frac{q}2dx.
\end{align}
For $x\in B$, the balls $B_*(x)$ and $B$ are disjoint whenever $|x-x_B|\ge r_*(x)+r_B$. Since the condition $r_B\ge\frac14r_*(x_B)$ and the $\frac18$-Lipschitz property of $r_*$ imply
\[r_*(x)+r_B\le r_*(x_B)+r_B+\tfrac18|x-x_B|\le5r_B+\tfrac18|x-x_B|,\]
we deduce that $B_*(x)$ and $B$ are disjoint whenever $|x-x_B|\ge5r_B+\frac18|x-x_B|$, that is, whenever $|x-x_B|\ge6r_B$. This implies
\[\int_{\R^d}\Big(\fint_{B_*(x)}\mathds1_{B}|h|^2\Big)^\frac{q}2dx\,\le\,\int_{6B}\Big(\fint_{B_*(x)}|h|^2\Big)^\frac{q}2dx,\]
so that the claim~\eqref{eq:rewr-wB1+} follows from~\eqref{eq:rewr-wB1-appl-CZ}.

\medskip\noindent
We turn to the proof of~\eqref{eq:wB2+}.
For $2\le q\le\infty$, taking the $\frac q2$th power of~\eqref{eq:wB2-bis}, using Jensen's inequality, and taking the expectation, we easily obtain~\eqref{eq:wB2+}.
In order to conclude for all values $1\le q\le \infty$, it suffices to establish the following improved version of~\eqref{eq:wB2-bis}: for all $x_0\in\frac1{24}B$,
\begin{align}\label{eq:claim-L1}
\Big(\fint_{B_*(x_0)}|\nabla w_{B,1}|^2\Big)^\frac12\,\lesssim\,\fint_{B}\Big(\fint_{B_*(x)}|\nabla w_{B,1}|^2\Big)^\frac12dx.
\end{align}
If $r_*(x_0)\ge \frac1{6}r_B$, this already follows from~\eqref{eq:equiv-3b} (with $B$ replaced by $\frac1{24}B$).
Let now $x_0\in\frac1{24}B$ be fixed with $r_*(x_0)\le\frac16r_B$.
Since $w_{B,1}$ is $\Aa$-harmonic on $B$ and since the ball $B(x_0,r_*(x_0)+\frac1{24}r_B)$ is contained in $\frac14B$, the quenched large-scale Lipschitz regularity theory (e.g.~\cite[Theorem~1]{GNO-reg}) implies
\begin{align*}
\Big(\fint_{B_*(x_0)}|\nabla w_{B,1}|^2\Big)^\frac12\,\lesssim\,\Big(\fint_{B(x_0,r_*(x_0)+\frac1{24}r_B)}|\nabla w_{B,1}|^2\Big)^\frac12\,\lesssim\,\Big(\fint_{\frac14B}|\nabla w_{B,1}|^2\Big)^\frac12,
\end{align*}
and Lemma~\ref{lem:Bella} (with $B$ replaced by $\frac12B$) then leads to
\begin{align*}
\Big(\fint_{B_*(x_0)}|\nabla w_{B,1}|^2\Big)^\frac12\,\lesssim\,\fint_{\frac12B}|\nabla w_{B,1}|.
\end{align*}
Since for $r_*(x_0)\le\frac16r_B$ the $\frac18$-Lipschitz property of $r_*$ yields $r_*(x_B)\le r_*(x_0)+\frac1{192}r_B\le2r_B$, the conclusion~\eqref{eq:claim-L1} follows from~\eqref{eq:equiv-2} (with $B$ replaced by $\frac12B$) and from Jensen's inequality.

\medskip
\step2 Conclusion.\\
Combining~\eqref{eq:wB1+} and~\eqref{eq:wB2+}, and applying Lemma~\ref{lem:shen} with
\begin{gather*}
1<p_0=q\le p_1<\infty,\quad f(x)=\E\bigg[\Big(\fint_{B_*(x)}|\nabla w|^2\Big)^\frac{q}2\bigg]^\frac1q,\quad g(x)=\E\bigg[\Big(\fint_{B_*(x)}|h|^2\Big)^\frac{q}2\bigg]^\frac1{q},\\
f_{B,0}(x)=\E\bigg[\Big(\fint_{B_*(x)}|\nabla w_{B,0}|^2\Big)^\frac{q}2\bigg]^\frac1{q},\quad f_{B,1}(x)=\E\bigg[\Big(\fint_{B_*(x)}|\nabla w_{B,1}|^2\Big)^\frac{q}2\bigg]^\frac1{q},
\end{gather*}
we deduce for all $1<q\le p<\infty$,
\begin{align*}
\Bigg(\int_{\R^d}\E\bigg[\Big(\fint_{B_*(x)}|\nabla w|^2\Big)^\frac{q}2\bigg]^\frac{p}{q}dx\Bigg)^\frac1p\,\lesssim_{p,q}\,\Bigg(\int_{\R^d}\E\bigg[\Big(\fint_{B_*(x)}|h|^2\Big)^\frac{q}2\bigg]^\frac{p}{q}dx\Bigg)^\frac1p,
\end{align*}
which is the main stated result~\eqref{eq:pre-res-CZann-st}.
It remains to post-process this result by rewriting both sides of the inequality:
combining~\eqref{eq:pre-res-CZann-st} with the different bounds stated in Lemma~\ref{lem:postproc-r*1},
and noting that for $q\le p$ there holds $(\frac1q-\frac12)_++(\frac12-\frac1p)_+=\frac1{q\wedge2}-\frac1{p\vee2}$,
the conclusion~\eqref{eq:CZ} follows for $q\le p$, while the corresponding result for $p\le q$ is deduced by duality.
\end{proof}

\subsection{Proof of Corollary~\ref{cor:UT}}\label{sec:pr-cor62}
For $d>1$, appealing to the integral representation
\[U_{\ee,\circ} g(x)\,=\,C\int_{\R^d}\frac{y}{|y|^{d}}\,\ee\cdot g(x-y)\,dy,\]
the triangle inequality yields
\[\|[U_{\ee,\circ} g]_2\|_{\Ld^p(\R^d;\Ld^q(\Omega))}\lesssim\bigg(\int_{\R^d}\Big(\int_{\R^d}\frac{\expecm{[g]_2^q(x-y)}^\frac{1}{q}}{|y|^{d-1}}dy\Big)^pdx\bigg)^\frac1p,\]
and the result~\eqref{eq:UT} follows from the Hardy-Littlewood-Sobolev inequality.
Next, writing $U_\ee=T\Aa^{-1}U_{\ee,\circ}\Aa$, the estimate~\eqref{eq:UTb} follows from the combination of~\eqref{eq:CZ} and~\eqref{eq:UT}.
\qed

\subsection{Proof of Lemma~\ref{lem:aver-Bstar}}
We start with the proof of~(i),
and thus let $r_B\ge\frac14r_*(x_B)$.
For $|x-z|\le \frac12 r_*(z)$, the $\frac18$-Lipschitz property of $r_*$ implies $|x-z|\le \frac12r_*(x)+\frac1{16}|x-z|$, hence $|x-z|\le r_*(x)$, and similarly $|B_*(x)|\simeq|B_*(z)|$. These observations lead to the lower bound
\begin{multline}\label{eq:lower-bound-int-Bstar}
\int_{2B}\Big(\fint_{B_*(x)}|f|\Big)dx\,=\,\int_{\R^d}|f(z)|\bigg(\int_{2B}\frac{\mathds1_{|x-z|\le r_*(x)}}{|B_*(x)|}dx\bigg)dz\\
\,\ge\,\int_{\R^d}|f(z)|\bigg(\int_{2B}\frac{\mathds1_{|x-z|\le \frac12 r_*(z)}}{|B_*(x)|}dx\bigg)dz\,\gtrsim\,\int_{\R^d}|f(z)|\frac{|2B\cap\frac12B_*(z)|}{|B_*(z)|}dz.
\end{multline}
Since $r_*(x_B)\le4r_B$, we note that for all $z\in B$ the ball $\frac15B_*(z)$ is included in
\[B\big(x_B,r_B+\tfrac15r_*(z)\big)\subset B\big(x_B,r_B+\tfrac15r_*(x_B)+\tfrac1{40}r_B\big)\subset 2B,\]
and~\eqref{eq:equiv-2} follows.
The upper bound in~\eqref{eq:equiv-1} is a consequence~\eqref{eq:equiv-2} with $r_B\uparrow\infty$, so that it remains to establish the lower bound.
For that purpose, we write again
\[\int_{\R^d}\Big(\fint_{B_*(x)}|f|\Big)dx\,=\,\int_{\R^d}|f(z)|\Big(\int_{\R^d}\frac{\mathds1_{|x-z|\le r_*(x)}}{|B_*(x)|}\,dx\Big)dz.\]
Since for $|x-z|\le r_*(x)$ the $\frac18$-Lipschitz property of $r_*$ implies $r_*(x)\simeq r_*(z)$, the last integral in bracket is $\lesssim1$, and the desired lower bound follows.

\medskip\noindent
We turn to the proof of~(ii).
By Jensen's inequality, it suffices to establish~\eqref{eq:equiv-3b}.
Let $x_0\in B$ with $r_B\le\frac14 r_*(x_0)$. Noting that for all $x\in B$ the ball $B_*(x_0)$ is contained in $B(x,r_*(x_0)+2r_B)$, we find
\begin{eqnarray*}
\Big(\fint_{B_*(x_0)}|f|\Big)^\frac12 &\lesssim& \fint_{B}\Big(\fint_{B(x,r_*(x_0)+2r_B)}|f|\Big)^\frac12dx.
\end{eqnarray*}
The ball $B_{4r_B}$ at the origin can be covered by $N\simeq1$ balls of radius $r_B$. Denote by $(z_i)_{i=1}^N\subset B_{4r_B}$ their centers. The collection $(B(z_i,r_*(x_0)-r_B))_{i=1}^N$ then covers $B_{r_*(x_0)+2r_B}$. By subadditivity of the square root, we may decompose
\begin{eqnarray*}
\Big(\fint_{B_*(x_0)}|f|\Big)^\frac12 &\lesssim&\fint_B\Big(\sum_{i=1}^N\fint_{B(x+z_i,r_*(x_0)-r_B)}|f|\Big)^\frac12dx\\
&\lesssim&\sum_{i=1}^N \fint_{B}\Big(\fint_{B(x+z_i,r_*(x_0)-r_B)}|f|\Big)^\frac12dx,
\end{eqnarray*}
hence, noting that $z_i+B\subset5B$ for all $i$,
\begin{eqnarray*}
\Big(\fint_{B_*(x_0)}|f|\Big)^\frac12&\lesssim&\fint_{5B}\Big(\fint_{B(x,r_*(x_0)-r_B)}|f|\Big)^\frac12dx.
\end{eqnarray*}
For $x\in5B$, the $\frac18$-Lipschitz property of $r_*$ yields
\[r_*(x_0)-r_B\le r_*(x)+\tfrac18|x-x_0|-r_B\le r_*(x)+\tfrac{3}{4} r_B-r_B\le r_*(x),\]
and~(iii) follows.\qed

\subsection{Proof of Lemma~\ref{lem:Bella}}
By scaling and translation invariance, we may assume that $B$ is the unit ball at the origin and that $\fint_Bw=0$.
Choose a smooth cut-off function $\chi$ with $\chi=1$ in $\frac12B$ and $\chi=0$ outside $B$. Caccioppoli's inequality yields
\begin{align}\label{eq:Cacc-bella}
\Big(\int_{\R^d}|\nabla(\chi w)|^2\Big)^\frac12\,\le\,\Big(\int_{\R^d}\chi^2|\nabla w|^2\Big)^\frac12+\Big(\int_{\R^d}|\nabla\chi|^2|w|^2\Big)^\frac12\,\lesssim\,\Big(\int_{\R^d}|\nabla\chi|^2|w|^2\Big)^\frac12.
\end{align}
Sobolev's inequality combined with Poincaré's inequality (with vanishing boundary condition) then implies for some suitable $p=p(d)>2$,
\begin{multline}\label{eq:Cacc-bella-bis}
\Big(\int_{\R^d}|\chi w|^p\Big)^\frac1p\,\lesssim\,\Big(\int_{\R^d}|\nabla(\chi w)|^2+|\chi w|^2\Big)^\frac12\,\lesssim\,\Big(\int_{\R^d}|\nabla(\chi w)|^2\Big)^\frac12\\
\,\stackrel{\eqref{eq:Cacc-bella}}\lesssim\,\Big(\int_{\R^d}|\nabla\chi|^2|w|^2\Big)^\frac12.
\end{multline}
Choosing $q:=2\frac{p-1}{p-2}$ and writing $\chi=\xi^q$, we deduce by Hölder's inequality,
\begin{multline*}
\Big(\int_{\R^d}|\nabla\chi|^2|w|^2\Big)^\frac12\,\lesssim\,\Big(\int_{\R^d}|\xi^{q-1} w|^2\Big)^\frac12\,=\,\Big(\int_{\R^d}|\xi^{q}w|^{\frac{p}{p-1}}|w|^{\frac{p-2}{p-1}}\Big)^\frac12\\
\,\le\,\Big(\int_{\R^d}|\chi w|^p\Big)^\frac1{2(p-1)}\Big(\int_{B}|w|\Big)^{\frac1q}\,\stackrel{\eqref{eq:Cacc-bella-bis}}\lesssim\,\Big(\int_{\R^d}|\nabla\chi|^2|w|^2\Big)^\frac{q-1}{2q}\Big(\int_{B}|w|\Big)^{\frac1q},
\end{multline*}
hence, by Young's inequality,
\[\Big(\int_{\R^d}|\nabla\chi|^2|w|^2\Big)^\frac1{2}\,\lesssim\,\int_{B}|w|.\]
Injecting this into~\eqref{eq:Cacc-bella} and using Poincaré's inequality (with vanishing mean-value), we conclude
\[\Big(\int_{\frac12B}|\nabla w|^2\Big)^\frac12\,\le\,\Big(\int_{\R^d}|\nabla(\chi w)|^2\Big)^\frac12\,\lesssim\,\int_{B}|w|\,\lesssim\,\int_{B}|\nabla w|,\]
as desired.\qed

\subsection{Proof of Lemma~\ref{lem:postproc-r*1}}
$ $

\medskip
\step1 Proof of the upper bound~(i).\\
Let $1\le q\le p<\infty$.
We start with the following deterministic version of the claimed upper bound: for all $R\ge1$ we have
\begin{align}\label{eq:(i)-determ}
\bigg(\int_{\R^d}\E\Big[\Big(\fint_{B_R(x)}|f|^2\Big)^\frac{q}2\Big]^\frac{p}{q}dx\bigg)^\frac1p\,\lesssim\,R^{d(\frac1{q}-\frac{1}2)_+}\bigg(\int_{\R^d}\expec{[f]_2^{q}}^\frac p{q}\bigg)^\frac1p.
\end{align}
If $q\ge2$, Jensen's inequality on $t\mapsto t^{\frac q2}$ yields
\begin{align}\label{eq:p0large}
\Big(\fint_{B_R(x)}|f|^2\Big)^{\frac{q}2}\,\lesssim\,\Big(\fint_{B_R(x)}[f]_2^2\Big)^{\frac{q}2}\,\lesssim\,\fint_{B_R(x)}[f]_2^{q}.
\end{align}
If $q\le2$, covering the ball $B_R$ by $N\simeq R^d$ balls of radius $\frac12$ with centers $(z_i)_{i=1}^N\subset B_{R-\frac14}$ such that $\min_{i\ne j}|z_i-z_j|\ge\frac12$, it follows from the discrete $\ell^{2}$--$\ell^{q}$ inequality that
\begin{multline*}
\Big(\int_{B_R(x)}|f|^2\Big)^{\frac{q}2}\,\le\,\Big(\sum_{i=1}^N\int_{B_{\frac12}(x+z_i)}|f|^2\Big)^{\frac{q}2}\,\le\,\sum_{i=1}^N\Big(\int_{B_{\frac12}(x+z_i)}|f|^2\Big)^{\frac{q}2}\\
\,\le\,\sum_{i=1}^N\fint_{B_{\frac14}(x+z_i)}\Big(\int_{B(y)}|f|^2\Big)^{\frac{q}2}dy\,\lesssim\,\sum_{i=1}^N\int_{B_{\frac14}(x+z_i)}[f]^{q}_2,
\end{multline*}
and hence, since the balls $(B_{\frac14}(x+z_i))_{i=1}^N$ are disjoint and included in $B_{R}(x)$,
\begin{align}\label{eq:p0small}
\Big(\int_{B_R(x)}|f|^2\Big)^{\frac{q}2}\,\lesssim\,\int_{B_{R}(x)}[f]_2^{q}.
\end{align}
These two estimates~\eqref{eq:p0large} and~\eqref{eq:p0small} lead to
\begin{align*}
\bigg(\int_{\R^d}\E\Big[\Big(\fint_{B_R(x)}|f|^2\Big)^\frac{q}2\Big]^\frac{p}{q}dx\bigg)^\frac1p\,\lesssim\,R^{d(\frac1{q}-\frac{1}2)_+}\bigg(\int_{\R^d}\E\Big[\fint_{B_R(x)}[f]_2^{q}\Big]^\frac{p}{q}dx\bigg)^\frac1p,
\end{align*}
and the claim~\eqref{eq:(i)-determ} follows from Jensen's inequality on $t\mapsto t^{\frac pq}$.

\medskip\noindent
We turn to the proof of~(i).
Conditioning with respect to the value of $r_*(x)$ on dyadic scale and using Jensen's inequality on $t\mapsto t^{\frac pq}$ with $\sum_{n=0}^\infty2^{-nd\delta}\simeq\delta^{-1}$,
we find for $0<\delta\le1$,
\begin{eqnarray*}
\lefteqn{\int_{\R^d}\E\bigg[\Big(\fint_{B_*(x)}|f|^2\Big)^{\frac{q}2}\bigg]^\frac{p}{q}dx\,\lesssim\,\int_{\R^d}\E\bigg[\sum_{n=0}^\infty\mathds1_{2^n-1<r_*(x)\le2^{n+1}-1}\Big(\fint_{B_{2^{n+1}}(x)}|f|^2\Big)^{\frac{q}2}\bigg]^\frac{p}{q}dx}\\
&\lesssim&\int_{\R^d}\E\bigg[\sum_{n=0}^\infty2^{-nd\delta}r_*(x)^{d\delta}\mathds1_{2^n-1<r_*(x)\le2^{n+1}-1}\Big(\fint_{B_{2^{n+1}}(x)}|f|^2\Big)^{\frac{q}2}\bigg]^\frac{p}{q}dx\\
&\lesssim&\delta^{1-\frac p{q}}\sum_{n=0}^\infty2^{nd\delta(\frac p{q}-1)}\int_{\R^d}\E\bigg[\mathds1_{2^n-1< r_*(x)\le2^{n+1}-1}\Big(\fint_{B_{2^{n+1}}(x)}|f|^2\Big)^{\frac{q}2}\bigg]^\frac{p}{q}dx.
\end{eqnarray*}
Taking advantage of the $\frac18$-Lipschitz property of~$r_*$, we can rewrite
\begin{multline*}
\int_{\R^d}\E\bigg[\Big(\fint_{B_*(x)}|f|^2\Big)^{\frac{q}2}\bigg]^\frac{p}{q}dx\\
\,\lesssim\,\delta^{1-\frac p{q}}\sum_{n=0}^\infty2^{nd\delta(\frac p{q}-1)}\int_{\R^d}\E\bigg[\Big(\fint_{B_{2^{n+1}}(x)}\mathds1_{2^{n-1}-1< r_*\le2^{n+2}-1}|f|^2\Big)^{\frac{q}2}\bigg]^\frac{p}{q}dx,
\end{multline*}
and applying~\eqref{eq:(i)-determ} with $R=2^{n+1}$ then leads to
\[{\int_{\R^d}\E\bigg[\Big(\fint_{B_*(x)}|f|^2\Big)^{\frac{q}2}\bigg]^\frac{p}{q}dx}
\,\lesssim\,\delta^{1-\frac p{q}}\sum_{n=0}^\infty(2^{ndp})^{(\frac1{q}-\frac12)_++\delta(\frac1{q}-\frac1p)}\int_{\R^d}\expec{\mathds1_{2^{n-1}-2\le r_*\le2^{n+2}}[f]_2^{q}}^\frac{p}{q}.\]
Hence, for $r>q$, using Hölder's inequality and the moment bounds on the stationary $r_*$,
\begin{eqnarray*}
\lefteqn{\int_{\R^d}\E\bigg[\Big(\fint_{B_*(x)}|f|^2\Big)^{\frac{q}2}\bigg]^\frac{p}{q}dx}\\
&\lesssim&\delta^{1-\frac p{q}}\sum_{n=0}^\infty(2^{ndp})^{(\frac1{q}-\frac12)_++\delta(\frac1{q}-\frac1p)}\,\pr{r_*\ge2^{n-1}-2}^{p\frac{r-q}{rq}}\int_{\R^d}\expec{[f]_2^{r}}^\frac{p}{r}\\
&\lesssim_{q,p}&\delta^{1-\frac p{q}}\Big(\frac{rq}{r-q}\Big)^{p\big((\frac1{q}-\frac12)_++\delta(\frac1{q}-\frac1p)\big)}\int_{\R^d}\expec{[f]_2^{r}}^\frac{p}{r},
\end{eqnarray*}
where in the last estimate we used $\sum_n(2^{nd})^\alpha e^{-2^{nd}\theta}\lesssim_{\alpha,d}\theta^{-\alpha}$.
Optimizing $\frac1{\delta}\theta^{-\delta}$ in $0<\delta\le1$ through $\delta=(\log\frac1\theta)^{-1}$ yielding $\sim\log\frac1\theta$, the desired upper bound follows.

\medskip
\step2 Proof of the lower bound~(ii).\\
Let $1\le r\le p\le\infty$. We start with the following deterministic version of the claimed lower bound: for all $R\ge1$ we have
\begin{align}\label{eq:(ii)-determ}
\bigg(\int_{\R^d}\E\Big[{\Big(\fint_{B_R(x)}|f|^{2}\Big)^\frac{r}2}\Big]^\frac{p}{r}dx\bigg)^\frac1{p}
\,\gtrsim\,R^{-d(\frac12-\frac1p)_+}\Big(\int_{\R^d}\expec{[f]_2^{r}}^\frac{p}{r}\Big)^\frac1p.
\end{align}
If $2\le r\le p$, using the discrete $\ell^{r}$--$\ell^2$ inequality and the discrete $\ell^p$--$\ell^r$ inequality,
we indeed find
\begin{multline*}
\int_{\R^d}\E\Big[\Big(\fint_{B_R(x)}|f|^2\Big)^\frac{r}2\Big]^\frac{p}{r}dx\,\gtrsim\,R^{-dp(\frac{1}2-\frac1{r})}\int_{\R^d}\E\Big[\fint_{B_R(x)}[f]_2^{r}\Big]^\frac{p}{r}dx\\
\,\gtrsim\,R^{-dp(\frac1{2}-\frac1{p})}\int_{\R^d}\expec{[f]_2^{r}}^\frac{p}{r}.
\end{multline*}
If $r\le2\le p$, the triangle inequality in $\Ld^{\frac2r}(B_R(x))$ and the discrete $\ell^p$--$\ell^2$ inequality lead to
\begin{align*}
\int_{\R^d}\E\Big[{\Big(\fint_{B_R(x)}|f|^{2}\Big)^\frac{r}2}\Big]^\frac{p}{r}dx
\,\gtrsim\,\int_{\R^d}\Big(\fint_{B_R(x)}\expec{[f]_2^{r}}^\frac{2}{r}\Big)^\frac{p}2dx
\,\gtrsim\,R^{-dp(\frac12-\frac1p)}\int_{\R^d}\expec{[f]_2^{r}}^\frac{p}{r}.
\end{align*}
Finally, if $r\le p\le2$, using again the triangle inequality in $\Ld^{\frac2r}(B_R(x))$ and Jensen's inequality on $t\mapsto t^\frac2p$, we obtain
\begin{align*}
\int_{\R^d}\E\Big[{\Big(\fint_{B_R(x)}|f|^{2}\Big)^\frac{r}2}\Big]^\frac{p}{r}dx
\,\gtrsim\,\int_{\R^d}\Big(\fint_{B_R(x)}\expec{[f]_2^{r}}^\frac{2}{r}\Big)^\frac{p}2dx\,\gtrsim\,\int_{\R^d}\expec{[f]_2^{r}}^\frac{p}{r}.
\end{align*}
This proves the claim~\eqref{eq:(ii)-determ} in all cases.

\medskip\noindent
We turn to the proof of~(ii). Conditioning with respect to the value of $r_*(x)$ on dyadic scale and using Jensen's inequality on $t\mapsto t^\frac pr$ with $\sum_{n=0}^\infty2^{-nd\delta}\simeq\delta^{-1}$, we find for $0<\delta\le1$,
\begin{eqnarray*}
\int_{\R^d}\expec{[f]_2^r}^\frac pr&\lesssim&\int_{\R^d}\E\bigg[{\sum_{n=0}^\infty2^{-nd\delta}r_*(x)^{d\delta}\mathds1_{2^n-1< r_*(x)\le2^{n+1}-1}[f]_2^r(x)}\bigg]^\frac prdx\\
&\lesssim&\delta^{1-\frac pr}\sum_{n=0}^\infty2^{nd\delta(\frac pr-1)}\int_{\R^d}\expec{\mathds1_{2^n-1< r_*(x)\le2^{n+1}-1}[f]_2^r(x)}^\frac prdx.
\end{eqnarray*}
Applying~\eqref{eq:(ii)-determ} with $R=(2^{n}-2)\vee1$ and taking advantage of the $\frac18$-Lipschitz property of~$r_*$, we deduce
\begin{multline*}
\int_{\R^d}\expec{[f]_2^r}^\frac pr\\
\,\lesssim\,\delta^{1-\frac pr}\sum_{n=0}^\infty(2^{ndp})^{(\frac12-\frac1p)_++\delta(\frac1r-\frac1p)}\int_{\R^d}\E\Big[\mathds1_{2^{n-1}-2\le r_*(x)\le2^{n+2}}\Big(\fint_{B_*(x)}|f|^2\Big)^\frac{r}2\Big]^\frac prdx.
\end{multline*}
Hence, for $q>r$, using Hölder's inequality and the moment bounds on $r_*$, and proceeding as for~(i),
\begin{eqnarray*}
\lefteqn{\int_{\R^d}\expec{[f]_2^r}^\frac pr}\\
&\lesssim&\delta^{1-\frac pr}\sum_{n=0}^\infty(2^{ndp})^{(\frac12-\frac1p)_++\delta(\frac1r-\frac1p)}\pr{r_*\ge2^{n-1}-2}^{p\frac{q-r}{qr}}\int_{\R^d}\E\Big[\Big(\fint_{B_*(x)}|f|^2\Big)^\frac{q}2\Big]^\frac p{q}dx\\
&\lesssim_{q,p}&\delta^{1-\frac pr}\Big(\frac{qr}{q-r}\Big)^{p\big((\frac12-\frac1p)_++\delta(\frac1r-\frac1p)\big)}\int_{\R^d}\E\Big[\Big(\fint_{B_*(x)}|f|^2\Big)^\frac{q}2\Big]^\frac p{q}dx.
\end{eqnarray*}
Optimizing in $0<\delta\le1$ yields the desired lower bound.
\qed

\section{Proof of the higher-order pathwise result}\label{sec:pathw}
In order to establish the accuracy of the two-scale expansion of higher-order commutators (hence, the higher-order pathwise structure), we appeal to the representation formula~\eqref{eq:commut-Xin-3} established in Proposition~\ref{prop:1st-variation} and the conclusion easily follows from the following estimates on the Malliavin derivative of higher-order correctors.

\begin{lem}\label{lem:key-estimates-decomp}
Let $T$, $U_{\ee}$, and $U_{\ee,\circ}$ denote the operators defined in Theorem~\ref{th:CZ-ann} and Corollary~\ref{cor:UT} with $\Aa$ replaced by its adjoint $\Aa^*$.
For all $g\in C^\infty_c(\R^d;\Ld^\infty(\Omega))^d$ and $0\le s\le\ell-1$,
\begin{multline}\label{eq:decomp-Js}
J_s(g)^2\,:=\,\int_{\R^d}\Big[\int_{\R^d}g\cdot\Aa\nabla D\varphi^{s+1}\Big]_1^2
+\int_{\R^d}\Big[\int_{\R^d}g\cdot\Aa D\varphi^{s}\Big]_1^2+\int_{\R^d}\Big[\int_{\R^d}g\cdot D\sigma^{s}\Big]_1^2\\
\lesssim\sum_{k=0}^s\sum_{\beta\in\{0,1\}}\big\|[U_\cdot^k T^\beta g]_2\big([\nabla\varphi^{s+1-k}]_2+[\varphi^{s-k}]_2\big)\big\|_{\Ld^2(\R^d)}^2,
\end{multline}
where $U_\cdot^k$ denotes any product of the form $U_1\ldots U_k$ with $U_j\in\{U_{\ee_i},U_{\ee_i,\circ}\}_{1\le i\le d}$ for all $j$, and where we implicitly sum over all possible such products.
Furthermore, for
all $\delta>0$ and $2\le p<\infty$,
\begin{equation}\label{eq:bound-E-Js}
\expec{J_s(g)^p}^\frac1p~\lesssim_{p,\delta}~\|[g]_2\|_{\Ld^{\frac{2d}{d+2s}}(\R^d;\Ld^{p+\delta}(\Omega))}.
\qedhere
\end{equation}
\end{lem}

\begin{proof}
Denote by $J_{s,1}(g)^2$, $J_{s,2}(g)^2$, and $J_{s,3}(g)^2$ the three left-hand side terms in~\eqref{eq:decomp-Js}, hence $J_s(g)^2=J_{s,1}(g)^2+J_{s,2}(g)^2+J_{s,3}(g)^2$.
We split the proof into two steps.

\medskip
\step1 Proof that for all $g\in C^\infty_c(\R^d;\Ld^\infty(\Omega))^d$ and $0\le s\le\ell-1$,
\begin{eqnarray}
\hspace{-0.5cm}J_{s,1}(g)&\lesssim&J_{s,2}(Tg)+J_{s,3}(Tg)+\big\|[Tg]_2\big([\nabla\varphi^{s+1}]_2+[\varphi^s]_2\big)\big\|_{\Ld^2(\R^d)},\label{eq:S1-path-est}\\
\hspace{-0.5cm}J_{s,2}(g)&=&J_{s-1,1}(Ug),\label{eq:S2-path-est}\\
\hspace{-0.5cm}J_{s,3}(g)&\lesssim&J_{s-1}(U_\circ g)
+\big\|[U_\circ g]_2\big([\nabla\varphi^s]_2+[\varphi^{s-1}]_2\big)\big\|_{\Ld^2(\R^d)}.\label{eq:S3-path-est}
\end{eqnarray}
We start with~\eqref{eq:S1-path-est}.
Taking the Malliavin derivative of the equation for $\varphi^{s+1}$ (cf.~Definition~\ref{def:cor}) and testing it with $(\nabla\cdot\Aa^*\nabla)^{-1}\nabla\cdot(\Aa^*g)$, we obtain by definition of $Tg=\nabla(\nabla\cdot\Aa^*\nabla)^{-1}\nabla\cdot(\Aa^*g)$,
\begin{eqnarray*}
\Big|\int_{\R^d}g\cdot\Aa\nabla D_z\varphi^{s+1}_{i_1\ldots i_{s+1}}\Big|&=&\Big|\int_{\R^d}Tg\cdot\Aa\nabla D_z\varphi^{s+1}_{i_1\ldots i_{s+1}}\Big|\\
&\lesssim&|Tg(z)|\big(|\nabla\varphi^{s+1}_{i_1\ldots i_{s+1}}(z)|+|\varphi_{i_1\ldots i_s}^s(z)|\big)\\
&&+\Big|\int_{\R^d}Tg\cdot \Aa\,D_z\varphi_{i_1\ldots i_s}^s\ee_{i_{s+1}}\Big|+\Big|\int_{\R^d}Tg\cdot D_z\sigma_{i_1\ldots i_s}^s\ee_{i_{s+1}}\Big|,
\end{eqnarray*}
where for the first right-hand side term we used that by definition~\eqref{eq:def-A} the Malliavin derivative of $\Aa$ takes the form
\begin{align}\label{eq:Da}
D_z\Aa=a_0'(G(z))\,\delta(\cdot-z).
\end{align}
The claim~\eqref{eq:S1-path-est} follows by applying $(\int_{\R^d}(\int_{B(x)}\cdot\,dz)^2dx)^\frac12$ to both sides of the estimate.
Next, by definition of $U_{\ee_{i_{s+1}}}g=\nabla(\nabla\cdot\Aa^*\nabla)^{-1}(\ee_{i_{s+1}}\cdot\Aa^*g)$,
we find
\[\Big|\int_{\R^d}g\cdot\Aa D_z\varphi^{s}_{i_1\ldots i_s}\ee_{i_{s+1}}\Big|=\Big|\int_{\R^d}U_{\ee_{i_{s+1}}}g\cdot\Aa \nabla D_z\varphi^{s}_{i_1\ldots i_s}\Big|,\]
so that~\eqref{eq:S2-path-est} follows directly.
We turn to~\eqref{eq:S3-path-est}. Writing $g=\triangle\triangle^{-1}g$, integrating by parts, and using the Malliavin derivative of the equation for $\triangle\sigma^s$ (cf.~Definition~\ref{def:cor}), we find
\begin{multline*}
{\Big|\int_{\R^d}g\cdot D_z\sigma^{s}_{i_1\ldots i_{s}}\ee_{i_{s+1}}\Big|}
\,\le\,\Big|\int_{\R^d}(\nabla_i\triangle^{-1}g_i)\,\ee_{i_{s+1}}\cdot D_zq^{s}_{i_1\ldots i_{s}}\Big|\\
+\Big|\int_{\R^d}(\nabla_{i_{s+1}}\triangle^{-1}g_{i})\,\ee_i\cdot D_zq^{s}_{i_1\ldots i_{s}}\,\Big|,
\end{multline*}
and the claim~\eqref{eq:S3-path-est} follows from taking the Malliavin derivative of the definition of $q^s$ and writing $\nabla\triangle^{-1}g_i=U_{\ee_i,\circ}g$.

\medskip
\step2 Conclusion.\\
Using the notation introduced in the statement, Step~1 yields for all $g\in C_c^\infty(\R^d;\Ld^\infty(\Omega))^d$,
\begin{multline*}
J_{s}(g)\lesssim_pJ_{s-1}(U_\cdot g)+J_{s-1}(U_\cdot Tg)
+\big\|[Tg]_2\big([\nabla\varphi^{s+1}]_2+[\varphi^{s}]_2\big)\big\|_{\Ld^2(\R^d)}\\
+\big\|\big([U_\circ g]_2+[U_\circ Tg]_2\big)\big([\nabla\varphi^{s}]_2+[\varphi^{s-1}]_2\big)\big\|_{\Ld^2(\R^d)}.
\end{multline*}
Since by definition $TU=U$ and $TU_\circ=U_\circ$, the decomposition~\eqref{eq:decomp-Js} follows by induction.
Next, combining~\eqref{eq:decomp-Js} with the corrector estimates of Proposition~\ref{prop:cor}, we find for all $g\in C^\infty_c(\R^d;\Ld^\infty(\Omega))^d$, $0\le s\le\ell-1$, $2\le p<\infty$, and $\delta>0$,
\begin{align*}
\expec{J_s(g)^p}^\frac1p\,\lesssim_{p,\delta}\,\sum_{k=0}^s\sum_{\beta\in\{0,1\}}\big\|[U_\cdot^k T^\beta g]_2\big\|_{\Ld^2(\R^d;\Ld^{p+\delta}(\Omega))},
\end{align*}
where the anchoring in $s=0$ follows from~\eqref{eq:S1-path-est} because of $J_{0,2}= J_{0,3}=0$ in view of $\varphi^0=1$ and $\sigma^0=0$. Here we have used that for a stationary $\psi$ we have by the triangle inequality in $\Ld^\frac p2(\Omega)$ and Hölder's inequality,
\[\expec{\|\psi g\|_{\Ld^2(\R^d)}^p}^\frac1p\,\le\,\|\psi g\|_{\Ld^2(\R^d;\Ld^p(\Omega))}\,\le\,\expec{|\psi|^{\frac{p(p+\delta)}\delta}}^\frac\delta{p(p+\delta)}\|g\|_{\Ld^2(\R^d;\Ld^{p+\delta}(\Omega))}.\]
Now, for $0\le s\le\ell-1$, $2\le p<\infty$, $0\le k\le s$, $\beta\in\{0,1\}$, and $\delta>0$,
we apply the bound of Corollary~\ref{cor:UT} $k$ times and the one of Theorem~\ref{th:CZ-ann} once, increasing at every application the stochastic integrability by $\frac\delta{k+1}$, to the effect of
\[\big\|[U_\cdot^k T^\beta g]_2\big\|_{\Ld^2(\R^d;\Ld^{p+\delta}(\Omega))}\,\lesssim_{p,\delta}\,\|[g]_2\|_{\Ld^{\frac{2d}{d+2k}}(\R^d;\Ld^{p+2\delta}(\Omega))},\]
given that the restriction $k\le s\le\ell-1$ indeed ensures $\frac{2d}{d+2k}>1$, and the conclusion follows.
\end{proof}

With these estimates at hand, we may now turn to the proof of Theorem~\ref{th:main}(ii), that is, of the higher-order pathwise result.

\begin{proof}[Proof of Theorem~\ref{th:main}(ii)]
Let $1\le p<\infty$. The starting point is the $\Ld^p$ inequality of Proposition~\ref{prop:Mall}(iii) together with~\eqref{eq:cov-L1-H}, in the form
\begin{multline*}
Q\,:=\,\expec{\Big(\int_{\R^d} g\cdot\big(\Xi^{n}[\nabla u]-\Xi^{\circ,n}[\nabla\bar u^{n}]\big)-\int_{\R^d} g\cdot\expec{\Xi^{n}[\nabla u]-\Xi^{\circ,n}[\nabla\bar u^{n}]}\Big)^{2p}}^\frac1{2p}\\
\,\lesssim_p\, \expec{\bigg(\int_{\R^d}\Big[\int_{\R^d}g\cdot D\big(\Xi^{n}[\nabla u]-\Xi^{\circ,n}[\nabla\bar u^{n}]\big)\Big]_1^2\bigg)^p}^\frac1{2p}.
\end{multline*}
Inserting identity~\eqref{eq:commut-Xin-3} of Proposition~\ref{prop:1st-variation}, integrating by parts in the $z$-integral to put derivatives on $g$ and $\bar u^n$, using~\eqref{eq:Da}, and using the corrector estimates of Proposition~\ref{prop:cor}, we deduce for all $1\le n\le\ell$,
\begin{multline}\label{eq:decomp-SG-(ii)}
Q\,\lesssim_p\,\sum_{k=0}^{n}\big\|\mu_{d,k}[\nabla^kg]_\infty\big\|_{\Ld^4(\R^d)}\big\|\big[\nabla u-E^{n}[\nabla\bar u^{n}]\big]_2\big\|_{\Ld^4(\R^d;\Ld^{4p}(\Omega))}\\
+\sum_{k=0}^{n}\sum_{s=n-k}^{n-1}\sum_{l=0}^{k+s-n}\big\|\mu_{d,k}\big[\nabla^l\big(g\nabla^{k+s+1-l}\bar u^{n}\big)\big]_\infty\big\|_{\Ld^2(\R^d)}\\
+\expec{\bigg(\int_{\R^d}\Big[\int_{\R^d}S^{1}\cdot\nabla Du\Big]_1^2\bigg)^p}^\frac1{2p}+\expec{\bigg(\int_{\R^d}\Big[\int_{\R^d}S^{1}\cdot DE^{n}[\nabla\bar u^{n}]\Big]_1^2\bigg)^p}^\frac1{2p}\\
+\sum_{s=0}^{n-1}\sum_{l=0}^{s}\expec{\bigg(\int_{\R^d}\Big[\int_{\R^d}S^{2;l,s}_{j_1\ldots j_{s+1}}\cdot D\big(\nabla\varphi^{s+1}_{j_1\ldots j_{s+1}}+\varphi^{s}_{j_1\ldots j_{s}}\ee_{j_{s+1}}\big)\Big]_1^2\bigg)^p}^\frac1{2p},
\end{multline}
where we have set
\begin{eqnarray*}
S^{1}&:=&\nabla^{n}_{i_1\ldots i_{n}}g_j\,\big(\Aa^*\varphi_{ji_1\ldots i_{n-1}}^{*,n}-\sigma_{ji_1\ldots i_{n-1}}^{*,n}\big)\ee_{i_{n}},\\
S^{2;l,s}_{j_1\ldots j_{s+1}}&:=&\nabla_{i_1\ldots i_l}^l\big(g_j\nabla^{n-l}_{i_{l+1}\ldots i_{n}}\nabla^{s+1}_{j_1\ldots j_{s+1}}\bar u^{n}\big)
\sym_{i_1\ldots i_{n}}\big((\Aa^*\varphi_{ji_1\ldots i_{n-1}}^{*,n}-\sigma_{ji_1\ldots i_{n-1}}^{*,n})\ee_{i_{n}}\big).
\end{eqnarray*}
We separately treat the different right-hand side terms in~\eqref{eq:decomp-SG-(ii)}.
First, from equation~\eqref{eq:u-dev-scale}, formula~\eqref{eq:link-E0E}, the annealed Calder\'on-Zygmund estimate of Theorem~\ref{th:CZ-ann}, and the corrector estimates in Proposition~\ref{prop:cor}, we obtain
\begin{multline}\label{eq:bound-nabu-E-err}
\big\|\big[\nabla u-E^n[\nabla\bar u^{n}]\big]_2\big\|_{\Ld^4(\R^d;\Ld^{4p}(\Omega))}\\
\,\lesssim_p\,\|\mu_{d,n}[\nabla^{n+1}\bar u^{n}]_\infty\|_{\Ld^4(\R^d)}+\sum_{k=2}^{n}~\sum_{l=n+2-k}^{n}\|\nabla^{k}\tilde u^{l}\|_{\Ld^4(\R^d)}.
\end{multline}
Second, recalling the notation $T=\nabla(\nabla\cdot\Aa^*\nabla)^{-1}\nabla\cdot\Aa^*$ from Theorem~\ref{th:CZ-ann} and using the equation for $u$ in the form $-\nabla\cdot\Aa \nabla D_zu=\nabla\cdot D_z\Aa\nabla u$, we find
\begin{eqnarray*}
\expec{\bigg(\int_{\R^d}\Big[\int_{\R^d}S^{1}\cdot \nabla Du\Big]_1^2\bigg)^p}^\frac1{2p}
&=&\expec{\bigg(\int_{\R^d}\Big[\int_{\R^d}T((\Aa^*)^{-1} S^{1})\cdot \Aa\nabla Du\Big]_1^2\bigg)^p}^\frac1{2p}\\
&\lesssim&\expec{\bigg(\int_{\R^d}[T((\Aa^*)^{-1} S^{1})]_2^2\,[\nabla u]_2^2\bigg)^p}^\frac1{2p}\\
&\le&\big\|[T((\Aa^*)^{-1} S^{1})]_2\big\|_{\Ld^4(\R^d;\Ld^{4p}(\Omega))}\big\|[\nabla u]_2\big\|_{\Ld^4(\R^d;\Ld^{4p}(\Omega))},
\end{eqnarray*}
hence, using Theorem~\ref{th:CZ-ann} to bound the operator $T$ and to estimate $\nabla u$, and applying the corrector estimates of Proposition~\ref{prop:cor}, we deduce
\begin{multline}\label{eq:CLT-scaling-pr}
\expec{\bigg(\int_{\R^d}\Big[\int_{\R^d}S^{1}\cdot \nabla Du\Big]_1^2\bigg)^p}^\frac1{2p}
\,\lesssim_{p,\delta}\,\|[S^{1}]_2\|_{\Ld^4(\R^d;\Ld^{4p+\delta}(\Omega))}\|[f]_2\|_{\Ld^4(\R^d)}\\
\,\lesssim\,\|\mu_{d,n}[\nabla^ng]_\infty\|_{\Ld^4(\R^d)}\|[f]_2\|_{\Ld^4(\R^d)}.
\end{multline}
Third, decomposing
\[D_zE^n[\nabla\bar u^{n}]=\sum_{k=0}^{n-1}\big(D_z\nabla\varphi_{i_1\ldots i_{k+1}}^{k+1}+D_z\varphi_{i_1\ldots i_k}^k\ee_{i_{k+1}}\big)\nabla_{i_1\ldots i_{k+1}}^{k+1}\bar u^{n},\]
applying~\eqref{eq:bound-E-Js}, and using the corrector estimates of Proposition~\ref{prop:cor},
\begin{multline*}
\expec{\bigg(\int_{\R^d}\Big[\int_{\R^d}S^{1}\cdot DE^{n}[\nabla\bar u^{n}]\Big]_1^2\bigg)^p}^\frac1{2p}
\,\lesssim_{p,\delta}\,\sum_{k=0}^{n-1}\big\|[S^{1}\nabla^{k+1}\bar u^n]_2\big\|_{\Ld^{\frac{2d}{d+2k}}(\R^d;\Ld^{2p+\delta}(\Omega))}\\
\,\lesssim_{p,\delta}\,\sum_{k=0}^{n-1}\big\|\mu_{d,n}[\nabla^ng\nabla^{k+1}\bar u^n]_\infty\big\|_{\Ld^{\frac{2d}{d+2k}}(\R^d)}.
\end{multline*}
Fourth, applying~\eqref{eq:bound-E-Js} and using the corrector estimates of Proposition~\ref{prop:cor},
\begin{multline*}
\sum_{s=0}^{n-1}\sum_{l=0}^{s}\expec{\bigg(\int_{\R^d}\Big[\int_{\R^d}S^{2;l,s}_{j_1\ldots j_{s+1}}\cdot D\big(\nabla\varphi^{s+1}_{j_1\ldots j_{s+1}}+\varphi^{s}_{j_1\ldots j_{s}}\ee_{j_{s+1}}\big)\Big]_1^2\bigg)^p}^\frac1{2p}\\
\,\lesssim_{p,\delta}\,\sum_{s=0}^{n-1}\sum_{l=0}^{s}\big\|[S^{2;l,s}]_2\big\|_{\Ld^{\frac{2d}{d+2s}}(\R^d;\Ld^{2p+\delta}(\Omega))}\\
\,\lesssim_{p,\delta}\,\sum_{s=0}^{n-1}\sum_{l=0}^{s}\big\|\mu_{d,n}\big[\nabla^l(g\nabla^{n+s+1-l}\bar u^{n})\big]_\infty\big\|_{\Ld^\frac{2d}{d+2s}(\R^d)}.
\end{multline*}
Collecting the above estimates on the right-hand side terms of~\eqref{eq:decomp-SG-(ii)}, we see that all of them are bilinear in $g$ and $f$. Crucial for us is their scaling in the underlying lengthscale, which has three origins: the (total) number of derivatives (to be counted negatively), the spatial integrability exponent~$p$ (to be counted in form of $\frac dp$), and the weights $\mu_{d,n}$ (cf.~Proposition~\ref{prop:cor}). For instance, in view of~\eqref{eq:bound-nabu-E-err}, the first right-hand side term in~\eqref{eq:decomp-SG-(ii)} has scaling $\le L^{-n+\frac d4+\frac d4}\mu_{d,n}(L)=L^{-n+\frac d2}\mu_{d,n}(L)$, where $L$ stands for the lengthscale, because $\nabla\tilde u^l$ has the same scaling as $\nabla^{l-1}f$ (cf.~Definition~\ref{def:homog-eqn}). Let us take the last term as a second example: the scaling of its $(s,l)$-contribution is $L^{-(n+s)+\frac{d+2s}2}\mu_{d,n}(L)=L^{-n+\frac{d}2}\mu_{d,n}(L)$. In fact, all terms have a scaling at most $L^{-n+\frac d2}\mu_{d,n}(L)$, which after $\e$-rescaling ($L=\e^{-1}$) turns into the desired estimate.
\end{proof}

\section{Proof of the convergence of the covariance structure}\label{sec:cov-conv}
Combining the representation formula~\eqref{eq:commut-Xin-2} of Proposition~\ref{prop:1st-variation} with the Helffer-Sjöstrand identity of Proposition~\ref{prop:Mall}(ii), we establish Theorem~\ref{th:main}(iii), that is, the convergence of the covariance structure of the higher-order standard homogenization commutator.

\begin{proof}[Proof of Theorem~\ref{th:main}(iii)]
Combining the Helffer-Sjöstrand identity of Proposition~\ref{prop:Mall}(ii) with the representation formula~\eqref{eq:commut-Xin-2} tested with $g\in C^\infty_c(\R^d)^d$, we obtain, using~\eqref{eq:Da} and recalling that $(1+\Lc)^{-1}$ has operator norm $\le1$,
\begin{multline}\label{eq:decomp-SL1S+T}
\bigg|\var{\int_{\R^d}g\cdot\Xi^{\circ,n}[\nabla\bar w]}-\expec{\langle S,(1+\Lc)^{-1}S\rangle_\Hf}\bigg|\\
\le\,2\,\expec{\|S\|_\Hf^2}^\frac12\expec{\|R\|_\Hf^2}^\frac12+\expec{\|R\|_\Hf^2},
\end{multline}
where we have set
\begin{align*}
S(z)\,:=\,\sum_{k=0}^{n}\nabla^k_{i_1\ldots i_k}g_j(z)\,\big(\mathds1_{k\ne n}\nabla\varphi_{ji_1\ldots i_k}^{*,k+1}+\varphi_{ji_1\ldots i_{k-1}}^{*,k}\ee_{i_k}\big)(z)\cdot a_0'(G(z))\,E^{n}[\nabla\bar w](z),
\end{align*}
and
\begingroup\allowdisplaybreaks
\begin{multline*}
R(z):=-\sum_{k=0}^{n}(-1)^{k}\sum_{s=n-k}^{n-1}~\sum_{l=0}^{k+s-n}(-1)^l\binom{k}{l}\nabla_{i_1\ldots i_l}^l\big(g_j\nabla^{k-l}_{i_{l+1}\ldots i_k}\nabla^{s+1}_{j_1\ldots j_{s+1}}\bar w\big)(z)\\
\times\sym_{i_1\ldots i_k}\big(\mathds1_{k\ne n}\nabla\varphi_{ji_1\ldots i_k}^{*,k+1}+\varphi_{ji_1\ldots i_{k-1}}^{*,k}\ee_{i_k}\big)(z)\cdot a_0'(G(z))\big(\nabla\varphi^{s+1}_{j_1\ldots j_{s+1}}+\varphi^{s}_{j_1\ldots j_{s}}\ee_{j_{s+1}}\big)(z)\\
+(-1)^{n}\sum_{s=0}^{n-1}~\sum_{l=s+1}^{n}(-1)^l\binom{n}{l}\int_{\R^d}\nabla_{i_1\ldots i_l}^l\big(g_j\nabla^{n-l}_{i_{l+1}\ldots i_{n}}\nabla^{s+1}_{j_1\ldots j_{s+1}}\bar w\big)\\
\times\sym_{i_1\ldots i_{n}}\big((\Aa^*\varphi_{ji_1\ldots i_{n-1}}^{*,n}-\sigma_{ji_1\ldots i_{n-1}}^{*,n})\ee_{i_{n}}\big)\cdot D_z\big(\nabla\varphi^{s+1}_{j_1\ldots j_{s+1}}+\varphi^{s}_{j_1\ldots j_{s}}\ee_{j_{s+1}}\big).
\end{multline*}
\endgroup
Appealing to the definition of $E^n[\nabla\bar w]$ (cf.~Definition~\ref{def:EnFn}), to~\eqref{eq:cov-L1-H}, and to the corrector estimates of Proposition~\ref{prop:cor}, we find
\[\expec{\|S\|_\Hf^2}^\frac12\,\lesssim\,\sum_{k=0}^n\sum_{l=0}^{n-1}\big\|\mu_{d,n}[\nabla^kg]_\infty[\nabla^{l+1}\bar w]_\infty\big\|_{\Ld^2(\R^d)},\]
and similarly,
\begin{multline*}
\expec{\|R\|_\Hf^2}^\frac12\,\lesssim\,
\sum_{k=0}^{n}\sum_{s=n-k}^{n-1}~\sum_{l=0}^{k+s-n}\big\|\mu_{d,k}[\nabla^l(g\nabla^{k+s+1-l}\bar w)]_\infty\big\|_{\Ld^2(\R^d)}\\
+\sum_{s=0}^{n-1}~\sum_{l=s+1}^{n}\expec{\int_{\R^d}\Big[\int_{\R^d}S^{2;l,s}_{j_1\ldots j_{s+1}}\cdot D_z\big(\nabla\varphi^{s+1}_{j_1\ldots j_{s+1}}+\varphi^{s}_{j_1\ldots j_{s}}\ee_{j_{s+1}}\big)\Big]_1^2}^\frac12,
\end{multline*}
in terms of
\[S^{2;l,s}_{j_1\ldots j_{s+1}}\,:=\,\nabla_{i_1\ldots i_l}^l\big(g_j\nabla^{n-l}_{i_{l+1}\ldots i_{n}}\nabla^{s+1}_{j_1\ldots j_{s+1}}\bar w\big)
\sym_{i_1\ldots i_{n}}\big((\Aa^*\varphi_{ji_1\ldots i_{n-1}}^{*,n}-\sigma_{ji_1\ldots i_{n-1}}^{*,n})\ee_{i_{n}}\big).\]
Then applying~\eqref{eq:bound-E-Js} and again using the corrector estimates of Proposition~\ref{prop:cor}, we deduce for all $\delta>0$,
\begin{eqnarray*}
\expec{\|R\|_\Hf^2}^\frac12&\lesssim_\delta&
\sum_{k=0}^{n}\sum_{s=n-k}^{n-1}~\sum_{l=0}^{k+s-n}\big\|\mu_{d,k}[\nabla^l(g\nabla^{k+s+1-l}\bar w)]_\infty\big\|_{\Ld^2(\R^d)}\\
&&\qquad\qquad+\sum_{s=0}^{n-1}~\sum_{l=s+1}^{n}\|[S^{2;l,s}]_2\|_{\Ld^\frac{2d}{d+2s}(\R^d;\Ld^{2+\delta}(\Omega))}\\
&\lesssim_{\delta}&\sum_{k=0}^{n}\sum_{s=n-k}^{n-1}~\sum_{l=0}^{k+s-n}\big\|\mu_{d,k}[\nabla^l(g\nabla^{k+s+1-l}\bar w)]_\infty\big\|_{\Ld^2(\R^d)}\\
&&\qquad\qquad+\sum_{s=0}^{n-1}~\sum_{l=s+1}^{n}\big\|\mu_{d,n}\big[\nabla^l(g\nabla^{n+s+1-l}\bar w)\big]_\infty\big\|_{\Ld^{\frac{2d}{d+2s}}(\R^d)}.
\end{eqnarray*}
As in the proof of Theorem~\ref{th:main}(ii) (cf.~Section~\ref{sec:pathw}), it is easily checked that these expressions in $(g,\nabla\bar w)$ have the desired length scaling, hence lead to the claimed error estimate after $\e$-rescaling.

\medskip\noindent
Now that the right-hand side of~\eqref{eq:decomp-SL1S+T} is properly estimated, we turn to a suitable decomposition of $\expec{\langle S,(1+\Lc)^{-1}S\rangle_\Hf}$.
Inserting the definition of $E^n[\nabla\bar w]$ (cf.~Definition~\ref{def:EnFn}) into $S$, using the definition~\eqref{eq:def-Hf-ps} of the scalar product in $\Hf$, expanding the product, and focusing on terms of order $\e^k$ with $k\le n-1$, we write
\begin{multline*}
\bigg|\,\expec{\langle S,(1+\Lc)^{-1}S\rangle_\Hf}-\sum_{l,l',k,k'=0}^{n-1}\mathds1_{l+l'+k+k'\le n-1}\\
\times\iint_{\R^d\times\R^d}c(z'-z)\,\nabla^{l'}_{i_1'\ldots i_{l'}'}g_{j'}(z')\,\nabla_{j_1'\ldots j_{k'+1}'}^{k'+1}\bar w(z')\,\nabla^l_{i_1\ldots i_l}g_j(z)\,\nabla_{j_1\ldots j_{k+1}}^{k+1}\bar w(z)\\
\hspace{3cm}\times\expec{T^{l',k'}_{j'i_1'\ldots i_{l'}'j_1'\ldots j_{k'+1}'}(z')\,(1+\Lc)^{-1}T^{l,k}_{ji_1\ldots i_lj_1\ldots j_{k+1}}(z)}dzdz'\,\bigg|\\
\lesssim\sum_{l,l'=0}^{n}\sum_{k,k'=0}^{n-1}\mathds1_{l+l'+k+k'\ge n}\big\|\mu_{d,l'}[\nabla^{l'}g]_\infty[\nabla^{k'+1}\bar w]_\infty\big\|_{\Ld^2(\R^d)}\big\|\mu_{d,l}[\nabla^{l}g]_\infty[\nabla^{k+1}\bar w]_\infty\big\|_{\Ld^2(\R^d)},
\end{multline*}
where we have set
\begin{equation*}
T^{l,k}_{ji_1\ldots i_lj_1\ldots j_{k+1}}:=
\big(\mathds1_{l\ne n}\nabla\varphi_{ji_1\ldots i_l}^{*,l+1}+\varphi_{ji_1\ldots i_{l-1}}^{*,l}\ee_{i_l}\big)
\cdot a_0'(G)\,
\big(\nabla\varphi_{j_1\ldots j_{k+1}}^{k+1}+\varphi^{k}_{j_1\ldots j_k}\ee_{j_{k+1}}\big).
\end{equation*}
Combining this with~\eqref{eq:decomp-SL1S+T} and with the above estimates on $\expec{\|S\|_\Hf^2}^\frac12$ and $\expec{\|R\|_\Hf^2}^\frac12$ and appealing to the stationarity of $T^{l,k}_{ji_1\ldots i_lj_1\ldots j_{k+1}}$ (recalling that $\Lc$ commutes with shifts), we deduce after $\e$-rescaling,
\begin{multline*}
\bigg|\var{\e^{-\frac d2}\int_{\R^d}g\cdot\Xi_{\e}^{\circ,n}[\nabla\bar w]}-\sum_{m=0}^{n-1}\e^m\sum_{l,l',k,k'\ge0\atop l+l'+k+k'=m}\\
\times\iint_{\R^d\times\R^d}c(y)\,\nabla^{l'}_{i'_1\ldots i'_{l'}}g_{j'}(z+\e y)\,\nabla^{k'+1}_{j'_1\ldots j'_{k'+1}}\bar w(z+\e y)\,\nabla^{l}_{i_1\ldots i_l}g_j(z)\,\nabla^{k+1}_{j_1\ldots j_{k+1}}\bar w(z)\\
\times\expec{T^{l',k'}_{j'i_1'\ldots i_{l'}'j_1'\ldots j_{k'+1}'}(y)\,(1+\Lc)^{-1}T^{l,k}_{ji_1\ldots i_lj_1\ldots j_{k+1}}(0)}dydz\bigg|~~
\lesssim_{n,f,g}\, \e^n\mu_{d,n}(\tfrac1\e).
\end{multline*}
Appealing to the boundedness of $(1+\Lc)^{-1}$ and to the corrector estimates of Proposition~\ref{prop:cor} in the form
\[\big|\expecm{T^{l',k'}_{j'i_1'\ldots i_{l'}'j_1'\ldots j_{k'+1}'}(y)\,(1+\Lc)^{-1}T^{l,k}_{ji_1\ldots i_lj_1\ldots j_{k+1}}(0)}\big|\lesssim1,\]
the conclusion follows after Taylor-expanding $\nabla^{l'}g_{j'}(z+\e y)\nabla^{k'+1}\bar w(z+\e y)$ to order $\e^{n-m}$ with the integrability assumption $\int_{\R^d}|y|^{n}|c(y)|dy\le1$.
\end{proof}

\section{Proof of the normal approximation result}\label{sec:normal}

In order to establish the asymptotic normality of the higher-order standard homogenization commutators, we combine the representation formula of Proposition~\ref{prop:2nd-variation} with the second-order Poincaré inequality of Proposition~\ref{prop:Mall}(iv).
We then need the following version of Lemma~\ref{lem:key-estimates-decomp} for the estimation of second Malliavin derivatives of higher-order correctors.

\begingroup\allowdisplaybreaks
\begin{lem}\label{lem:key-estimates-2nd}
For all $\zeta\in C^\infty_c(\R^d)$, $g\in C^\infty_c(\R^d)^d$, $0\le s\le\ell-1$, and $p\ge1$,
\begin{multline*}
K_s(g,\zeta):=\bigg|\int_{\R^d\times\R^d}\zeta(x)\zeta(y)\Big(\int_{\R^d}g\cdot\Aa\nabla D_{xy}^2\varphi^{s+1}\Big)\,dxdy\bigg|\\
+\bigg|\int_{\R^d\times\R^d}\zeta(x)\zeta(y)\Big(\int_{\R^d}g\cdot\Aa D^2_{xy}\varphi^{s}\Big)\,dxdy\bigg|
+\bigg|\int_{\R^d\times\R^d}\zeta(x)\zeta(y)\Big(\int_{\R^d}g\cdot D^2_{xy}\sigma^{s}\Big)\,dxdy\bigg|\\
\,\lesssim\,\|[\zeta]_\infty\|_{\Ld^\frac{2p}{p-1}(\R^d)}^2\sum_{k=0}^{s}\sum_{\beta\in\{0,1\}}\big\|[U_\cdot^kT^\beta g]_2\big([\nabla\varphi^{s+1-k}]_2+[\varphi^{s-k}]_2\big)\big\|_{\Ld^p(\R^d)}\\
+\|[\zeta]_\infty\|_{\Ld^2(\R^d)}\sum_{k=0}^{s}\sum_{\beta\in\{0,1\}}J_{s-k}\big(\zeta(\Aa^*)^{-1}a_0' (G)^*U_\cdot^kT^\beta g\big),
\end{multline*}
where the notation is
taken from
Lemma~\ref{lem:key-estimates-decomp}.
\end{lem}

\begin{proof}
Denote by $K_{s,1}(g,\zeta)$, $K_{s,2}(g,\zeta)$, and $K_{s,3}(g,\zeta)$ the three left-hand side terms in the claimed estimate, hence $K_s(g,\zeta)=K_{s,1}(g,\zeta)+K_{s,2}(g,\zeta)+K_{s,3}(g,\zeta)$.
We prove that for all $g\in C^\infty_c(\R^d;\Ld^\infty(\R^d))^d$, $\zeta\in C^\infty_c(\R^d)$, $0\le s\le\ell-1$, and $p\ge1$,
\begin{eqnarray}
\hspace{-0.5cm}K_{s,1}(g,\zeta)&\lesssim&K_{s,2}(Tg,\zeta)+K_{s,3}(Tg,\zeta)+\|[\zeta]_\infty\|_{\Ld^2(\R^d)}J_{s}\big(\zeta(\Aa^*)^{-1} a_0'(G)^*Tg\big)\nonumber\\
\hspace{-0.5cm}&&\hspace{2cm}+\|[\zeta]_\infty\|_{\Ld^{\frac{2p}{p-1}}(\R^d)}^2\big\|\big([\nabla \varphi^{s+1}]_2+[\varphi^{s}]_2\big)[Tg]_2\big\|_{\Ld^p(\R^d)},\label{eq:S1-path-est-2}\\
\hspace{-0.5cm}K_{s,2}(g,\zeta)&=&K_{s-1,1}(Ug,\zeta),\label{eq:S2-path-est-2}\\
\hspace{-0.5cm}K_{s,3}(g,\zeta)&\lesssim&K_{s-1}(U_\circ g,\zeta)+\|[\zeta]_\infty\|_{\Ld^2(\R^d)}J_{s-1}\big(\zeta(\Aa^*)^{-1}a_0'(G)^*U_\circ g\big)\nonumber\\
\hspace{-0.5cm}&&\hspace{2cm}+\|[\zeta]_\infty\|_{\Ld^\frac{2p}{p-1}(\R^d)}^2\big\|\big([\nabla\varphi^s]_2+[\varphi^{s-1}]_2\big)[U_\circ g]_2\big\|_{\Ld^p(\R^d)}.\label{eq:S3-path-est-2}
\end{eqnarray}
These estimates combine to
\begin{multline*}
K_s(g,\zeta)\,\lesssim\,\sum_{\beta\in\{0,1\}}\bigg(K_{s-1}\big(U_\cdot T^\beta g,\zeta\big)\\
+\|[\zeta]_\infty\|_{\Ld^\frac{2p}{p-1}(\R^d)}\Big\|\big([\nabla\varphi^{s+1}]_2+[\varphi^s]_2\big)[Tg]_2+\big([\nabla\varphi^s]_2+[\varphi^{s-1}]_2\big)[U_\circ T^\beta g]_2\Big\|_{\Ld^p(\R^d)}\\
+\|[\zeta]_\infty\|_{\Ld^2(\R^d)}\Big(J_s\big(\zeta(\Aa^*)^{-1} a_0'(G)^*Tg\big)+J_{s-1}\big(\zeta(\Aa^*)^{-1}a_0'(G)^* U_\circ T^\beta g\big)\Big)\bigg).
\end{multline*}
Noting as in the proof of Lemma~\ref{lem:key-estimates-decomp} that $TU=U$ and $TU_\circ=U_\circ$, the conclusion follows by induction.

\medskip\noindent
Let $0\le s\le \ell-1$. We start with the proof of~\eqref{eq:S1-path-est-2}.
Taking the second Malliavin derivative of the equation for $\varphi^{s+1}$ (cf.~Definition~\ref{def:cor}), we find
\begin{multline*}
-\nabla\cdot\Aa\nabla D^2_{xy}\varphi^{s+1}_{i_1\ldots i_{s+1}}\,=\,
\nabla\cdot\big((\Aa D^2_{xy}\varphi^{s}_{i_1\ldots i_{s}}-D^2_{xy}\sigma^{s}_{i_1\ldots i_{s}})\,\ee_{i_{s+1}}\big)\\
+\nabla\cdot\big(D_x\Aa (\nabla D_y\varphi^{s+1}_{i_1\ldots i_{s+1}}+D_y\varphi^{s}_{i_1\ldots i_{s}}\,\ee_{i_{s+1}})\big)
+\nabla\cdot\big(D_y\Aa (\nabla D_x\varphi^{s+1}_{i_1\ldots i_{s+1}}+D_x\varphi^{s}_{i_1\ldots i_{s}}\,\ee_{i_{s+1}})\big)\\
+\nabla\cdot \big(D_{xy}^2\Aa(\nabla \varphi^{s+1}_{i_1\ldots i_{s+1}}+\varphi^{s}_{i_1\ldots i_{s}}\,\ee_{i_{s+1}})\big).
\end{multline*}
Testing with $(\nabla\cdot\Aa^*\nabla)^{-1}\nabla\cdot(\Aa^*g)$, recalling the definition of $T$, and using~\eqref{eq:Da} and
\begin{equation}\label{eq:D2a}
D_{xy}^2\Aa=a_0''(G(x))\,\delta(\cdot-x)\,\delta(x-y),
\end{equation}
we find after integration against $\zeta\otimes\zeta$,
\begin{multline}\label{eq:foaddit}
\bigg|\int_{\R^d\times\R^d}\zeta(x)\zeta(y)\Big(\int_{\R^d}g\cdot\Aa\nabla D^2_{xy}\varphi^{s+1}_{i_1\ldots i_{s+1}}\Big)\,dxdy\bigg|\\
\lesssim\,\bigg|\int_{\R^d\times\R^d}\zeta(x)\zeta(y)\Big(\int_{\R^d}Tg\cdot\big(\Aa D^2_{xy}\varphi^{s}_{i_1\ldots i_{s}}-D^2_{xy}\sigma^{s}_{i_1\ldots i_{s}}\big)\ee_{i_{s+1}}\Big)\,dxdy\bigg|\\
+\int_{\R^d}|\zeta(y)|\,\bigg|\int_{\R^d}\zeta\,Tg\cdot a_0'(G)\big(\nabla D_y\varphi_{i_1\ldots i_{s+1}}^{s+1}+D_y\varphi_{i_1\ldots i_{s}}^{s}\ee_{i_{s+1}}\big)\bigg|\,dy\\
+\int_{\R^d}[\zeta]_\infty^2[Tg]_2\big([\nabla \varphi^{s+1}]_2+[\varphi^{s}]_2\big),
\end{multline}
and the claim~\eqref{eq:S1-path-est-2} follows from Hölder's inequality.
Next, \eqref{eq:S2-path-est-2} follows from the definition of $U$.
We turn to~\eqref{eq:S3-path-est-2}. Writing $g=\triangle\triangle^{-1}g$, integrating by parts, and using the equation for $\triangle\sigma^s$ (cf.~Definition~\ref{def:cor}), we find
\begin{multline*}
\bigg|\int_{\R^d\times\R^d}\zeta(x)\zeta(y)\Big(\int_{\R^d}g\cdot D^2_{xy}\sigma^{s}_{i_1\ldots i_s}\ee_{i_{s+1}}\Big)\,dxdy\bigg|\\
\le\,\bigg|\int_{\R^d\times\R^d}\zeta(x)\zeta(y)\Big(\int_{\R^d}(\nabla_i\triangle^{-1}g_i)\,\ee_{i_{s+1}}\cdot D^2_{xy} q^{s}_{i_1\ldots i_s}\Big)\,dxdy\bigg|\\
+\bigg|\int_{\R^d\times\R^d}\zeta(x)\zeta(y)\Big(\int_{\R^d}(\nabla_{i_{s+1}}\triangle^{-1}g_i)\,\ee_i\cdot D^2_{xy} q^{s}_{i_1\ldots i_s}\Big)\,dxdy\bigg|.
\end{multline*}
Hence, by the definition of $q^s$ to which we apply $D^2_{xy}$, we obtain an analogous structure as in~\eqref{eq:foaddit},
and the claim~\eqref{eq:S3-path-est-2} follows.
\end{proof}
\endgroup

With these estimates at hand, we now turn to the proof of Theorem~\ref{th:main}(iv), that is, of the normal approximation result for the higher-order standard homogenization commutators.

\begin{proof}[Proof of Theorem~\ref{th:main}(iv)]
We split the proof into two steps.

\nopagebreak
\medskip
\step1 Proof that for all $8\le p,q<\infty$ and $1\le n\le\ell$,
\begin{multline*}
\expec{\Big\|D^2\int_{\R^d}g\cdot\Xi^{\circ,n}[\nabla\bar w]\Big\|_{\op}^4}^\frac14\\
\,\lesssim\,
\sum_{k=0}^{n}\sum_{s=0}^{n-1}p^{c_k+c_s}\Big(\big\|\mu_{d,k}[\nabla^kg\nabla^{s+1}\bar w]_\infty\big\|_{\Ld^p(\R^d)}+\sum_{l=0}^{k+s-n}\big\|\mu_{d,k}[\nabla^l(g\nabla^{k-l+s+1}\bar w)]_\infty\big\|_{\Ld^p(\R^d)}\Big)\\
+\sum_{k=0}^n\sum_{s=0}^{n-1}\sum_{m=0}^sp^{c_k+c_{s-m}}\big(p\log^2 p\big)^{\frac{m(m+1)}{2d}}
\bigg(\big\|\mu_{d,k}[\nabla^kg\nabla^{s+1}\bar w]_\infty\big\|_{\Ld^{\frac{dp}{2d+mp}}(\R^d)}\\
\hspace{6cm}+\sum_{l=0}^{k+s-n}\big\|\mu_{d,k}[\nabla^l(g\nabla^{k-l+s+1}\bar w)]_\infty\big\|_{\Ld^{\frac{dp}{2d+mp}}(\R^d)}\bigg)\\
+C_q\sum_{s=0}^{n-1}
\sum_{l=s+1}^{n}\big\|\mu_{d,n}\big[\nabla^l\big(g\nabla^{n-l+s+1}\bar w\big)\big]_\infty\big\|_{\Ld^\frac{dq}{2d+sq}(\R^d)}.
\end{multline*}
In order to not overestimate the logarithmic correction in the final result, we have to be specific here on the $p$-dependence; however, this is only necessary in the leading-order terms (that is, in the first two terms).
Let $1\le n\le\ell$.
Applying the representation formula of Proposition~\ref{prop:2nd-variation}, using~\eqref{eq:Da} and~\eqref{eq:D2a}, inserting the definition of $E_\e^n[\nabla\bar w]$ (cf.~Definition~\ref{def:EnFn}), and ordering terms by the number of Malliavin derivatives on $\Aa$, we find
\begin{align}\label{eq:decomp-D2X}
D^2\int_{\R^d}g\cdot\Xi^{\circ,n}[\nabla\bar w]\,=\,U_1+U_2+U_3,
\end{align}
in terms of
\begin{eqnarray*}
U_1(x,y)&:=&\delta(x-y)\,S^{1}(x),\\
U_2(x,y)&:=&\tilde U_2(x,y)+\tilde U_2(y,x),\\
\tilde U_2(x,y)&:=&\sum_{s=0}^{n-1}S^{2;s}_{j_1\ldots j_{s+1}}(x)\cdot D_y\big(\nabla\varphi^{s+1}_{j_1\ldots j_{s+1}}+\varphi^{s}_{j_1\ldots j_{s}}\ee_{j_{s+1}}\big)(x),\\
U_3(x,y)&:=&\sum_{s=0}^{n-1}\int_{\R^d}S^{3;s}_{j_1\ldots j_{s+1}}\cdot D_{xy}^2\big(\nabla\varphi^{s+1}_{j_1\ldots j_{s+1}}+\varphi^{s}_{j_1\ldots j_{s}}\ee_{j_{s+1}}\big),
\end{eqnarray*}
where we identify the operators $U_i$ with their kernels and we have set
\begingroup\allowdisplaybreaks
\begin{eqnarray*}
S^{1}&:=&\sum_{k=0}^{n}\sum_{s=0}^{n-1}\bigg(\big(\nabla^k_{i_1\ldots i_k}g_j\nabla_{j_1\ldots j_{s+1}}^{s+1}\bar w\big)\\
&&\hspace{3cm}-\sum_{l=0}^{k+s-n}\binom{k}{l}(-1)^{k-l}\,\nabla_{i_1\ldots i_l}^l\big(g_j\nabla^{k-l}_{i_{l+1}\ldots i_k}\nabla^{s+1}_{j_1\ldots j_{s+1}}\bar w\big)\bigg)\\
&&\hspace{-0.8cm}\times\sym_{i_1\ldots i_k}\big(\mathds1_{k\ne n}\nabla\varphi_{ji_1\ldots i_k}^{*,k+1}+\varphi_{ji_1\ldots i_{k-1}}^{*,k}\ee_{i_k}\big)
\cdot a_0''(G)\big(\nabla\varphi^{s+1}_{j_1\ldots j_{s+1}}+\varphi^{s}_{j_1\ldots j_{s}}\ee_{j_{s+1}}\big),\\
S^{2;s}_{j_1\ldots j_{s+1}}&:=&\sum_{k=0}^{n}\bigg(\big(\nabla^k_{i_1\ldots i_k}g_j\nabla_{j_1\ldots j_{s+1}}^{s+1}\bar w\big)\\
&&\hspace{2cm}-\sum_{l=0}^{k+s-n}\binom{k}{l}(-1)^{k+l}\nabla_{i_1\ldots i_l}^l\big(g_j\nabla^{k-l}_{i_{l+1}\ldots i_k}\nabla^{s+1}_{j_1\ldots j_{s+1}}\bar w\big)\bigg)\\
&&\hspace{1cm}\times\sym_{i_1\ldots i_k}\big(\mathds1_{k\ne n}\nabla\varphi_{ji_1\ldots i_k}^{*,k+1}+\varphi_{ji_1\ldots i_{k-1}}^{*,k}\ee_{i_k}\big) a_0'(G),\\
S^{3;s}_{j_1\ldots j_{s+1}}&:=&\sum_{l=s+1}^{n}\binom{n}{l}(-1)^{n+l}\nabla_{i_1\ldots i_l}^l\big(g_j\nabla^{n-l}_{i_{l+1}\ldots i_{n}}\nabla^{s+1}_{j_1\ldots j_{s+1}}\bar w\big)\\
&&\hspace{3cm}\times\sym_{i_1\ldots i_{n}}\big((\Aa^*\varphi_{ji_1\ldots i_{n-1}}^{*,n}-\sigma_{ji_1\ldots i_{n-1}}^{*,n})\ee_{i_{n}}\big),
\end{eqnarray*}
\endgroup
with the implicit understanding that the contribution of $\sum_{l=0}^{k+s-n}$ vanishes if \mbox{$k<n-s$}.
We analyze the three right-hand side terms in~\eqref{eq:decomp-D2X} separately and we start with $U_1$.
Decomposing the covariance function as $c=c_0\ast c_0$ so that $\|\zeta\|_\Hf=\|c_0\ast\zeta\|_{\Ld^2(\R^d)}$, the definition~\eqref{eq:def-op} of the norm $\|\cdot\|_\op$ is equivalent to
\begin{equation}\label{eq:rewrite-norm}
\|U_1\|_{\op}\,=\,\sup_{\|\zeta\|_{\Ld^2(\R^d)}=\|\zeta'\|_{\Ld^2(\R^d)}=1}\Big|\int_{\R^d\times\R^d}(c_0\ast\zeta)(x)\,U_1(x,y)\,(c_0\ast\zeta')(y)\,dxdy\Big|.
\end{equation}
For $p\ge4$, inserting the special form of $U_1(x,y)=\delta(x-y)\,S^1(x)$ and noting that the discrete $\ell^{\frac{2p}{p-2}}$--$\ell^2$ inequality and the assumption~\eqref{eq:cov-L1} imply
\begin{equation}\label{eq:rewrite-norm-seq}
\|[c_0\ast\zeta]_\infty\|_{\Ld^{\frac{2p}{p-2}}(\R^d)}\,\lesssim\,\|[c_0\ast\zeta]_\infty\|_{\Ld^2(\R^d)}\,\le\,\|\zeta\|_{\Ld^2(\R^d)},
\end{equation}
we get
\[\|U_1\|_\op\,\lesssim\,\|[S^1]_1\|_{\Ld^p(\R^d)},\]
hence,
\[\expec{\|U_1\|_\op^4}^\frac14\,\le\,\expec{\|U_1\|_\op^p}^\frac1p\,\lesssim\,\|[S^1]_1\|_{\Ld^p(\R^d;\Ld^p(\Omega))}.\]
By definition of $S^1$, appealing to the corrector estimates of Proposition~\ref{prop:cor}, we deduce
\begin{multline*}
\expec{\|U_1\|_{\op}^4}^\frac14\,\lesssim\,\sum_{k=0}^{n}\sum_{s=0}^{n-1}p^{c_k+c_s}\Big(\big\|\mu_{d,k}[\nabla^kg\nabla^{s+1}\bar w]_\infty\big\|_{\Ld^p(\R^d)}\\
+\sum_{l=0}^{k+s-n}\big\|\mu_{d,k}[\nabla^l(g\nabla^{k-l+s+1}\bar w)]_\infty\big\|_{\Ld^p(\R^d)}\Big).
\end{multline*}
We turn to $U_2$. By symmetry, it suffices to estimate the norm of $\tilde U^2$. From~\eqref{eq:rewrite-norm} and $\|[c_0\ast\zeta]_\infty\|_{\Ld^2(\R^d)}\le\|\zeta\|_{\Ld^2(\R^d)}$, we see
\[\|\tilde U_2\|_{\op}^2\,\lesssim\,\sup_{\|[\zeta]_\infty\|_{\Ld^2(\R^d)}=1}\int_{\R^d}\Big[\int_{\R^d}\tilde U_2(x,\cdot)\zeta(x)\,dx\Big]_1^2.\]
Taking the expectation and arguing by duality, we find
\begin{equation}\label{eq:op-bound-D2X}
\expec{\|\tilde U_2\|_{\op}^4}^\frac14
\,\lesssim\,\sup_{\|[\zeta]_\infty\|_{\Ld^4(\Omega;\Ld^2(\R^d))}=1}M(\zeta),
\end{equation}
where
\[M(\zeta)\,:=\,\E\bigg[{\int_{\R^d}\Big[\int_{\R^d}\tilde U_2(x,\cdot) \zeta(x)dx\Big]_1^2}\bigg]^\frac12.\]
Let $\zeta\in C^\infty_c(\R^d;\Ld^\infty(\Omega))$ be fixed with $\|[\zeta]_\infty\|_{\Ld^4(\Omega;\Ld^2(\R^d))}=1$.
Noting that Jensen's inequality yields
\[\|[\zeta]_\infty\|_{\Ld^2(\R^d;\Ld^2(\Omega))}\,=\,\|[\zeta]_\infty\|_{\Ld^2(\Omega;\Ld^2(\R^d))}\,\le\,\|[\zeta]_\infty\|_{\Ld^4(\Omega;\Ld^2(\R^d))},\]
and that the discrete $\ell^4$--$\ell^2$ inequality gives
\[\|[\zeta]_\infty\|_{\Ld^4(\R^d;\Ld^4(\Omega))}\,\lesssim\,\|[\zeta]_\infty\|_{\Ld^4(\Omega;\Ld^2(\R^d))},\]
we deduce by interpolation, for all $2\le r\le4$,
\begin{align}\label{eq:xi-LrLr}
\|[\zeta]_\infty\|_{\Ld^r(\R^d;\Ld^r(\Omega))}\,\lesssim\,\|[\zeta]_\infty\|_{\Ld^4(\Omega;\Ld^2(\R^d))}=1.
\end{align}
By definition of $\tilde U_2$ and by Lemma~\ref{lem:key-estimates-decomp} in the form~\eqref{eq:decomp-Js}, we get
\[M(\zeta)\,\lesssim\,\sum_{s=0}^{n-1}\sum_{m=0}^s\sum_{\beta\in\{0,1\}}\big\|\big[U_\cdot^m T^\beta\big((\Aa^*)^{-1}\zeta S^{2;s}\big)\big]_2\big([\nabla\varphi^{s+1-m}]_2+[\varphi^{s-m}]_2\big)\big\|_{\Ld^2(\R^d;\Ld^2(\Omega))},\]
and the corrector estimates of Proposition~\ref{prop:cor} yield for all $p\ge1$,
\begin{equation}\label{eq:start-est-Mxi}
M(\zeta)
\,\lesssim\,\sum_{s=0}^{n-1}\sum_{m=0}^sp^{c_{s-m}}\sum_{\beta\in\{0,1\}}\big\|\big[U_\cdot^m T^\beta\big((\Aa^*)^{-1}\zeta S^{2;s}\big)\big]_2\big\|_{\Ld^2(\R^d;\Ld^{\frac{2p}{p-1}}(\Omega))}.
\end{equation}
For $p\ge2$, applying $m$ times the bound of Corollary~\ref{cor:UT} and once the one of Theorem~\ref{th:CZ-ann} with $\delta=\frac1{m+1}\frac{2p}{(p-1)(p-2)}\sim\frac1p$ in order to pass from the stochastic integrability $\frac{2p}{p-1}$ to $\frac{2p}{p-2}$, we obtain the following,
\begin{align*}
M(\zeta)
\,\lesssim\,\sum_{s=0}^{n-1}\sum_{m=0}^sp^{c_{s-m}}\big(p\log^2 p\big)^{\frac{m(m+1)}{2d}}\|[\zeta S^{2;s}]_2\|_{\Ld^{\frac{2d}{d+2m}}(\R^d;\Ld^{\frac{2p}{p-2}}(\Omega))},
\end{align*}
where the exponent $\frac{m(m+1)}d$ of $p\log^2p$ comes from the large prefactors in Corollary~\ref{cor:UT} and Theorem~\ref{th:CZ-ann} and is due to the successive deviations of spatial integrability from~$2$.
Hölder's inequality together with~\eqref{eq:xi-LrLr} yields for $p\ge8$,
\begin{multline*}
\|[\zeta S^{2;s}]_2\|_{\Ld^{\frac{2d}{d+2m}}(\R^d;\Ld^{\frac{2p}{p-2}}(\Omega))}
\,\le\,\|[\zeta]_\infty\|_{\Ld^\frac{2p}{p-4}(\R^d;\Ld^\frac{2p}{p-4}(\Omega))}\|[S^{2;s}]_2\|_{\Ld^\frac{dp}{2d+mp}(\R^d;\Ld^p(\Omega))}\\
\,\lesssim\,\|[S^{2;s}]_2\|_{\Ld^\frac{dp}{2d+mp}(\R^d;\Ld^p(\Omega))}.
\end{multline*}
By definition of $S^{2;s}$,
using the corrector estimates of Proposition~\ref{prop:cor}, we conclude
\begin{multline}\label{eq:proof-bound-U2}
\expec{\|\tilde U_2\|_{\op}^4}^\frac14
\,\lesssim\,\sum_{s=0}^{n-1}\sum_{m=0}^sp^{c_{s-m}}\big(p\log^2 p\big)^{\frac{m(m+1)}{2d}}
\sum_{k=0}^np^{c_k}\bigg(\|\mu_{d,k}[\nabla^kg\nabla^{s+1}\bar w]_\infty\|_{\Ld^{\frac{dp}{2d+mp}}(\R^d)}\\
+\sum_{l=0}^{k+s-n}\|\mu_{d,k}[\nabla^l(g\nabla^{k-l+s+1}\bar w)]_\infty\|_{\Ld^{\frac{dp}{2d+mp}}(\R^d)}\bigg).
\end{multline}
It remains to analyze $U_3$.
For $p\ge2$, we first appeal to the definition of $U_3$ and to Lemma~\ref{lem:key-estimates-2nd} (with $p$ replaced by $\frac p2$) in the form
\begin{multline*}
\Big|\int_{\R^d\times\R^d}\zeta(x)\,U_3(x,y)\,\zeta(y)\,dxdy\Big|\\
\,\lesssim\|[\zeta]_\infty\|_{\Ld^{\frac{2p}{p-2}}(\R^d)}^2\sum_{s=0}^{n-1}\sum_{k=0}^s\sum_{\beta\in\{0,1\}}\big\|\big[U_\cdot^kT^\beta\big((\Aa^*)^{-1}S^{3;s}\big)\big]_2\big([\nabla\varphi^{s+1-k}]_2+[\varphi^{s-k}]_2\big)\big\|_{\Ld^\frac p2(\R^d)}\\
+\|[\zeta]_\infty\|_{\Ld^2(\R^d)}\sum_{s=0}^{n-1}\sum_{k=0}^s\sum_{\beta\in\{0,1\}}J_{s-k}\big(\zeta(\Aa^*)^{-1}a_0'(G)^* U_\cdot^k T^\beta \big((\Aa^*)^{-1}S^{3;s}\big)\big).
\end{multline*}
From~\eqref{eq:rewrite-norm} and~\eqref{eq:rewrite-norm-seq}, we deduce
\begin{multline*}
\|U_3\|_{\op}
\,\lesssim\,\sum_{s=0}^{n-1}\sum_{k=0}^{s}\sum_{\beta\in\{0,1\}}\big\|\big[U_\cdot^kT^\beta\big((\Aa^*)^{-1}S^{3;s}\big)\big]_2\big([\nabla\varphi^{s+1-k}]_2+[\varphi^{s-k}]_2\big)\big\|_{\Ld^\frac p2(\R^d)}\\
+\sum_{s=0}^{n-1}\sum_{k=0}^{s}\sum_{\beta\in\{0,1\}}\sup_{\|[\zeta]_\infty\|_{\Ld^2(\R^d)}=1}J_{s-k}\big(\zeta(\Aa^*)^{-1}a_0'(G)^* U_\cdot^k T^\beta \big((\Aa^*)^{-1}S^{3;s}\big)\big).
\end{multline*}
Taking the $\Ld^4(\Omega)$ norm and using the corrector estimates of Proposition~\ref{prop:cor} in the first right-hand side term, and arguing by duality as in~\eqref{eq:op-bound-D2X} and repeating the proof of~\eqref{eq:proof-bound-U2} for the second right-hand side term, we obtain for all $8\le p<\infty$,
\begin{multline*}
\expec{\|U_3\|_{\op}^4}^\frac14
\,\lesssim_p\,\sum_{s=0}^{n-1}\sum_{k=0}^{s}\sum_{\beta\in\{0,1\}}\big\|\big[U_\cdot^kT^\beta \big((\Aa^*)^{-1}S^{3;s}\big)\big]_2\big\|_{\Ld^\frac p2(\R^d;\Ld^{p}(\Omega))}\\
+\sum_{s=0}^{n-1}\sum_{k=0}^{s}\sum_{\beta\in\{0,1\}}\sum_{m=0}^{s-k}\big\|\big[U_\cdot^kT^\beta\big((\Aa^*)^{-1}S^{3;s}\big)\big]_2\big\|_{\Ld^\frac{dp}{2d+mp}(\R^d;\Ld^p(\Omega))}.
\end{multline*}
A multiple use of the annealed bounds of Theorem~\ref{th:CZ-ann} and Corollary~\ref{cor:UT} then leads to
\[\expec{\|U_3\|_{\op}^4}^\frac14
\lesssim_p\sum_{s=0}^{n-1}\sum_{k=0}^{s}
\sum_{m=0}^{s-k}\|[S^{3;s}]_2\|_{\Ld^\frac{dp}{2d+(k+m)p}(\R^d;\Ld^{2p}(\Omega))}
\!\!\lesssim_p\sum_{s=0}^{n-1}\|[S^{3;s}]_2\|_{\Ld^\frac{dp}{2d+sp}(\R^d;\Ld^{2p}(\Omega))}.\]
By definition of $S^{3;s}$, using the corrector estimates of Proposition~\ref{prop:cor}, we conclude
\begin{align*}
\expec{\|U_3\|_{\op}^4}^\frac14
\,\lesssim_p\,\sum_{s=0}^{n-1}
\sum_{l=s+1}^{n}\big\|\mu_{d,n}\big[\nabla^l\big(g\nabla^{n-l+s+1}\bar w\big)\big]_\infty\big\|_{\Ld^\frac{dp}{2d+sp}(\R^d)}.
\end{align*}
Gathering the above estimates and replacing $p$ by $q$ in the estimate for $U^3$, the claim follows.

\medskip
\step2 Conclusion.\\
By scaling, the conclusion of Step~1 yields for all $8\le p,q<\infty$ and $1\le n\le\ell$,
\begin{multline*}
\e^{-\frac d2}\,\expec{\Big\|D^2\int_{\R^d}g\cdot\Xi_\e^{\circ,n}[\nabla\bar w]\Big\|_{\op}^4}^\frac14
\,\lesssim_{g,\bar w}\,
\e^{\frac d2-\frac dp}\sum_{k=0}^{n}\sum_{s=0}^{n-1}\mu_{d,k}(\tfrac1\e)\e^{k+s}p^{c_k+c_s}\\
+\e^{\frac d2-\frac{2d}p}\sum_{k=0}^n\sum_{s=0}^{n-1}\sum_{m=0}^s\mu_{d,k}(\tfrac1\e)\e^{k+s-m}p^{c_k+c_{s-m}}\big(p\log^2 p\big)^{\frac{m(m+1)}{2d}}
+C_q\e^{\frac d2-\frac{2d}q}\mu_{d,n}(\tfrac1\e)\e^{n},
\end{multline*}
hence, choosing $p=\Log$ and $q=8\vee4d$, and recalling that $c_0=\frac12$,
\begin{align*}
\e^{-\frac d2}\,\expec{\Big\|D^2\int_{\R^d}g\cdot\Xi_\e^{\circ,n}[\nabla\bar w]\Big\|_{\op}^4}^\frac14
\,\lesssim_{f,g}\,\e^{\frac d2}\Log\big(\Log\log^2\Log\big)^{\frac1{2d}n(n-1)}.
\end{align*}
In terms of $X_\e^n:=\e^{-\frac d2}\int_{\R^d}g\cdot\Xi_\e^{\circ,n}[\nabla\bar w]$, we now apply Proposition~\ref{prop:Mall}(iv) in the form
\begin{align*}
\quad\operatorname{d}_\Nc(X_\e^n)
\,\lesssim\,\frac{\expec{\|D^2X_\e^n\|_{\op}^4}^\frac14\expec{\|DX_\e^n\|_\Hf^4}^\frac14}{\var{X_\e^n}}
\end{align*}
Inserting the above estimate on $\|D^2X_\e^n\|_{\op}$ and noting that the proof of Theorem~\ref{th:main}(iii) yields
$\expec{\|DX_\e^n\|_\Hf^4}\,\lesssim_{f,g}\,1$,
the conclusion follows.
\end{proof}

\appendix
\section{More Malliavin calculus}\label{app:Mall}

In this appendix, we include for completeness a short, essentially self-contained proof of Proposition~\ref{prop:Mall}.

\begin{proof}[Proof of Proposition~\ref{prop:Mall}]
We split the proof into four steps.

\medskip
\step1 Proof of~(i).\\
In terms of the Ornstein-Uhlenbeck semigroup $e^{-t\Lc}$ (e.g.~\cite[Section~1.4]{Nualart}), we may write
\[\var{X}=-\int_0^\infty\partial_t\expec{(e^{-t\Lc}X)^2}\,dt\]
where Mehler's formula (e.g.~\cite[(1.67)]{Nualart}) indeed ensures that $\expec{(e^{-t\Lc}X)^2}\to\expec{X}^2$ as $t\uparrow\infty$. Computing the derivative in $t$ yields
\[\var{X}\,=\,2\int_0^\infty\expec{(e^{-t\Lc}X)\Lc(e^{-t\Lc}X)}\,dt\,=\,2\int_0^\infty\expec{\|De^{-t\Lc}X\|_\Hf^2}\,dt,\]
hence, appealing to the commutation relation~\eqref{eq:commut-Mall} in the form $De^{-t\Lc}=e^{-t}e^{-t\Lc}D$,
\[\var{X}\,=\,2\int_0^\infty e^{-2t}\,\expec{\|e^{-t\Lc}DX\|_\Hf^2}\,dt,\]
and the positivity of $\Lc$ leads to the conclusion,
\[\var{X}\,\le\,2\int_0^\infty e^{-2t}\,\expec{\|DX\|_\Hf^2}\,dt\,=\,\expec{\|DX\|_\Hf^2}.\]

\medskip
\step2 Proof of~(ii).\\
Let $X,Y\in\Dm^{1,2}$ with $\expec{X}=\expec{Y}=0$.
First note that the Poincaré inequality~(i) implies that the restriction of $\Lc$ to $\Ld^2(\Omega)/\C:=\{U\in\Ld^2(\Omega):\expec{U}=0\}$ is invertible.
In particular, there exists $Z\in\Ld^2(\Omega)$ such that $Y=\Lc Z$.
By definition of the adjoint $D^*$ (cf.~\eqref{eq:def-Dst}), we then find
\begin{align}\label{eq:pre-formHS}
\expec{XY}=\expec{X\Lc Z}=\expec{XD^* D Z}=\expec{\langle DX,DZ\rangle_{\Hf}}.
\end{align}
Appealing to the commutator relation~\eqref{eq:commut-Mall}
in the form $DY=D\Lc Z=(1+\Lc)DZ$,
we conclude
\[\expec{XY}=\expec{\langle DX,(1+\Lc)^{-1}DY\rangle_{\Hf}}.\]

\medskip
\step3 Proof of~(iii).
\nopagebreak

\noindent
Similarly as in Step~1, in terms of the Ornstein-Uhlenbeck semigroup $e^{-t\Lc}$, we may write
\begin{multline*}
\ent{X^2}\,=\,-\int_0^\infty\partial_t\,\expec{(e^{-t\Lc}X^2)\log(e^{-t\Lc}X^2)}\,dt\\
\,=\,\int_0^\infty\expec{(\Lc e^{-t\Lc}X^2)\log(e^{-t\Lc}X^2)}\,dt
\,=\,\int_0^\infty\expec{\frac{\|De^{-t\Lc}X^2\|_{\Hf}^2}{e^{-t\Lc}X^2}}\,dt,
\end{multline*}
hence, by the commutation relation~\eqref{eq:commut-Mall} in the form $De^{-t\Lc}=e^{-t}e^{-t\Lc}D$,
\begin{eqnarray*}
\ent{X^2}&=&\int_0^\infty e^{-2t}\,\expec{\frac{\|e^{-t\Lc}DX^2\|_{\Hf}^2}{e^{-t\Lc}X^2}}\,dt.
\end{eqnarray*}
Appealing to Mehler's formula for the Ornstein-Uhlenbeck semigroup $e^{-t\Lc}$ (e.g.~\cite[(1.67)]{Nualart}), the Cauchy-Schwarz inequality leads to
\[\|e^{-t\Lc}DX^2\|_{\Hf}^2\,=\,4\,\|e^{-t\Lc}(XDX)\|_{\Hf}^2\,\le\, 4\,\big(e^{-t\Lc}X^2\big)\big(e^{-t\Lc}\|DX\|_{\Hf}^2\big),\]
so that the above becomes
\begin{eqnarray*}
\ent{X^2}\,\le\,4\int_0^\infty e^{-2t}\,\expec{e^{-t\Lc}\|DX\|_{\Hf}^2}\,dt
\,\le\,2\,\expec{\|DX\|_\Hf^2},
\end{eqnarray*}
which proves the logarithmic Sobolev inequality.
Next, higher integrability is deduced by integration as e.g.\@ in~\cite[Theorem~3.4]{Aida-Stroock-94}: we write
\[\expec{|X|^{2p}}^\frac1p-\expec{X^2}=\int_1^p\frac1{q^2}\expec{|X|^{2q}}^{\frac1q-1}\ent{|X|^{2q}}dq,\]
so that the logarithmic Sobolev inequality implies
\begin{eqnarray*}
\expec{|X|^{2p}}^\frac1p-\expec{X^2}&\le&2\int_1^p\frac1{q^2}\expec{|X|^{2q}}^{\frac1q-1}\expec{\|D|X|^q\|_\Hf^2}dq\\
&=&2\int_1^p\expec{|X|^{2q}}^{\frac1q-1}\expec{|X|^{2(q-1)}\|DX\|_\Hf^2}dq\\
&\le&2\int_1^p\expec{\|DX\|_\Hf^{2q}}^\frac1qdq~\le~2p\,\expec{\|DX\|_\Hf^{2p}}^\frac1p,
\end{eqnarray*}
and the conclusion follows from the Poincaré inequality~(i).

\medskip
\step4 Proof of~(iv).
\nopagebreak

\noindent
Let
$X\in\Ld^2(\Omega)$ with $\expec{X}=0$ and $\var X=1$.      
For $h\in \Ld^\infty(\R)$, we define its Stein transform $S_h$ as the solution of Stein's equation
\[S_h'(x)-xS_h(x)=h(x)-\expec{h(\Nc)}.\]
As in Step~2, there exists $Z\in\Ld^2(\Omega)$ such that $Y=\Lc Z$.
We then compute
\[\expec{h(X)}-\expec{h(\Nc)}=\expec{S_h'(X)-XS_h(X)}=\expec{S_h'(X)-(D^*DZ)S_h(X)},\]
and hence, integrating by parts and using~\eqref{eq:pre-formHS} in the form $\expec{\langle DZ,DX\rangle_\Hf}=\var X=1$,
\begin{eqnarray}\label{eq:IPP-Stein}
\big|\expec{h(X)}-\expec{h(\Nc)}\!\big|&=&\big|\expec{S_h'(X)\big(1-\langle DZ,DX\rangle_\Hf\big)}\!\big|\nonumber\\
&\le&\|S_h'\|_{\Ld^\infty}\var{\langle DZ,DX\rangle_\Hf}^\frac12.
\end{eqnarray}
Noting that $\|S_h'\|_{\Ld^\infty}\le2\|h\|_{\Ld^\infty}$ (e.g.~\cite[Theorem~3.3.1]{NP-book}) and taking the supremum over all $h\in\Ld^\infty(\R)$, we deduce
\begin{eqnarray*}
\dTV{X}{\Nc}\,\le\,2\,\var{\langle DZ,DX\rangle_\Hf}^\frac12.
\end{eqnarray*}
Noting that for $h$ Lipschitz-continuous there holds $\|S_h'\|_{\Ld^\infty}\le\sqrt{\frac2\pi}\|h'\|_{\Ld^\infty}$ (e.g.~\cite[Proposition~3.5.1]{NP-book}), we can deduce a similar bound on the $1$-Wasserstein distance.
The corresponding bound on the $2$-Wasserstein distance takes the form
\[W_2(X;\Nc)\le\expec{|\langle DZ,DX\rangle_\Hf-1|^2}^\frac12=\var{\langle DZ,DX\rangle_\Hf}^\frac12;\]
its proof is of a different nature and is based on an optimal transport argument in density space (cf.~\cite[Proposition~3.1]{LNP-15}).

\medskip\noindent
It remains to estimate the variance $\var{\langle DZ,DX\rangle_\Hf}$.
For that purpose, we apply the first-order Poincaré inequality~(i) in the form
\begin{eqnarray*}
\var{\langle DZ,DX\rangle_\Hf}&\le&\expec{\|\langle D^2Z,DX\rangle_\Hf+\langle DZ,D^2X\rangle_\Hf\|_{\Hf}^2}.
\end{eqnarray*}
Noting that the commutation relation~\eqref{eq:commut-Mall} leads to $DX=(1+\Lc)DZ$ and $D^2X=(2+\Lc)D^2Z$, we deduce
\begin{eqnarray*}
\var{\langle DZ,DX\rangle_\Hf}&\le&\expec{\big\|\big\langle D^2X,\big((1+\Lc)^{-1}+(2+\Lc)^{-1}\big)DX\big\rangle_\Hf\big\|_{\Hf}^2}\\
&\le&\expec{\|D^2X\|_\op^4}^\frac14\expec{\big\|\big((1+\Lc)^{-1}+(2+\Lc)^{-1}\big)DX\big\|_\Hf^4}^\frac14,
\end{eqnarray*}
Noting as in~\cite[Proposition~3.2]{MO} that $(1+\Lc)^{-1}$ and $(2+\Lc)^{-1}$ have operator norms on~$\Ld^4(\Omega)$ bounded by $1$ and $\frac12$, respectively, the conclusion follows.
\end{proof}

\section*{Acknowledgements}
\noindent
The work of MD was supported by F.R.S.-FNRS and by the CNRS-Momentum program.

\bibliographystyle{plain}
\bibliography{biblio}

\end{document}